\documentclass[twoside,a4paper]{amsart}
\usepackage[english]{babel}
\usepackage[T1]{fontenc}
\usepackage{xcolor}
\usepackage{graphicx}
\usepackage{float}
\usepackage{amsfonts,amsmath,amsthm,amssymb,latexsym}
\usepackage{vmargin}
\usepackage{amscd}
\usepackage[all]{xy}
\usepackage[final]{showlabels} 
\usepackage{enumerate}
\usepackage{version}
\usepackage[normalem]{ulem}

\newtheorem{thm}{Theorem}[section]
\newtheorem{cor}[thm]{Corollary}
\newtheorem{lem}[thm]{Lemma}
\newtheorem{prop}[thm]{Proposition}
\theoremstyle{definition}
\newtheorem{defn}[thm]{Definition}
\newtheorem{es}[thm]{Example}

\newtheorem{rmk}[thm]{Remark}

\newcommand{\Id}{\mathrm{Id}}

\newcommand{\Hom}{\mathrm{Hom}}

\newcommand{\cc}{\mathcal{C}}
\newcommand{\dd}{\mathcal{D}}
\newcommand{\e}{\mathcal{E}}
\newcommand{\f}{\mathcal{F}}
\newcommand{\g}{\mathcal{G}}

\newcommand{\p}{\mathcal{P}}

\newcommand{\id}{\mathrm{Id}}

\newcommand{\Cat}{{\sf Cat}}

\newcommand{\Set}{{\sf Set}}

\newenvironment{invisible}{{\noindent\sc \colorbox{yellow}{Invisible:}\;}\color{gray}}{\medskip}
\excludeversion{invisible}

\begin{document}
\title{On (naturally) semifull and (semi)separable semifunctors}

\author{Lucrezia Bottegoni}
\address{%
\parbox[b]{\linewidth}{University of Turin, Department of Mathematics ``G. Peano'', via
Carlo Alberto 10, I-10123 Torino, Italy}}
\email{lucrezia.bottegoni@unito.it}

\subjclass[2020]{Primary 18A22; Secondary 18A20, 18A40} 

\begin{abstract}
The notion of semifunctor between categories, due to S. Hayashi (1985), is defined as a functor that does not necessarily preserve identities. In this paper we study how several properties of functors, such as fullness, full faithfulness, separability, natural fullness, can be formulated for semifunctors. Since a full semifunctor is actually a functor, we are led to introduce a notion of semifullness (and then semifull faithfulness) for semifunctors. In order to show that these conditions can be derived from requirements on the hom-set components associated with a semifunctor, we look at ``semisplitting properties'' for seminatural transformations and we investigate the corresponding properties for morphisms whose source or target is the image of a semifunctor.  We define the notion of naturally semifull semifunctor and we characterize natural semifullness for semifunctors that are part of a semiadjunction in terms of semisplitting conditions for the unit and counit attached to the semiadjunction. We study the behavior of semifunctors with respect to (semi)separability and we prove Rafael-type Theorems for (semi)separable semifunctors and a Maschke-type Theorem for separable semifunctors. We provide examples of semifunctors on which we test the properties considered so far.
\end{abstract}
\keywords{Semifunctor, Semiadjunction, Natural semi-isomorphism, Separable functor, (Naturally) full functor, Idempotent completion}
\maketitle
\tableofcontents

\section*{Introduction}
The notion of \emph{semifunctor} between categories, that appeared in \cite{EZ76} under the name of \emph{weak functor}, was investigated by S. Hayashi in \cite{Ha85}, in order to develop a categorical semantics for non-extensional typed lambda calculus. A semifunctor $F:\cc\to\dd$ is defined as a functor, but it is not required to preserve identities. In the same paper Hayashi introduced the notion of \emph{semiadjunction} between semifunctors. Afterwards, in \cite{Ho93} R. Hoofman defined an appropriate notion of morphism between semifunctors, calling it \emph{seminatural transformation}. Explicitly, a natural transformation $\alpha: F\to F^{\prime }$ between semifunctors $F,F^\prime:\cc\to\dd$ is defined as a natural transformation between functors; if in addition $\alpha _{X}\circ F \mathrm{Id}_{X}=\alpha _{X}$ holds true for every object $X$ in $\cc$, then $\alpha$ is a seminatural transformation. Moreover, if there exists a natural transformation $\beta: F^\prime\to F$ such that  $\beta\circ F^\prime\id =\beta$, $\alpha\circ\beta = F^\prime\id$ and $\beta\circ\alpha = F\id$, then $\alpha$ is said to be a \emph{natural semi-isomorphism}.\par 
In this paper we study how some notable properties of functors (e.g. fullness, full faithfulness, separability, natural fullness) can be formulated for semifunctors. Given a semifunctor $F: \cc \rightarrow \dd$, we consider the associated natural transformation $\f : \mathrm{Hom}_{\cc}(-,-)\rightarrow \Hom_{\dd}(F-, F-)$, defined by $\f_{X,Y}(f)= F(f)$, for any morphism $f:X\rightarrow Y$ in $\cc$. By requiring that $\f_{X,Y}$ is injective (resp. surjective, bijective) for every pair of objects $X,Y\in\cc$, then $F$ is a \emph{faithful} (resp. \emph{full}, \emph{fully faithful}) semifunctor. 
Since a full semifunctor is actually a functor, motivated by the behavior of an endosemifunctor on $\Set$ (Example \ref{es:Set-semifull-faithful}), we introduce a weaker notion of fullness for semifunctors that we call \emph{semifullness}. We say that a semifunctor $F:\cc\to\dd$ is \emph{semifull} if for any morphism $f:FX\to FY$ in $\dd$ there exists a morphism $g:X\to Y$ in $\cc$ such that $F(g)=F\id_Y\circ f\circ F\id_X$. We define a semifunctor to be \emph{semifully faithful} if it is faithful and semifull. In order to show that the semifull and semifully faithful conditions can be derived from requirements on the natural transformation $\f$ associated with a semifunctor, we look at particular ``semisplitting'' properties for seminatural transformations. We call a seminatural transformation $\alpha:F\to F'$ \emph{natural semisplit-mono} (resp. \emph{natural semisplit-epi}) if there exists a seminatural transformation $\beta:F'\to F$ such that $\beta\circ\alpha= F\id$ (resp. $\alpha\circ\beta= F'\id$). We investigate the corresponding semisplitting properties for morphisms whose source or target is the image of a semifunctor. Explicitly, given semifunctors $F:\cc\to\dd$, $F':\cc'\to\dd$, we say that\vskip0.1cm 
\noindent $\bullet$ a morphism $f:FC\to D$ in $\dd$ is an \emph{$F_C$-semisplit-mono} if there exists a morphism $g:D\to FC$ in $\dd$ such that $g\circ f=F\id_C$;\\
$\bullet$ a morphism $f:FC\to F'C'$ in $\dd$ is an \emph{$(F_C,F'_{C'})$-semisplit-mono} if $f\circ F\id_C=f$ and there exists a morphism $g:F'C'\to FC$ in $\dd$ such that $g\circ f=F\id_C$ and $g\circ F'\id_{C'}=g$.
\vskip0.1cm
\noindent Analogously, one can introduce the notions of \emph{$F_C$-semisplit-epi} and \emph{$(F_C,F'_{C'})$-semisplit-epi}. We call a morphism $f:FC\to F'C'$ in $\dd$ an \emph{$(F_C,F'_{C'})$-semi-isomorphism} if $f\circ F\id_C=f$ and there exists a morphism $g: F'C'\to FC$ in $\dd$ such that $g\circ f=F\id_C$ and $f\circ g=F'\id_{C'}$. In Proposition \ref{prop:semiiso-semisplit} we prove that a morphism is an $(F_C,F'_{C'})$-semi-isomorphism if and only if it is both an $(F_C,F'_{C'})$-semisplit-mono and an $(F_C,F'_{C'})$-semisplit-epi.  When the semifunctors $F, F'$ are understood, we usually omit them in the notation. Given a seminatural transformation $\alpha:F\to F'$ of semifunctors, if $\alpha$ is a natural semisplit-mono (resp. natural semisplit-epi, natural semi-isomorphism), then every component morphism $\alpha_C:FC\to F'C$ is an $(F_C,F'_C)$-semisplit-mono (resp. $(F_C,F'_C)$-semisplit-epi, $(F_C,F'_C)$-semi-isomorphism) in $\dd$.
\par
In Proposition \ref{prop:trasfnat-semifull} we show that, if for every $X, Y\in\cc$, $\f_{X,Y}$ is a $\Hom_{\dd}(F-, F-)_{(X,Y)}$-semisplit-epi (resp. $((X,Y),(X,Y))$-semi-isomorphism), then $F$ is semifull (resp. semifully faithful). Further, semifully faithful semifunctors reflect $(F_C,F'_{C'})$-semi-isomorphisms (Proposition \ref{prop:refl-semiiso}). In Proposition \ref{prop:char-faithful-semifull} we provide a characterization of faithfulness and semifullness for semifunctors that are part of a semiadjunction. Recall that a \emph{semiadjunction} $F\dashv_\mathrm{s} G$ between semifunctors is the datum of semifunctors $F:\cc\rightarrow \dd$ and $G:\dd\rightarrow \cc$ equipped with (semi)natural transformations $\eta :\mathrm{Id}_{\cc}\rightarrow GF$ (unit) and $\epsilon :FG\rightarrow \mathrm{Id}_{\dd}$ (counit) such that $G\epsilon \circ \eta G=G\mathrm{Id}$ and $\epsilon F\circ F\eta =F\mathrm{Id}$ hold. Explicitly, given a semiadjunction $F\dashv_\mathrm{s} G:\dd\to\cc$ with unit $\eta$ and counit $\epsilon$, we prove that $F$ (resp. $G$) is semifull if and only if $\eta_C$ is a $C$-semisplit-epi in $\cc$ for every $C\in \cc$ (resp. $\epsilon_D$ is a $D$-semisplit-mono in $\dd$ for every $D\in \dd$).\par 
Then, we explore the behavior of semifunctors with respect to separability. Separable functors were introduced in \cite{NVV89} in order to interpret the notion of separable ring extension from a categorical point of view. Several results and applications of separable functors are illustrated e.g. in \cite{CMZ02}. A functor $F:\mathcal{C}\to\mathcal{D}$ is said to be \emph{separable} if the associated natural transformation $\mathcal{F}$ has a left inverse, i.e. there is a natural transformation $\mathcal{P} : \mathrm{Hom}_{\dd}(F-, F-)\rightarrow \mathrm{Hom}_{\cc}(-,-)$ such that $\mathcal{P}_{X,Y}\circ\mathcal{F}_{X,Y} = \mathrm{Id}_{\mathrm{Hom}_{\cc}(X,Y)}$ for all $X$ and $Y$ in $\cc$. We define a semifunctor $F:\cc\to\dd$ to be separable by requiring the same condition on the associated natural transformation $\f$, and we discuss general properties. The first difference with the functorial case is in the so-called Maschke Theorem \cite[Proposition 1.2]{NVV89}, see also \cite[Corollary 5]{CMZ02}, which states that, given a separable functor $F:\cc\to\dd$ and a morphism $f:C\to C'$ in $\cc$, if $F(f)$ is a split-mono (resp. split-epi) in $\dd$, then so is $f$. In Theorem \ref{thm:maschke} we show that if $F$ is a separable semifunctor and $F(f):FC\to FC'$ is an $F_C$-semisplit-mono (resp. $F_{C'}$-semisplit-epi), then $f$ is a split-mono (resp. split-epi), obtaining a Maschke-type Theorem for  semifunctors. A key result for separable functors is given by Rafael Theorem \cite{Raf90} which characterizes the separability of functors that have an adjoint in terms of splitting properties for the unit and the counit. In Theorem \ref{Th.Rafael} we prove an analogue Rafael-type Theorem for separable semifunctors.\par 
As separable functors are naturally faithful, in a somehow dual way \emph{naturally full} functors have been introduced in \cite{AMCM06} by requiring that the natural transformation $\f$ associated with a functor has a right inverse. It is clear that the natural fullness condition (which implies fullness) on a semifunctor reduces to the naturally full functor case, as a full semifunctor is actually a functor. Thus, we investigate whether a notion of natural semifullness is possible for semifunctors. We call a semifunctor $F:\cc\to\dd$ \emph{naturally semifull} if there is a natural transformation $\p : \Hom_{\dd}(F-, F-)\rightarrow \Hom_{\cc}(-,-)$ such that $(\f_{X,Y}\circ\p_{X,Y})(f) = F\id_Y \circ f \circ F\id_X$, for every morphism $f:FX\rightarrow FY$ in $\dd$. A naturally semifull semifunctor is obviously semifull. In Proposition \ref{prop:char-sff} we show that a semifunctor is semifully faithful if and only if it is separable and naturally semifull. Then, we obtain a Rafael-type Theorem for naturally semifull semifunctors which are part of a semiadjunction in terms of semisplitting properties for the unit and the counit (Theorem \ref{thm:Raf-nat.full}). Explicitly, given a semiadjunction $F\dashv_\mathrm{s} G:\dd\to \cc$ with unit $\eta$ and counit $\epsilon$, we prove that $F$ is naturally semifull if and only if $\eta$ is a natural semisplit-epi and that $G$ is naturally semifull if and only if $\epsilon$ is a natural semisplit-mono.
\par
Recently, in \cite{AB22} semiseparable functors have been introduced in order to treat separability and natural fullness in a unified way.
The same notion fits well for semifunctors. In fact, a semifunctor results to be separable (resp. naturally semifull) if and only if it is semiseparable and faithful (resp. semiseparable and semifull) (Proposition \ref{prop:sep-nat.full}). In Proposition \ref{prop:completion-ffs} and Corollary \ref{cor:sep-natsful-semiff-compl} through the idempotent completion we show that a semifunctor $F:\cc\to\dd$ is semiseparable (resp. semifull, naturally semifull, separable, semifully faithful) if and only if its completion $F^\natural:\cc^\natural\to\dd^\natural$ is a semiseparable (resp. full, naturally full, separable, fully faithful) functor. As a consequence, if we consider a semiadjoint triple $F\dashv_\mathrm{s} G\dashv_\mathrm{s} H:\cc\to\dd$ of semifunctors, these properties pass from $F$ to $H$ and viceversa (Proposition \ref{prop:semiadj-triple}). 
\par
Finally, we provide examples of semifunctors on which we test the properties studied so far. The first one we consider is the so-called \emph{forgetful semifunctor} (Example \ref{es:iotaequtr}). Given a category $\cc$ and its idempotent completion $\cc^\natural$, the forgetful semifunctor $\upsilon_\cc : \cc^\natural \to\cc$, which maps an object $\left( X,e\right) $ in $\cc^\natural$ to the underlying object $X$ and a morphism $f:(X,e)\to (X',e') $ in $\cc^\natural$ to the underlying morphism $\upsilon_\cc f:X\rightarrow X'$ in $\cc$ such that $e'\circ\upsilon_\cc f\circ e=\upsilon_\cc f$, results to be semifully faithful. Then, we show that there are semifunctors which are neither faithful, nor semifull in general, e.g. the \emph{semi-product semifunctor} (Example \ref{es:semiproduct}) and the \emph{constant semifunctor} (Example \ref{es:cost}). To any idempotent seminatural transformation $e=(e_X)_{X\in\cc}:\id_\cc\to\id_\cc$ on a category $\cc$ it is possible to attach a canonical semifunctor $E^e$ that is self-semiadjoint (Proposition \ref{prop:E}). In Example \ref{es:LH} we show that it reveals to be naturally semifull, while it is separable if and only if $E^e=\id_\cc$. Then, given a semiadjunction $F\dashv_\mathrm{s} G:\dd\to\cc$ of semifunctors, we get the semiadjunction $FE^e\dashv_\mathrm{s} E^eG:\dd\to\cc$ (Corollary \ref{cor:F'G'}). In Example \ref{es:morph-ring} we see an instance of a semiadjunction constructed in this way. Explicitly, given a morphism of rings $\varphi :R\to S$, we consider the restriction of scalars functor $\varphi_* : S\text{-}\mathrm{Mod}\rightarrow R\text{-}\mathrm{Mod}$ and the extension of scalars functor $\varphi^*:= S\otimes_{R}(-):R\text{-}\mathrm{Mod}\rightarrow S\text{-}\mathrm{Mod}$, which form an adjunction $\varphi^*\dashv\varphi_*$. If $e=(e_X)_{X\in R\text{-}\mathrm{Mod}}:\id_{R\text{-}\mathrm{Mod}}\to\id_{R\text{-}\mathrm{Mod}}$ is an idempotent seminatural transformation, then we obtain the semiadjunction $\varphi^*_e\dashv_\mathrm{s}\varphi_{*}^e:S\text{-}\mathrm{Mod}\rightarrow R\text{-}\mathrm{Mod}$, where $\varphi^*_e:=\varphi^*\circ E^e$ and $\varphi_{*}^e:=E^e\circ\varphi_*$. In \cite[Proposition 3.1]{AMCM06} it is shown that the functor $\varphi_*$ is naturally full if and only if it is full, while $\varphi^*$ is naturally full if and only if $\varphi$ is a split-epi as an $R$-bimodule map, i.e. if there is $\pi\in {}_{R}\Hom_{R}(S,R)$ such that $ \varphi\circ \pi=\id$. In Proposition \ref{prop:ringmorph-natsemifull} we give conditions under which $\varphi_*^e$ and $\varphi^*_e$ are naturally semifull. We prove that, if $\varphi^*_e$ is naturally semifull, then there is $\psi\in {}_{R}\Hom_{R}(S,R)$ such that $r_S^{-1}\varphi:R\to \varphi_*^e\varphi^*_eR$ is an $R$-semisplit-epi as an $R$-bimodule map through $\psi r_S$, where $r_S:S\otimes_R R\to S$ is the canonical isomorphism; if in addition $1_S=e_{\varphi_*(S)}(\varphi(e_R(1_R)))$ holds true, then $r_S^{-1}\varphi$ is an $(R,R)$-semisplit-epi.\par It is known that a monoid can be seen as a category with a single object and arrows the elements of the monoid. Any semigroup homomorphism between monoids defines a semifunctor. In Example \ref{monoid} we exhibit a semifunctor between monoids that is separable, but not semifull in general, hence not even naturally semifull. Similarly, in Example \ref{es:ring}, we see an example of a semifunctor between unital rings (viewed as categories with a single object) which is naturally semifull but not separable in general. In particular, we show that the non-unital homomorphism of rings $f:R\to \mathrm{M}_n(R)$, $m\mapsto m\mathrm{E}_{ii}$, where $\mathrm{M}_n(R)$ is the ring of square matrices of order $n\in \mathbb{N}$ with coefficients in the unital ring $R$ and $\mathrm{E}_{ii}=(\delta_{ia}\delta_{ib})_{ab}$ is the matrix unit, defines a semifully faithful semifunctor.\\ 
\par
The paper is organized as follows. In Section \ref{sect:semifunct} we recall mainly from \cite{Ho93} the notions of semifunctor, seminatural transformation and semiadjunction between semifunctors. We remind the idempotent completion construction. We attach to any idempotent seminatural transformation $e=(e_X)_{X\in\cc}:\id_\cc\to\id_\cc$ on a category $\cc$ a canonical semifunctor $E^e$ that results to be self-semiadjoint. Section \ref{sect:semisplit-morph} deals with the announced semisplitting properties for morphisms whose source or target is the image of a semifunctor. 
In particular, given semifunctors $F:\cc\to\dd$, $F':\cc'\to\dd$, we define the notions of $F_C$-semisplit-mono, $F_C$-semisplit-epi, $(F_C,F'_{C'})$-semisplit-mono, $(F_C,F'_{C'})$-semisplit-epi, $(F_C,F'_{C'})$-semi-isomorphism, and we discuss their properties. In Section \ref{sect:faithful-semifull} we introduce and investigate the notions of semifull and semifully faithful semifunctor. In Section \ref{sect:separ-natsemifull} we focus on the separability property for semifunctors. We prove a Maschke-type Theorem and a Rafael-type Theorem for separable semifunctors. Section \ref{sect:natsemifull} treats the notion of natural semifullness for semifunctors and provides a Rafael-type Theorem for naturally semifull semifunctors. In Section \ref{sect:semisep} we study semiseparable semifunctors. Section \ref{sect:examples} collects examples of (naturally) semifull, (semi)separable and semifully faithful semifunctors. \\
\par
\noindent\emph{Notations.} Given an object $X$ in a category $\cc$, the identity morphism on $X$ will be denoted either by $\id_X$ or $X$ for short. For categories $\cc$ and $\dd$, a functor $F:\cc\to \dd$ just means a covariant functor. By $\id_{\cc}$ we denote the identity functor on $\cc$. For any (semi)functor $F:\cc\to \dd$, we denote $\id_{F}:F\to F$ the natural transformation defined by $(\id_{F})_X:=\id_{FX}$, for any $X$ in $\cc$. The symbol $\circ$ is used for the composition of composable morphisms, composable functors, composable natural transformations, and sometimes it is omitted for short.

\section{Semifunctors and semiadjunctions}\label{sect:semifunct}
In this section we recall mainly from \cite{Ha85} and \cite{Ho93} the notions of semifunctor, seminatural transformation and natural semi-isomorphism. We introduce the notions of \emph{natural semisplit-mono} and \emph{natural semisplit-epi} for seminatural transformations. We remind the idempotent completion construction  which provides a canonical way to turn semifunctors into functors, and we review the notion of semiadjunction \cite{Ho93} between semifunctors and its main properties. 
\subsection{Semifunctors and seminatural transformations}
\begin{defn}\cite[Definition 1.1]{Ha85}
Let $\cc$ and $\dd$ be categories. A \emph{semifunctor} $F:\cc\to\dd$ is the datum of an object map $\mathrm{Obj}(\cc )\to\mathrm{Obj}(\dd )$, $X\mapsto F(X)$, between the classes of objects of $\cc$ and $\dd$, and of a morphism map $\f_{X,Y}:\Hom_\cc(X,Y)\to\Hom_\dd(F(X),F(Y))$, $f\mapsto F(f)$, for every pair of objects $X, Y$ in $\cc$, preserving compositions, i.e. $F(g\circ f)=F(g)\circ F(f)$, for every pair of composable morphisms $f:X\to Y$, $g:Y\to Z$ in $\cc$.
\end{defn}
The image of an object $X\in\cc$ through a semifunctor $F:\cc\to\dd$ is written $F(X)$ or simply $FX$; the image of a morphism $f:X\to Y$ in $\cc$ is written $F(f)$ or just $Ff$. A semifunctor is defined as a functor but it is not required to preserve identities. Then, any functor is a semifunctor. 
Note that the image of an identity morphism $\id_X$ through a semifunctor $F:\cc\to\dd$ is an idempotent morphism in $\dd$ as $F(\id_X)=F(\id_X\circ\id_X)=F(\id_X)\circ F(\id_X)$.
The notion of semifunctor appeared in \cite[Definition 4.1]{EZ76} under the name of \emph{weak functor}. In \cite[1.284]{FrSc90} a semifunctor is what is called \emph{prefunctor}. There is a related notion of morphism between semifunctors. As in the functorial case, a \emph{natural transformation} $\alpha: F\to F^{\prime }$ between semifunctors $F,F^{\prime}:\cc\rightarrow \dd$ is defined as a family $(\alpha_X:FX\to F^{\prime }X)_{X\in \cc}$ of morphisms in $\dd$ such that $\alpha _{X'}\circ Ff=F^{\prime }f\circ \alpha _{X}$ for any morphism $f:X\rightarrow X'$ in $\cc$. Given a semifunctor $F:\cc\to\dd$, there is a natural transformation $F\id: F\to F$ with components $F\id_X:FX\to FX$ and a natural transformation $\id_F: F\to F$ with components $\id_{FX}:FX\to FX$. Note that $F\id\neq \id_F$ in general, unless $F$ is a functor.
Moreover, see \cite[Definition 2.4]{Ho93}, a \emph{seminatural transformation} $\alpha: F\to F^{\prime }$ between semifunctors $F,F^{\prime}:\cc\rightarrow \dd$ is a natural transformation with the additional property that $\alpha _{X}\circ F \mathrm{Id}_{X}
=\alpha _{X}$, for every $X$ in $\cc$.
%
If (at least) one of the semifunctors $F,F^\prime$ is a functor, then the notions of natural transformation and seminatural transformation coincide \cite[Theorem 2.5]{Ho93}. Semifunctors $F, F^\prime:\cc\to\dd$ are said to be \emph{naturally semi-isomorphic} (denoted by $F\cong_\mathrm{s} F^\prime$) if and only if there are natural transformations $\alpha: F\to F^\prime$ and $\beta: F^\prime\to F$ such that \begin{itemize}
\item[(i)] $\alpha\circ F\id =\alpha$;
\item[(ii)] $\beta\circ F^\prime\id =\beta$;
\item[(iii)] $\alpha\circ\beta = F^\prime\id$;
\item[(iv)] $\beta\circ\alpha = F\id$.
\end{itemize}
In this case $\alpha$ is said to be a \emph{natural semi-isomorphism} \cite[Subsection 2.2]{Ho93}. Since its semi-inverse $\beta$ is uniquely determined by $\alpha$, it will be usually denoted by $\alpha^{-1}$. If $F, F^\prime$ are functors, then any natural semi-isomorphism $\alpha: F\to F^\prime$ is actually a natural isomorphism.\par
We introduce the following terminology for ``semisplitting'' properties of a seminatural transformation $\alpha:F\to F'$ between semifunctors $F, F^\prime:\cc\to\dd$. We say that $\alpha$ is a
\begin{itemize}
\item \textbf{natural semisplit-mono} if there exists a seminatural transformation $\beta:F'\to F$ such that $\beta\circ\alpha= F\id$;
\item \textbf{natural semisplit-epi} if there exists a seminatural transformation $\beta:F'\to F$ such that $\alpha\circ\beta= F'\id$.
\end{itemize}
Moreover, $\alpha$ is a \emph{natural split-mono} (resp. \emph{natural split-epi}), if there exists a seminatural transformation $\beta:F'\to F$ such that $\beta\circ\alpha=\id_F$ (resp. $\alpha\circ \beta=\id_{F'}$). Note that if $F$ is a functor, then $\alpha$ is a natural semisplit-mono if and only if $\alpha$ is a natural split-mono; if $F'$ is a functor, then $\alpha$ is a natural semisplit-epi if and only if $\alpha$ is a natural split-epi.
\begin{lem}\label{lem:seminat-semisplit}
A seminatural transformation $\alpha:F\to F'$ between semifunctors $F,F'$ is a natural semi-isomorphism if and only if $\alpha$ is both a natural semisplit-mono and a natural semisplit-epi. 
\end{lem}
\begin{proof}
	If $\alpha$ is a natural semi-isomorphism, then it is trivially both a natural semisplit-mono and a natural semisplit-epi. Conversely, if $\alpha$ is a natural semisplit-mono and a natural semisplit-epi, then there exists a seminatural transformation $\beta:F'\to F$ such that $\beta\circ\alpha= F\id$ and there is a seminatural transformation $\beta':F'\to F$ such that $\alpha\circ\beta'= F'\id$. Note that $\beta=\beta\circ F^\prime\id =\beta\circ\alpha\circ\beta'=F\id\circ\beta'=\beta'$, thus $\alpha$ is a natural semi-isomorphism.
\end{proof}
In Section \ref{sect:semisplit-morph} we will study the corresponding semisplitting properties for the component morphisms of a seminatural transformation.
\subsection{Idempotent completion} \label{subsect:idempcompl}
Recall that an endomorphism $e:X\to X$ in a category $\cc$ is \emph{idempotent} if $e^2:=e\circ e=e$. An idempotent morphism $e:X\to X$ in $\cc$ \emph{splits} if there exist two morphisms $h:X\to Y$ and $k:Y\to X$ in $\cc$ such that $e=k\circ h$ and $h\circ k=\id _Y$; the category $\cc$ is said to be \emph{idempotent complete} or \emph{Cauchy complete} if all idempotents split. As an instance, any category equipped with (co)equalizers is idempotent complete \cite[Theorem 2.15]{Ho93}, see also \cite[Proposition 6.5.4]{Bor94}.  The \emph{idempotent completion} $\cc^\natural$ (also known under the names of \emph{Cauchy completion} \cite{Law73} or \emph{Karoubi envelope} \cite{Kar78}) of a category $\cc$ is the category whose objects are pairs $(X,e)$, where $X$ is an object in $\cc$ and $e:X\to X$ is an idempotent morphism in $\cc$, and a morphism $f:(X,e)\to (X',e')$ in $\cc^\natural$ is a morphism $f:X\to X'$ in $\cc$ such that $f=e'\circ f\circ e$ (or equivalently, such that $e'\circ f=f=f\circ e$). Note that $\id_{(X,e)}\neq\id_X$ but $\id_{(X,e)}=e:(X,e)\to (X,e)$. The category $\cc^\natural$ is idempotent complete. There is a canonical functor $\iota_\cc : \cc\to\cc^\natural$, $X\mapsto (X,\id_X)$, $[f:X\to Y]\mapsto [f:(X,\id_X)\to (Y,\id_Y)]$, which is fully faithful; $\iota_\cc$ is an equivalence if and only if $\cc$ is idempotent complete. Given any functor $F:\cc\to\dd$, it can be extended to a functor $F^\natural : \cc^\natural\to\dd^\natural$ by $F^\natural (X,e)= (FX,Fe)$ and $F^\natural (f)=F(f)$, so that $\iota_\dd\circ F =F^\natural\circ\iota_\cc$. Any semifunctor $F:\cc\rightarrow \dd$ induces a functor $%
F^{\natural }:\cc^{\natural }\rightarrow \dd^{\natural }$
such that $F^{\natural }\left( X,e\right) =\left( FX,Fe\right) $ and $%
F^{\natural }f=Ff.$ In fact $F^{\natural }\mathrm{Id}_{\left( X,e\right)
}=Fe=\mathrm{Id}_{\left( FX,Fe\right) }=\mathrm{Id}_{F^{\natural }\left(
X,e\right) },$ as observed in \cite[Definition 1.3]{Ha85}. Note that in general $\iota_\dd\circ F\neq F^\natural\circ\iota_\cc$ unless F is a functor. On the other hand, if $G:\cc^{\natural }\to \dd^{\natural }$ is a functor, then there is a semifunctor $F:\cc\to\dd$ given as in the proof of \cite[Theorem 1]{HoM95} such that $F^\natural = G$, cf. \cite[Proposition 1.4]{Ha85}. Moreover, given semifunctors $F,F':\cc\to\dd$, $F\cong_\mathrm{s} F'$ is a natural semi-isomorphism if and only if $F^\natural\cong (F')^\natural$ is a natural isomorphism of functors \cite[Theorem 2.12]{Ho93}.
A (semi)natural transformation $\alpha :F\rightarrow F^{\prime }$ of semifunctors $F,F':\cc\to\dd$ induces the natural transformation $\alpha ^{\natural }:F^{\natural }\rightarrow \left( F^{\prime}\right) ^{\natural }$ with components $\alpha _{\left( X,e\right)
}^{\natural }:=\alpha _{X}\circ Fe=F^{\prime }e\circ \alpha _{X}$, cf. \cite[Theorem 7.3]{Ho93}. \begin{invisible} Indeed, for any morphism $f:(X,e)\to (X',e')$ in $\cc^\natural$ we have that ${F'}^{\natural}f\circ \alpha^\natural_{(X,e)}=F'f\circ F'e\circ\alpha_X=F'(fe)\circ\alpha_X=F'(e'f)\circ\alpha_X=F'e'\circ F'f\circ\alpha_X=F'e'\circ\alpha_{X'}\circ Ff=\alpha_{X'}\circ Fe'\circ Ff=\alpha^\natural_{(X',e')}\circ F^\natural f$.
 \end{invisible}
The category $\Cat_\mathrm{s}$ with categories as objects, semifunctors as arrows, and seminatural transformations as $2$-cells is a $2$-category \cite[Theorem 7.2]{Ho93}. Since any functor is in particular a semifunctor, there is an inclusion of the $2$-category $\Cat$ of categories, functors and natural transformations, in $\Cat_\mathrm{s}$. Conversely, the idempotent completion construction provides a canonical way to transform semifunctors into functors. In fact, the \emph{Karoubi envelope functor} $\kappa:\Cat_\mathrm{s}\to\Cat$, defined by $\kappa (\cc)=\cc^\natural$, $\kappa (F)=F^\natural$, for any category $\cc$ and any semifunctor $F:\cc\to\dd$, is the right adjoint of the inclusion functor $i:\Cat\to\Cat_\mathrm{s}$ (see \cite[Theorem 2.10]{Ho93}). Moreover, $\kappa$ defines an equivalence of 2-categories between $\Cat_\mathrm{s}$ and the full 2-subcategory of $\Cat$ having idempotent complete categories as objects \cite[Theorem 1]{HoM95}. Thus, many standard properties for functors can be extended to semifunctors, as for instance the notion of adjunction.

\subsection{Semiadjunctions}
Let $F:\cc\to\dd$ be a semifunctor, let $\cc^\mathrm{op}$ be the opposite category of $\cc$, and consider the semifunctor $$\Hom_\dd(F-,-):\cc^{\mathrm{op}}\times \dd\to \Set,$$
$$(C,D)\mapsto \Hom_\dd(FC,D),\quad (f:C'\to C, g:D\to D')\mapsto \Hom_\dd(Ff,g)(-)=g\circ -\circ Ff$$
Since in general, for any morphism $h:FC\to D$ in $\dd$, $\Hom_\dd(F\id_C,\id_D)(h)=\id_D\circ h\circ F\id_C=h\circ F\id_C\neq h$, then $\Hom_\dd(F-,-)$ is really a semifunctor. Analogously, for a semifunctor $G:\dd\to\cc$ one can consider the semifunctor $\Hom_\cc(-,G-):\cc^{\mathrm{op}}\times\dd\to\Set$. 
\begin{defn}\cite[Definition 3.1]{Ho93}
A \emph{semiadjunction} is a triple $(F:\cc\to\dd,G:\dd\to\cc,\tau)$, where $F$, $G$ are semifunctors and $\tau$ is a natural semi-isomorphism given, for any $C\in\cc$ and $D\in \dd$, by
\begin{equation}\label{tau}
\tau_{C,D}: \Hom_{\dd}(FC,D)\cong_\mathrm{s} \Hom_{\cc}(C,GD)
\end{equation}
\end{defn}
\begin{invisible}
\begin{rmk}
For each semifunctor $F:\cc\to\dd$ there exists a functor denoted by $\Hom_\dd(F-,-)_\mathrm{s}$ that is semi-isomorphic to $\Hom_\dd(F-,-)$ \cite[Theorem 2.21]{Ho93}, hence a semiadjunction can be also defined as a triple $(F:\cc\to\dd,G:\dd\to\cc,\psi)$, where $F$, $G$ are semifunctors and $\psi$ is a natural isomorphism \begin{equation}\label{psi}
\psi_{C,D}: \Hom_{\cc}(C,GD)_\mathrm{s}\cong\Hom_{\dd}(FC,D)_\mathrm{s}.
\end{equation}
\end{rmk}
\end{invisible}
Equivalently, by \cite[Theorem 3.10]{Ho93} a semiadjunction $\left( F,G,\eta
,\epsilon \right) $ is the datum of semifunctors $F:\cc\rightarrow \dd$ and $G:%
\dd\rightarrow \cc$ equipped with natural transformations $\eta :\mathrm{Id}_{\cc}\rightarrow GF$ (unit) and $\epsilon :FG\rightarrow \mathrm{Id}_{\dd}$ (counit) such that the ``semitriangular'' identities
\begin{equation}\label{eq:semitr-id}
G\epsilon \circ \eta G=G\mathrm{Id}\text{ and }\epsilon F\circ F\eta =F\mathrm{Id}
\end{equation}
hold true, see also \cite[Definition 22]{Ho90}. In particular, $\eta$ and $\epsilon$ are indeed seminatural transformations. We usually denote a semiadjunction $\left( F,G, \eta
,\epsilon  \right)$ by $F\dashv_\mathrm{s} G$. Note that the natural semi-isomorphism $\tau$ in \eqref{tau} is given as in the functorial case by 
\begin{equation}\label{eq:tau}
\tau_{C,D}(h)=G(h)\circ\eta_C,
\end{equation}
for any morphism $h:FC\to D$ in $\dd$. Its semi-inverse $\sigma$ is given by 
\begin{equation}\label{eq:sigma}
\sigma_{C,D}(g)=\epsilon_D\circ F(g),
\end{equation}
for any $g:C\to GD$ in $\cc$. 
\begin{invisible} Indeed, for any $C\in \cc$, $D\in\dd$, we have that $(\tau_{C,D}\circ\Hom_\dd(F\id_C,\id_D))(h)=\tau_{C,D}(\id_D\circ h\circ F\id_C)=\tau_{C,D}(h\circ F\id_C)=Gh\circ GF\id_C\circ\eta_C= Gh\circ \eta_C=\tau_{C,D}(h)$, for every $h:FC\to D$ in $\dd$. Analogously, we have $(\sigma_{C,D}\circ\Hom_\cc(\id_C,G\id_D))(g)=\sigma_{C,D}(G\id_D\circ g\circ\id_C)=\epsilon_D\circ FG\id_D\circ Fg=\epsilon_D\circ Fg=\sigma_{C,D}(g)$, for every $g:C\to GD$ in $\cc$. Moreover, for every $g:C\to GD$ in $\cc$ we have that $\tau_{C,D}\sigma_{C,D}(g)=\tau_{C,D}(\epsilon_D\circ Fg)=G(\epsilon_D\circ Fg)\circ\eta_C= G\epsilon_D\circ GFg\circ\eta_C=G\epsilon_D\circ\eta_{GD}\circ g=G\id_D\circ g=G\id_D\circ g\circ\id_C=\Hom_\cc(\id_C,G\id_D)(g)$, and for every morphism $h:FC\to D$ in $\dd$ we have that $\sigma_{C,D} \tau_{C,D}(h)=\sigma_{C,D}(Gh\circ\eta_C)=\epsilon_D\circ FGh\circ F\eta_C=h\circ\epsilon_{FC}\circ F\eta_C=h\circ F\id_C=\id_D\circ h\circ F\id_C=\Hom_\dd(F\id_C,\id_D)(h)$. It is easy to see that $\tau$ and $\sigma$ are natural.
Indeed, for any morphism $f:C'\to C$ in $\cc^\mathrm{op}$ and $g:D\to D'$ in $\dd$, we have that $(\tau_{C,D'}\circ \Hom_\dd(Ff,g))(h)=\tau_{C,D'}(g\circ h\circ Ff)=Gg\circ Gh\circ GFf\circ\eta_{C}=Gg\circ Gh\circ\eta_{C'}\circ f=Gg\circ\tau_{C',D}(h)\circ f$ and $(\sigma_{C,D'}\circ\Hom_\cc(f, Gg))(k)=\sigma_{C,D'}(Gg\circ k\circ f)=\epsilon_{D'}\circ FGg\circ Fk\circ Ff=g\circ\epsilon_D\circ Fk\circ Ff=g\circ\sigma_{C'D}(k)\circ Ff$.
\end{invisible}
Any adjunction of functors is trivially a semiadjunction, and if $(F,G,\eta,\epsilon)$ is a semiadjunction, then $(F^\natural,G^\natural,\eta^\natural,\epsilon^\natural)$ is an adjunction of functors \cite[Theorem 1.9]{Ha85}, 
with unit and counit given on components by $\eta _{\left( C,c\right) }^{\natural
}=\eta _{C}\circ c:\left( C,c\right) \rightarrow \left( GFC,GFc\right) $ and $\epsilon _{\left( D,d\right) }^{\natural }=d\circ \epsilon _{D}:\left(FGD,FGd\right) \rightarrow \left( D,d\right)$, respectively. 
\begin{invisible}
Indeed, $\eta _{\left( C,c\right) }^{\natural}=\eta_C\circ c$ is a morphism in $\cc^\natural$ as $\eta_C\circ c=\eta_C\circ c\circ c=GFc\circ\eta_C\circ c=GFc\circ(\eta_C\circ c)\circ c$; $\epsilon _{\left( D,d\right) }^{\natural }=d\circ \epsilon _{D}$ is a morphism in $\dd^\natural$ as $d\circ \epsilon_D=d\circ d\circ \epsilon_D=d\circ \epsilon_D\circ FGd=d\circ (d\circ \epsilon_D)\circ FGd$; $\eta^\natural$ is natural as for any morphism $f:(C,c)\to (C',c')$ in $\cc^\natural$ we have $\eta^\natural_{(C',c')}\circ f=\eta_{C'}\circ c'\circ f=\eta_{C'}\circ f\circ c= GFf\circ\eta_C\circ c=G^\natural F^\natural f\circ\eta^\natural_{(C,c)}$; $\epsilon^\natural$ is natural as for any morphism $g:(D,d)\to (D',d')$ in $\dd^\natural$ we have $\epsilon^\natural_{(D',d')}\circ FGg=d'\circ \epsilon_{D'}\circ FGg=d'\circ g\circ \epsilon_{D}=g\circ d\circ\epsilon_D=g\circ\epsilon^\natural_{(D,d)}$. Moreover, $G^\natural\epsilon^\natural\circ\eta^\natural G^\natural=G^\natural\id$ as for every $(D,d)\in\dd^\natural$ we have $G^\natural\epsilon^\natural_{(D,d)}\circ\eta^\natural_{(GD,Gd)}=G^\natural(d\circ\epsilon_D)\circ\eta_{GD}\circ Gd=Gd\circ G\epsilon_D\circ\eta_{GD}\circ Gd=Gd\circ G\id_D\circ Gd=Gd=\id_{(GD,Gd)}=G^\natural\id_{(D,d)}$, and $\epsilon^\natural F^\natural\circ F^\natural\eta^\natural=F^\natural\id$ as $\epsilon^\natural_{ F^\natural(C,c)}\circ F^\natural\eta^\natural_{(C,c)}=\epsilon^\natural_{(FC,Fc)}\circ F(\eta_C\circ c)=Fc\circ\epsilon_{FC}\circ F\eta_C\circ Fc=Fc\circ F\id_C\circ Fc=Fc=\id_{(FC,Fc)}=F^\natural\id_{(C,c)}$.
\end{invisible}
Moreover, as shown in \cite[Theorem 3.5]{Ho93}, $F\dashv_\mathrm{s} G$ if and only if $F^\natural\dashv G^\natural$.
It is known that semiadjoint semifunctors are not unique up to isomorphism, but they are unique up to natural semi-isomorphism, cf. \cite[Theorem 3.6]{Ho93}. We include a proof for completeness sake. Cf. \cite[Proof of Proposition 9]{CMZ02} for the case of functors.
\begin{prop}\label{prop:uniqueness-semiiso}
Let $F\dashv_\mathrm{s}G$, $F\dashv_\mathrm{s} G'$ be semiadjunctions of semifunctors. Then, $G$ and $G'$ are naturally semi-isomorphic.
\end{prop}
\begin{proof}
Let $F\dashv_\mathrm{s}G$, $F\dashv_\mathrm{s} G'$ be semiadjunctions with units $\eta$, $\eta'$, and counits $\epsilon$, $\epsilon'$, respectively. Consider $\gamma := G'\epsilon \circ\eta' G:G\to G'$ and $\gamma' := G\epsilon' \circ\eta G':G'\to G$. Note that $\gamma\circ G\id=G'\epsilon \circ\eta' G\circ G\id=G'\epsilon\circ G'FG\id\circ\eta' G=G'(\epsilon\circ FG\id)\circ\eta' G=G'\epsilon\circ\eta' G=\gamma$, and $\gamma'\circ G'\id=G\epsilon'\circ \eta G'\circ G'\id= G\epsilon'\circ GFG'\id\circ\eta G'=G(\epsilon'\circ FG'\id)\circ\eta G'=G\epsilon'\circ\eta G'=\gamma'$. Moreover, $\gamma$ and $\gamma'$ are natural as they are composition of natural transformations,
\begin{invisible} for any arrow $f:D\to D'$ in $\dd$ we have $\gamma_{D'}\circ Gf= G'\epsilon_{D'}\circ\eta'_{GD'}\circ Gf=G'\epsilon_{D'}\circ G'FG f\circ\eta'_{GD}=G'(\epsilon_{D'}\circ FGf)\circ\eta'_{GD}=G'(f\circ\epsilon_D)\circ \eta'_{GD}=G'f\circ G'\epsilon_D\circ\eta'_{GD}=G'f\circ\gamma_D$, and analogously $\gamma'_{D'}\circ G'f=Gf\circ\gamma'_D$,\end{invisible} 
so they are seminatural transformations. From naturality of $\eta$ it follows that $\eta G'\circ G'\epsilon =GFG'\epsilon\circ\eta G'FG$ and $\eta G'FG\circ\eta'G=GF\eta'G\circ\eta G$, and from naturality of $\epsilon'$ it follows that $G\epsilon\circ G\epsilon'FG=G\epsilon'\circ GFG'\epsilon$. Then, we obtain 
\[
\begin{split}
\gamma'\circ\gamma&= G\epsilon'\circ \eta G'\circ G'\epsilon\circ\eta' G= G\epsilon'\circ GFG'\epsilon\circ\eta G'FG\circ\eta' G\\
&= G\epsilon\circ G\epsilon' FG\circ GF\eta' G\circ\eta G= G\epsilon\circ GF\Id_G \circ\eta G=G\epsilon\circ \eta G=G\Id . 
\end{split}
\]
Similarly, from naturality of $\eta'$ and $\epsilon$, we have \[
\begin{split}
\gamma\circ\gamma'&= G'\epsilon\circ \eta' G\circ G\epsilon'\circ\eta G'= G'\epsilon\circ G'FG\epsilon'\circ\eta' GFG'\circ\eta G'\\
&= G'\epsilon'\circ G'\epsilon FG'\circ G'F\eta G' \circ\eta' G'= G'\epsilon'\circ G'F\Id_{G'} \circ\eta' G'=G'\epsilon'\circ \eta' G'=G'\Id .\qedhere
\end{split}
\]
\end{proof}
In \cite{Ho93} the terminology of \emph{right} (resp. \emph{left}) \emph{semiadjoint} is used to denote a semifunctor $G$ (resp. $F$) that is part of a semiadjunction $F\dashv_\mathrm{s} G$, that is, both semitriangular identities \eqref{eq:semitr-id} hold true. In the following definition we adopt the same terminology with a weaker meaning, inspired by \cite[Definition 1.3]{MW13} for functors. 
\begin{defn}\label{defn:rl-semiadj}
We say that: 
\begin{itemize}	
\item[(i)] a semifunctor $G:\dd\to\cc$ is a \emph{right semiadjoint} if there exist a semifunctor $F:\mathcal{C}\rightarrow \mathcal{D}$ and seminatural transformations $\eta :\mathrm{Id}_{\mathcal{C}}\rightarrow GF$ and $\epsilon :FG\rightarrow \mathrm{Id}_{\mathcal{D}}$, such that $G\epsilon \circ \eta G=G\mathrm{Id}$;
\item[(ii)] a semifunctor $F:\cc\to\dd$ is a \emph{left semiadjoint} if there exists a semifunctor $G:\mathcal{D}\rightarrow \mathcal{C}$ and seminatural transformations $\eta :\mathrm{Id}_{\mathcal{C}}\rightarrow GF$ and $\epsilon :FG\rightarrow \mathrm{Id}_{\mathcal{D}}$, such that $\epsilon F\circ F\eta=F\id$.
\end{itemize}
\end{defn}
\begin{rmk}\label{rmk:semiadj}
In a semiadjunction $F\dashv_\mathrm{s} G$, $F$ is a left semiadjoint and $G$ is a right semiadjoint.
\end{rmk}
Now we show that a right (resp. left) semiadjoint is actually part of a semiadjunction (cf. \cite[Lemma 2.16]{AB22-II}). 
In particular, we have the following characterization of left and right semiadjoints.
\begin{prop}\label{prop:semiadj}
\begin{enumerate}
\item[$(1)$] A semifunctor $G:\dd\to\cc$ is a right semiadjoint if and only if there is a
semifunctor $F':\mathcal{C}\rightarrow \mathcal{D}$ (unique up to natural semi-isomorphism), such that $F'\dashv_\mathrm{s}G $ is a semiadjunction.
\item[$(2)$] A semifunctor $F:\cc\to\dd$ is a left semiadjoint if and only if there is a
semifunctor $G':\mathcal{D}\rightarrow \mathcal{C}$ (unique up to natural semi-isomorphism), such that $ F\dashv_\mathrm{s} G' $ is a semiadjunction.
\end{enumerate}
\end{prop}

\begin{proof}
$(1)$. If $ F^{\prime }\dashv_\mathrm{s} G $ is a semiadjunction, then by Remark \ref{rmk:semiadj} the semifunctor $G$ is a right semiadjoint and $F':\cc\to\dd$ is a left semiadjoint. Conversely, if $G$ is a right semiadjoint, then there exist a semifunctor $F:\mathcal{C}\rightarrow \mathcal{D}$ and seminatural transformations $\eta :\mathrm{Id}_{\mathcal{C}}\rightarrow GF$ and $\epsilon :FG\rightarrow \mathrm{Id}_{\mathcal{D}}$, such that $G\epsilon \circ \eta G=G\mathrm{Id}$. Set $e:=\epsilon F\circ F\eta :F\rightarrow F$, which is an idempotent seminatural transformation. Indeed, it is natural as it is composition of natural transformations; for any $X\in \cc$ we have $e_X\circ F\id_X=\epsilon_{FX}\circ F\eta_X\circ F\id_X=\epsilon_{FX}\circ F(\eta_X\circ\id_X)=\epsilon_{FX}\circ F\eta_X=e_X$ and, cf. e.g. \cite[Lemma 1.4(2)]{MW13}, $e\circ e=\epsilon F\circ F\eta \circ \epsilon F\circ F\eta =\epsilon F\circ\epsilon FGF\circ FGF\eta \circ F\eta =\epsilon F\circ FG\epsilon F\circ F\eta GF\circ F\eta =\epsilon F\circ FG\Id_{F}\circ F\eta=\epsilon F\circ F\eta =e.$ 
Then, there is a semifunctor $F^{\prime }:\mathcal{C}%
\rightarrow \mathcal{D}$ that acts as $F$ on objects and sends a morphism $%
f:X\rightarrow Y$ in $\cc$ to $Ff\circ e_{X}$. Indeed, given $f:X\rightarrow Y$ and $g:Y\rightarrow Z$ in $\mathcal{C}$ we have %
$F^{\prime }g\circ F^{\prime }f=Fg\circ e_{Y}\circ Ff\circ e_{X}=Fg\circ
Ff\circ e_{X}\circ e_{X}=F\left( g\circ f\right) \circ e_{X}=F^{\prime
}\left( g\circ f\right)$, 
so that $F^{\prime }$ is a semifunctor. Now we show that $\left( F^{\prime
},G,\eta ^{\prime },\epsilon ^{\prime }\right) $ is a semiadjunction where $%
\eta _{C}^{\prime }:=\eta _{C}$ and $\epsilon _{D}^{\prime }:=\epsilon _{D}$, for every object $C\in\cc$ and $D\in\dd$.
Note that by the assumption $G\epsilon \circ \eta G=G\mathrm{Id}$, we get $\epsilon _{D}\circ e_{GD} =\epsilon _{D}\circ \epsilon _{FGD}\circ F\eta
_{GD}=\epsilon _{D}\circ FG\epsilon _{D}\circ F\eta _{GD}=\epsilon _{D}\circ
F\left( G\epsilon _{D}\circ \eta _{GD}\right) =\epsilon _{D}\circ FG \mathrm{Id}_{D} =\epsilon _{D},$ for every $D\in\mathcal{D}$, where the last equality follows from the seminaturality of $\epsilon$, so
\begin{equation}  \label{e1}
\epsilon \circ eG =\epsilon.
\end{equation}  
For every $D\in\mathcal{D}$, we have $\epsilon _{D}^{\prime }\circ
F^{\prime }G\mathrm{Id}_{D}=\epsilon_D\circ FG\id_D\circ e_{GD}\overset{\eqref{e1}}{=}\epsilon_D\circ e_{GD}\circ FG\id_D\circ e_{GD}=\epsilon_D\circ e_{GD}\circ e_{GD}\circ FG\id_D=\epsilon_D\circ e_{GD}\circ FG\id_D\overset{\eqref{e1}}{=}\epsilon _{D}\circ FG\mathrm{Id}_{D}=\epsilon
_{D}=\epsilon _{D}^{\prime }$, and for every morphism $f:D\rightarrow D'$ in $%
\mathcal{D}$ we have
$\epsilon _{D'}^{\prime }\circ F^{\prime }Gf=\epsilon _{D'}\circ FGf\circ
e_{GD}=f\circ \epsilon _{D}\circ e_{GD}\overset{(\ref{e1})}{=}f\circ
\epsilon _{D}=f\circ \epsilon _{D}^{\prime }$, 
so that $\epsilon ^{\prime
}:=\left( \epsilon _{D}\right) _{D\in \mathcal{D}}:F^{\prime }G\rightarrow
\mathrm{Id}_{\mathcal{D}}$ is indeed a seminatural transformation. For every object $C$ in $\mathcal{C}$, it holds $\eta _{C}^{\prime }\circ
\mathrm{Id}_{\mathcal{C}}\left( \mathrm{Id}_{C}\right) =\eta
_{C}^{\prime }\circ \mathrm{Id}_{C}=\eta _{C}^{\prime }$, and for every
morphism $f:X\rightarrow Y$ in $\mathcal{C}$ we have
\begin{equation*}
\begin{split}
GF^{\prime }f\circ \eta _{X}^{\prime }&=G\left( Ff\circ e_{X}\right) \circ
\eta _{X}=G\left( e_{Y}\circ Ff\right) \circ \eta _{X}=Ge_{Y}\circ GFf\circ
\eta _{X}\\
&=G(\epsilon_{FY}\circ F\eta_Y)\circ GFf\circ\eta _{X}=G\epsilon_{FY}\circ (GF\eta_Y\circ \eta_Y)\circ f=G\epsilon_{FY}\circ \eta_{GFY}\circ \eta_Y\circ f\\&=G\id_{FY}\circ GFf\circ\eta_X=G(\id_{FY}\circ Ff)\circ\eta_X=GFf\circ\eta_X=\eta_Y\circ f=\eta^{\prime}_Y\circ f
\end{split}
\end{equation*}
so that $\eta ^{\prime
}:=\left( \eta _{C}\right) _{C\in \mathcal{C}}:\mathrm{Id}_{\mathcal{C}%
}\rightarrow GF^{\prime }$ is indeed a seminatural transformation. Thus, from $G\epsilon _{D}^{\prime }\circ \eta _{GD}^{\prime }=G\epsilon _{D}\circ \eta
_{GD}=G\mathrm{Id}_D$ and $\epsilon _{F^{\prime }C}^{\prime }\circ F^{\prime }\eta _{C}^{\prime
}=\epsilon _{FC}\circ F^{\prime }\eta _{C}=\epsilon _{FC}\circ F\eta
_{C}\circ e_{C}=e_{C}\circ
e_{C}=e_{C}=F\id_C\circ e_C=F^{\prime }\mathrm{Id}_{C}$
it follows that $\left( F^{\prime },G,\eta ^{\prime },\epsilon ^{\prime }\right)$ is a semiadjunction. By the left analogue of Proposition \ref{prop:uniqueness-semiiso}, $F'$ is unique up to natural semi-isomorphism.
\par
$(2)$. It is dual to $(1)$.
\end{proof}
In the following proposition we show that the notion of right (resp. left) semiadjoint is stable under composition. As a consequence, as pointed out in \cite[page 4]{HoM95}, semiadjunctions remain stable under composition, similarly to the case of adjunctions of functors, cf. \cite[IV.8, Theorem 1]{Mac98}. 
\begin{prop}\label{prop:compsemiadj}
	\begin{itemize}
\item[(1)] Given two right semiadjoints $G:\dd\to\cc$ and $G':\e\to\dd$, then the composite semifunctor $G\circ G':\e\to\cc$ is a right semiadjoint. 
\item[$(2)$] Given two left semiadjoints $F:\cc\to\dd$ and $F':\dd\to\e$, then the composite semifunctor $F'\circ F:\cc\to\e$ is a left semiadjoint.
\end{itemize}
\end{prop}
\proof
(1). If $G:\dd\to\cc$ and $G':\e\to\dd$ are right semiadjoints, then by definition there exist a semifunctor $F:\mathcal{C}\rightarrow \mathcal{D}$ and seminatural transformations $\eta :\mathrm{Id}_{\mathcal{C}}\rightarrow GF$ and $\epsilon :FG\rightarrow \mathrm{Id}_{\mathcal{D}}$, such that $G\epsilon \circ \eta G=G\mathrm{Id}$, and there exist a semifunctor $F':\mathcal{D}\rightarrow \mathcal{E}$ and seminatural transformations $\eta' :\mathrm{Id}_{\mathcal{D}}\rightarrow G'F'$ and $\epsilon' :F'G'\rightarrow \mathrm{Id}_{\mathcal{E}}$, such that $G'\epsilon' \circ \eta' G'=G'\mathrm{Id}$, respectively. Set $\bar{\eta}:=G\eta' F\circ\eta$ and $\bar{\epsilon}:=\epsilon'\circ F'\epsilon G'$. We now show that the composite $G\circ G':\e\to\cc$ is a right semiadjoint through the semifunctor $F'\circ F:\cc\to\e$ and the seminatural transformations $\bar{\eta}: \id_\cc\to GG'F'F$ and $\bar{\epsilon}:F'FGG'\to\id_\e$. Indeed, we have
\begin{displaymath}
\begin{split}
GG'\bar{\epsilon}\circ\bar{\eta}GG'&=GG'\epsilon'\circ GG'F'\epsilon G'\circ G\eta' FGG'\circ\eta GG'=G(G'\epsilon'\circ G'F'\epsilon G'\circ\eta' FGG')\circ\eta GG'\\
&=G(G'\epsilon'\circ\eta' G'\circ\epsilon G')\circ\eta GG'=G(G'\id\circ\epsilon G')\circ\eta GG'= GG'\id\circ G\epsilon G'\circ\eta GG'\\&= GG'\id\circ (G\epsilon\circ\eta G)G'=GG'\id\circ G\id_{G'}=G(G'\id\circ \id_{G'})=GG'\id .
\end{split}
\end{displaymath}
(2). At the same way, if $F:\cc\to\dd$ and $F':\dd\to\e$ are left semiadjoints, then by definition there exist a semifunctor $G:\mathcal{D}\rightarrow \mathcal{C}$ and seminatural transformations $\eta :\mathrm{Id}_{\mathcal{C}}\rightarrow GF$ and $\epsilon :FG\rightarrow \mathrm{Id}_{\mathcal{D}}$, such that $\epsilon F\circ F\eta=F\id$, and there exist a semifunctor $G':\mathcal{E}\rightarrow \mathcal{D}$ and seminatural transformations $\eta' :\mathrm{Id}_{\mathcal{D}}\rightarrow G'F'$ and $\epsilon' :F'G'\rightarrow \mathrm{Id}_{\mathcal{E}}$, such that $\epsilon' F'\circ F'\eta'=F'\id$, respectively. By setting again $\bar{\eta}:=G\eta' F\circ\eta$ and $\bar{\epsilon}:=\epsilon'\circ F'\epsilon G'$, it holds that $\bar{\epsilon} F'F\circ F'F\bar{\eta}=F'F\id$, hence $F'\circ F:\cc\to\e$ is a left semiadjoint. In fact, we have that 
\begin{displaymath}
	\begin{split}
		\bar{\epsilon}F'F\circ &F'F\bar{\eta}=\epsilon'F'F\circ F'\epsilon G'F'F\circ F'FG\eta'F\circ F'F\eta =(\epsilon' F'\circ F'\epsilon G'F'\circ F'FG\eta' )F\circ F'F\eta\\
		&=(\epsilon' F'\circ F'\eta' \circ F'\epsilon )F\circ F'F\eta =(F'\id\circ F'\epsilon )F\circ F'F\eta = F'\epsilon F\circ F'F\eta=F'F\id . \qedhere
	\end{split}
\end{displaymath}
\endproof
\begin{cor}\label{cor:semiadj-compos}
Given semiadjunctions $\left( F\dashv_\mathrm{s} G:\dd\to\cc,\eta ,\epsilon \right) $ and $\left( F'\dashv_\mathrm{s} G':\e\to\dd,\eta',\epsilon' \right)$, then also $\left( F'F\dashv_\mathrm{s} GG':\e\to\cc,G\eta' F\circ\eta ,\epsilon'\circ F'\epsilon G' \right)$ is a semiadjunction. 
\end{cor}
\begin{proof}
By Remark \ref{rmk:semiadj} $G$ and $G'$ are right semiadjoints through $F, \eta,\epsilon$ and $F',\eta',\epsilon'$, respectively, and $F$, $F'$ are left semiadjoints through $G, \eta,\epsilon$ and $G', \eta',\epsilon'$, respectively. Then, by the proof of Proposition \ref{prop:compsemiadj} we know that $GG'\bar{\epsilon}\circ\bar{\eta}GG'=GG'\id$ and $\bar{\epsilon}F'F\circ F'F\bar{\eta}=F'F\id$, where $\bar{\eta}=G\eta' F\circ\eta : \id_\cc\to GG'F'F$ and $\bar{\epsilon}=\epsilon'\circ F'\epsilon G':F'FGG'\to\id_\e$. 
\end{proof}
Now we show how an idempotent (semi)natural transformation on the identity functor of a category allows to obtain a canonical semiadjunction of semifunctors.
\begin{prop}\label{prop:E}
	Given a category $\cc$, any idempotent (semi)natural transformation $e=(e_X)_{X\in\cc}:\id_\cc\to\id_\cc$ defines an endosemifunctor $E^e:\cc\to\cc$, which is self-semiadjoint, i.e. $E^e\dashv_\mathrm{s} E^e$. Conversely, any semifunctor which is self-semiadjoint defines an idempotent seminatural transformation.
\end{prop}
\begin{proof}
	Given the idempotent seminatural transformation $e:\id_\cc\to\id_\cc$, consider the assignment $$X\mapsto X,\quad [f:X\to Y]\mapsto f\circ e_X=e_Y\circ f,$$ for any object $X\in \cc$ and for any morphism $f$ in $\cc$. It defines a semifunctor $E^e:\cc\to\cc$. In fact, given morphisms $f:X\to Y$, $g:Y\to Z$ in $\cc$, we have that $E^e(g\circ f)=g\circ f\circ e_X=g\circ (f\circ e_X)\circ e_X=g\circ e_Y\circ f\circ e_X=E^e(g)\circ E^e(f)$ but $E^e(\id_X)=\id_X\circ e_X=e_X$, which is not necessarily $\id_X$. We show that $E^e\dashv_\mathrm{s} E^e$ is a semiadjunction with unit $\eta^e:\id_\cc\to E^eE^e$, $\eta_X^e=e_X$, and counit $\epsilon^e:E^eE^e\to\id_\cc$, $\epsilon^e_X=e_X$. Indeed, we have $E^e\epsilon^e_X\circ\eta^e_{E^eX}=E^e(e_X)\circ\eta^e_{X}=e_X\circ e_X\circ e_X=e_X=E^e\id_X$, and $\epsilon^{e}_{E^e X}\circ E^e\eta^{e}_X=e_X\circ \eta^{e}_X\circ e_X=e_X\circ e_X\circ e_X=e_X=E^e\id_X$. Conversely, if $E:\cc\to\cc$ is a self-semiadjoint semifunctor, then there exist seminatural transformations $\eta :\mathrm{Id}_{\mathcal{C}}\rightarrow EE$ and $\epsilon :EE\rightarrow \mathrm{Id}_{\mathcal{C}}$, such that $E\epsilon \circ \eta E=E\mathrm{Id}$ and $\epsilon E\circ E\eta=E\id$. As in the proof of Proposition \ref{prop:semiadj}, set $e:=\epsilon E\circ E\eta :E\rightarrow E$, which is an idempotent seminatural transformation. Indeed, it is natural as it is composition of natural transformations; for any $X\in \cc$ we have $e_X\circ E(\id_X)=\epsilon_{EX}\circ E\eta_X\circ E(\id_X)=\epsilon_{EX}\circ E(\eta_X\circ\id_X)=\epsilon_{EX}\circ E\eta_X=e_X$ and $e\circ e=E\mathrm{Id}\circ E\mathrm{Id}=E\mathrm{Id}=e$. 
\end{proof}
\begin{defn}\label{defn:E}	
We call the semifunctor $E^e$ given as in Proposition \ref{prop:E} the \textbf{canonical semifunctor} attached to an idempotent seminatural transformation $e=(e_X)_{X\in\cc}:\id_\cc\to\id_\cc$ on a category $\cc$.
\end{defn}
\begin{rmk}\label{rmk:E-id} 
	Given a category $\cc$, let $e=(e_X)_{X\in\cc}:\id_\cc\to\id_\cc$ be an idempotent (semi)natural transformation. Then, $e=\id_{\id_\cc}:\id_\cc\to\id_\cc$ if and only if $E^e:\cc\to\cc$ is the identity functor on $\cc$. 	
\end{rmk}
\begin{invisible}
\proof
If $e=\id_{\id_\cc}:\id_\cc\to\id_\cc$, we have that the semifunctor $E^e:\cc\to\cc$, given by $X\mapsto X$, $(f:X\to Y)\mapsto f\circ e_X=f\circ\id_X=f$ is actually the identity functor on $\cc$. Conversely, if $E^e$ is the identity functor on $\cc$, then $e$ results to be $\id_{\id_\cc}$. 
\endproof
\end{invisible}
As a consequence of Corollary \ref{cor:semiadj-compos} and Proposition \ref{prop:E}, given a (semi)adjunction of (semi)functors and an idempotent seminatural transformation, we can obtain another semiadjunction of semifunctors as follows.
\begin{cor}\label{cor:F'G'}
Let $F\dashv_\mathrm{s} G:\dd\to\cc$ be a semiadjunction with unit $\eta$ and counit $\epsilon$, and let $e:\id_\cc\to\id_\cc$ be an idempotent seminatural transformation. Consider the canonical semifunctor $E^e:\cc\to\cc$. Then, $F':=FE^e:\cc\to\dd$ and $G':=E^eG:\dd\to\cc$ form a semiadjunction $F'\dashv_\mathrm{s} G'$.
\end{cor}
In Example \ref{es:morph-ring} we will see an instance of a semiadjunction as in the previous corollary, constructed out of a morphism of rings.
\section{Semisplitting properties for morphisms}\label{sect:semisplit-morph}
In this section we study semisplitting properties for morphisms whose source or target is the image of a semifunctor. For semifunctors $F:\cc\to\dd$, $F':\cc'\to\dd$ and objects $C\in\cc$, $C'\in\cc'$, we introduce the notions of $F_C$-semisplit-mono, $F_C$-semisplit-epi, $(F_C,F'_{C'})$-semisplit-mono, $(F_C,F'_{C'})$-semisplit-epi, $(F_C,F'_{C'})$-semi-isomorphism. Thus, given a seminatural transformation $\alpha:F\to F'$ of semifunctors, if $\alpha$ is a natural semisplit-mono (resp. natural semisplit-epi, natural semi-isomorphism), then every component morphism $\alpha_C:FC\to F'C$ is an $(F_C,F'_C)$-semisplit-mono (resp. $(F_C,F'_C)$-semisplit-epi, $(F_C,F'_C)$-semi-isomorphism) in $\dd$.\par
\medskip
In order to provide a complete picture of these properties, we start by giving the definitions of $F_C$-semi-monomorphism and $F_C$-semi-epimorphism.
\begin{defn}
	Given a semifunctor $F:\cc\to\dd$, we say that 
	\begin{itemize}
		\item a morphism $f:FC\to D$ in $\dd$ is an \textbf{$F_C$-semi-monomorphism} if, for every parallel pair of morphisms $h,k:D'\to FC$ in $\dd$, the equality $f\circ h=f\circ k$ implies $F\id_C\circ h=F\id_C\circ k$;
		\item a morphism $f:D\to FC$ in $\dd$ is an \textbf{$F_C$-semi-epimorphism} if, for every parallel pair of morphisms $h,k:FC\to D'$ in $\dd$, the equality $h\circ f=k\circ f$ implies $h\circ F\id_C=k\circ F\id_C$.
	\end{itemize}
\end{defn}	
 
Note that if $f:FC\to D$ is a monomorphism in $\dd$, i.e. for every $h,k:D'\to FC$ in $\dd$ the equality $f\circ h=f\circ k$ implies $h=k$, then it is an $F_C$-semi-monomorphism. Analogously, if $f:D\to FC$ is an epimorphism in $\dd$, i.e. for every $h,k:FC\to D'$ in $\dd$, the equality $h\circ f=k\circ f$ implies $h=k$, then it is an $F_C$-semi-epimorphism. In case $F$ is a functor, then $f:FC\to D$ is an $F_C$-semi-monomorphism if and only if it is a monomorphism,  and $f:D\to FC$ is an $F_C$-semi-epimorphism if and only if it is an epimorphism.\medskip

We have the following properties for $F_C$-semi-monomorphisms. 

\begin{prop}\label{prop:Csemimono}
	Given semifunctors $F:\cc\to\dd$, $F':\cc'\to\dd$, we have that:
	\begin{itemize}
		\item[(1)] for every object $C$ in $\cc$, the morphism $F\id_C:FC\to FC$ in $\dd$ is an $F_C$-semi-monomorphism;
		\item[(2)] given an $F_C$-semi-monomorphism $f:FC\to D$ and an $F'_{C'}$-semi-monomorphism $g:F'C'\to FC$ such that $F\id_C\circ g=g$, then the composite $f\circ g:F'C'\to D$ is an $F'_{C'}$-semi-monomorphism;
		\item[(3)] if the composite $g\circ f$ of two morphisms $f:FC\to D$ and $g:D\to D'$ in $\dd$ is an $F_C$-semi-monomorphism, then $f$ is an $F_C$-semi-monomorphism.
	\end{itemize} 
\end{prop} 
\begin{proof}
	(1). It is obvious.\\
	(2). Let $h,k:D'\to F'C'$ be parallel arrows in $\dd$ such that $f\circ g\circ h=f\circ g\circ k$. Since $f$ is an $F_C$-semi-monomorphism, we have that $F\id_C\circ g\circ h=F\id_C\circ g\circ k$, and then from $F\id_C\circ g=g$ it follows that $g\circ h=g\circ k$. Since $g$ is an $F'_{C'}$-semi-monomorphism, we get that $F'\id_{C'}\circ h=F'\id_{C'}\circ k$, thus $f\circ g$ is an $F'_{C'}$-semi-monomorphism.\\
	(3). Let $h,k:D''\to FC$ be parallel arrows in $\dd$ such that $f\circ h= f\circ k$. Then, $g\circ f\circ h= g\circ f\circ k$. Since $g\circ f$ is an $F_C$-semi-monomorphism, we have that $F\id_{C}\circ h=F\id_{C}\circ k$, thus $f$ is an $F_C$-semi-monomorphism. 
\end{proof}
We report the analogous properties for $F_C$-semi-epimorphisms, whose proof is similar.
\begin{prop}\label{prop:Csemiepi}
	Given semifunctors $F:\cc\to\dd$, $F':\cc'\to\dd$, we have that:
	\begin{itemize}
		\item[(1)] for every object $C$ in $\cc$, the morphism $F\id_C:FC\to FC$ is an $F_C$-semi-epimorphism;
		\item[(2)] given an $F_C$-semi-epimorphism $f:D\to FC$ and an $F'_{C'}$-semi-epimorphism $g:FC\to F'C'$ such that $g\circ F\id_C=g$, then the composite $g\circ f:D\to F'C'$ is an $F'_{C'}$-semi-epimorphism;
		\item[(3)] if the composite $f\circ g$ of two morphisms $g:D'\to D$ and $f:D\to FC$ in $\dd$ is an $F_C$-semi-epimorphism, then $f$ is an $F_C$-semi-epimorphism.
	\end{itemize} 
\end{prop} 
Now, given a semifunctor $F:\cc\to\dd$, we say that
\begin{itemize}
	\item a morphism $f:FC\to D$ in $\dd$ is an \textbf{$F_C$-semisplit-mono} if there exists a morphism $g:D\to FC$ in $\dd$ such that $g\circ f=F\id_C$; 
	\item a morphism $f:D\to FC$ in $\dd$ is an \textbf{$F_C$-semisplit-epi} if there exists a morphism $g:FC\to D$ in $\dd$ such that $f\circ g=F\id_{C}$. 
\end{itemize}
If $f:FC\to D$ is an $F_C$-semisplit-mono, then the morphism $g:D\to FC$ in $\dd$ such that $g\circ f=F\id_C$ is an $F_C$-semisplit-epi. On the other hand, if $f:D\to FC$ is an $F_C$-semisplit-epi through $g:FC\to D$, then $g$ is an $F_C$-semisplit-mono. Note that in case $F$ is a functor, then $f:FC\to D$ is an $F_C$-semisplit-mono if and only if it is a split-mono, i.e. there exists a morphism $g:D\to FC$ in $\dd$ such that $g\circ f=\id_{FC}$; analogously, in case $F$ is a functor, $f$ is an $F_C$-semisplit-epi if and only if it is a split-epi, i.e. there exists a morphism $g:FC\to D$ in $\dd$ such that $f\circ g=\id_{FC}$.  
\begin{prop}\label{prop:semisect-retract}
	Given a semifunctor $F:\cc\to\dd$, we have that:
	\begin{itemize}
		\item[(1)] for every object $C$ in $\cc$, the morphism $F\id_C:FC\to FC$ is both an $F_C$-semisplit-mono and an $F_C$-semisplit-epi;
		\item[(2)] any $F_C$-semisplit-mono is an $F_C$-semi-monomorphism;
		\item[(3)] any $F_C$-semisplit-epi is an $F_C$-semi-epimorphism.	
	\end{itemize}		
\end{prop}
\begin{proof}
	(1). For every object $C$ in $\cc$ we have that $F\id_C\circ F\id_C=F\id_C$, so $F\id_C:FC\to FC$ is both an $F_C$-semisplit-mono and an $F_C$-semisplit-epi.\\
	(2). Let $f:FC\to D$ be an $F_C$-semisplit-mono in $\dd$. Then there exists a morphism $g:D\to FC$ in $\dd$ such that $g\circ f=F\id_C$. By Proposition \ref{prop:Csemimono} (1) $F\id_C$ is an $F_C$-semi-monomorphism, then so is $g\circ f$, hence by Proposition \ref{prop:Csemimono} (3) $f$ is an $F_C$-semi-monomorphism.\\
	(3). The proof is dual to (2).
	\begin{invisible}
		Consider $h,k:D'\to FC$ parallel arrows in $\dd$ such that $f\circ h=f\circ k$. Then, we have that $g\circ f\circ h=g\circ f\circ k$, hence $F\id_C\circ h=F\id_C\circ k$, thus $f$ is an $F_C$-semi-monomorphism. \\	
		(2). Let $f:D\to FC$ be an $F_C$-semisplit-epi in $\dd$. Then there exists a morphism $g:FC\to D$ in $\dd$ such that $f\circ g=F\id_C$. Consider $h,k:FC\to D'$ parallel arrows in $\dd$ such that $h\circ f=k\circ f$. Then, we have that $h\circ f\circ g=k\circ f\circ g$, hence $h\circ F\id_C=k\circ F\id_C$, thus $f$ is an $F_C$-semi-epimorphism.
	\end{invisible}
\end{proof}
In the next proposition we show the behavior of $F_C$-semisplit-monos (resp. $F_C$-semisplit-epis) with respect to composition.
\begin{prop}\label{prop:comp-semisplit}
	Given semifunctors $F:\cc\to\dd$, $F':\cc'\to\dd$, we have that:
	\begin{itemize}
	\item[(1)] if $f:FC\to D$ is an $F_C$-semisplit-mono and $f':F'C'\to FC$ is an $F'_{C'}$-semisplit-mono such that $F\id_C\circ f'=f'$, then the composite $f\circ f':F'C'\to D$ is an $F'_{C'}$-semisplit-mono;
	\item[(2)] if the composite $g\circ f$ of two morphisms $f:FC\to D$ and $g:D\to D'$ in $\dd$ is an $F_C$-semisplit-mono, then $f$ is an $F_C$-semisplit-mono;
	\item[(3)] if $f:D\to FC$ is an $F_C$-semisplit-epi and $f':FC\to F'C'$ is an $F'_{C'}$-semisplit-epi such that $f'\circ F\id_{C}=f'$, then the composite $f'\circ f:D\to F'C'$ is an $F'_{C'}$-semisplit-epi;
	\item[(4)] if the composite $f\circ g$ of two morphisms $f:D\to FC$ and $g:D'\to D$ in $\dd$ is an $F_C$-semisplit-epi, then $f$ is an $F_C$-semisplit-epi.
	\end{itemize} 		
\end{prop}
\begin{proof}
We prove only (1) and (2), as (3) and (4) follow similarly.\\
(1). If $f:FC\to D$ is an $F_C$-semisplit-mono, then there exists a morphism $g:D\to FC$ in $\dd$ such that $g\circ f=F\id_C$. Assume that $f':F'C'\to FC$ is an $F'_{C'}$-semisplit-mono, i.e. there exists a morphism $g':FC\to F'C'$ in $\dd$ such that $g'\circ f'=F'\id_{C'}$, and assume that $F\id_C\circ f'=f'$. Consider the composite $g'\circ g:D\to F'C'$. We have $g'\circ g\circ f\circ f'=g'\circ F\id_{C}\circ f'=g'\circ f'=F'\id_{C'}$, thus $f\circ f'$ is an $F'_{C'}$-semisplit-mono.\\
(2). If the composite $g\circ f:FC\to D'$ is an $F_C$-semisplit-mono, then there exists a morphism $h:D'\to FC$ in $\dd$ such that $h\circ g\circ f=F\id_C$, thus $f$ is an $F_C$-semisplit-mono.
\begin{invisible}
(3). If $f:D\to FC$ is an $F_C$-semisplit-epi, then there exists a morphism $g:FC\to D$ in $\dd$ such that $f\circ g=F\id_C$. Assume that $f':FC\to F'C'$ is an $F'_{C'}$-semisplit-epi such that $f'\circ F\id_{C}=f'$. Then, there exists a morphism $g':F'C'\to FC$ in $\dd$ such that $f'\circ g'=F'\id_{C'}$. We have $f'\circ f\circ g\circ g'=f'\circ F\id_{C}\circ g'=f'\circ g'=F'\id_{C'}$, thus $f'\circ f$ is an $F'_{C'}$-semisplit-epi.\\
(4). If the composite $f\circ g$ of two morphisms $f:D\to FC$ and $g:D'\to D$ in $\dd$ is an $F_C$-semisplit-epi, then there exists a morphism $h:FC\to D'$ in $\dd$ such that $f\circ g\circ h=F\id_C$, thus $f$ is an $F_C$-semisplit-epi.
\end{invisible}
\end{proof}
\begin{defn}\label{defn:CC'semisplit} 
	Given semifunctors $F:\cc\to\dd$, $F':\cc'\to\dd$, we say that a morphism $f:FC\to F'C'$ in $\dd$ is an 
	\begin{itemize}
		\item \textbf{$(F_C,F'_{C'})$-semisplit-mono} if $f\circ F\id_C=f$ and there exists a morphism $g:F'C'\to FC$ in $\dd$ such that $$g\circ f=F\id_C\quad \text{and}\quad g\circ F'\id_{C'}=g;$$
		\item \textbf{$(F_C,F'_{C'})$-semisplit-epi} if $F'\id_{C'}\circ f=f$ and there exists a morphism $g:F'C'\to FC$ in $\dd$ such that 
		$$f\circ g=F'\id_{C'}\quad \text{and} \quad F\id_C \circ g=g.$$
	\end{itemize}
\end{defn}
\begin{rmk}\label{rmk:CC'semisplit}
	Any $(F_C,F'_{C'})$-semisplit-mono is an $F_C$-semisplit-mono and any $(F_C,F'_{C'})$-semisplit-epi is an $F'_{C'}$-semisplit-epi.
\end{rmk}
The following properties hold true. 
\begin{prop}\label{prop:CC'semisplitmono}
	Given semifunctors $F:\cc\to\dd$, $F':\cc'\to\dd$, $F'':\cc''\to\dd$, we have that:
	\begin{itemize}
		\item[(1)] for every object $C$ in $\cc$, the morphism $F\id_C:FC\to FC$ is both an $(F_C,F_C)$-semisplit-mono and an $(F_C,F_C)$-semisplit-epi;
		\item[(2)] given an $(F_C,F'_{C'})$-semisplit-mono $f:FC\to F'C'$ and an $(F'_{C'}, F''_{C''})$-semisplit-mono $f':F'C'\to F''C''$, then the composite $f'\circ f:FC\to F''C''$ is an $(F_C,F''_{C''})$-semisplit-mono;
		\item[(3)] given an $(F_C,F'_{C'})$-semisplit-epi $f:FC\to F'C'$ and an $(F'_{C'}, F''_{C''})$-semisplit-epi $f':F'C'\to F''C''$, then the composite $f'\circ f:FC\to F''C''$ is an $(F_C,F''_{C''})$-semisplit-epi.
	\end{itemize} 
\end{prop}
\begin{proof}
	(1). It is clear.\\ 
	(2). If $f:FC\to F'C'$ is an $(F_C,F'_{C'})$-semisplit-mono, then $f\circ F\id_C=f$ and there exists a morphism $g:F'C'\to FC$ in $\dd$ such that $g\circ f=F\id_C$ and $g\circ F'\id_{C'}=g$. If $f':F'C'\to F''C''$ is an $(F'_{C'},F''_{C''})$-semisplit-mono, then $f'\circ F'\id_{C'}=f'$ and there exists a morphism $g':F''C''\to F'C'$ in $\dd$ such that $g'\circ f'=F'\id_{C'}$ and $g'\circ F''\id_{C''}=g'$. Consider the composite $g\circ g':F''C''\to FC$. We have $g\circ g'\circ f'\circ f=g\circ F'\id_{C'}\circ f=g\circ f=F\id_C$. Moreover, $f'\circ f\circ F\id_C=f'\circ f$ and $g\circ g'\circ F''\id_{C''}=g\circ g'$, thus $f'\circ f$ is an $(F_C,F''_{C''})$-semisplit-mono.\\
	(3). It is dual to $(2)$.
\end{proof}
\begin{defn}\label{defn:CC'semiiiso}
	Given semifunctors $F:\cc\to\dd$, $F':\cc'\to\dd$, we say that a morphism $f:FC\to F'C'$ in $\dd$ is an \textbf{$(F_C,F'_{C'})$-semi-isomorphism} if $f\circ F\id_C=f$ and there exists a morphism $g: F'C'\to FC$ in $\dd$ such that
	\begin{itemize}
		\item[(i)] $g\circ f=F\id_C$;
		\item[(ii)] $f\circ g=F'\id_{C'}$.
	\end{itemize} 
We call a morphism $g:F'C'\to FC$ in $\dd$ which satisfies (i) and (ii) the \textbf{$(F_C,F'_{C'})$-semi-inverse} of $f$ if $F\id_{C}\circ g=g$ holds true in addition. 
\end{defn}
In case both $F$ and $F'$ are functors, then $f:FC\to F'C'$ is an $(F_C,F'_{C'})$-semi-isomorphism if and only if it is an isomorphism in $\dd$, i.e. there exists a morphism $g: F'C'\to FC$ in $\dd$ such that $g\circ f=\id_{FC}$ and $f\circ g=\id_{F'C'}$.
\begin{lem}
	Let $F:\cc\to\dd$, $F':\cc'\to\dd$ be semifunctors and let $f: FC\to F'C'$ be an $(F_C,F'_{C'})$-semi-isomorphism in $\dd$. Then, $f=f\circ F\id_C$ is equivalent to $F'\id_{C'}\circ f=f$. Moreover, if a morphism $g:F'C'\to FC$ in $\dd$ satisfies (i) and (ii) as in Definition \ref{defn:CC'semiiiso}, then $F\id_{C}\circ g=g$ is equivalent to $g\circ F'\id_{C'}=g$.
\end{lem}
\begin{proof}
Let $f:FC\to F'C'$ be an $(F_C,F'_{C'})$-semi-isomorphism. Then, $f\circ F\id_C=f\circ g\circ f= F'\id_{C'}\circ f$, so $f=f\circ F\id_C$ is equivalent to $f= F'\id_{C'}\circ f$. Analogously, by interchanging the roles of $f$ and $g$, $F\id_{C}\circ g=g$ is equivalent to $g\circ F'\id_{C'}=g$.
\end{proof}
\begin{lem}\label{lem:semi-inverse}
Let $F:\cc\to\dd$, $F':\cc'\to\dd$ be semifunctors and let $f: FC\to F'C'$ be an $(F_C,F'_{C'})$-semi-isomorphism in $\dd$. Then, $f$ admits a unique $(F_C,F'_{C'})$-semi-inverse.
\end{lem}
\begin{proof}
If $f: FC\to F'C'$ is an $(F_C,F'_{C'})$-semi-isomorphism, then $f\circ F\id_C=f$ (which is equivalent to $F'\id_{C'}\circ f=f$) and there exists a morphism $g: F'C'\to FC$ in $\dd$ such that $g\circ f=F\id_C$ and $f\circ g=F'\id_{C'}$. Consider the composite $g':=g\circ f\circ g:F'C'\to FC$. 
Then, $g'\circ f=g\circ f\circ g\circ f=F\id_C\circ F\id_C=F\id_C$ and $f\circ g'=f\circ g\circ f\circ g=F'\id_{C'}\circ F'\id_{C'}=F'\id_{C'}$. Moreover, we have that $F\id_C\circ g'=F\id_C\circ g\circ f\circ g=g\circ f\circ g\circ f\circ g=g\circ F'\id_{C'}\circ f\circ g=g\circ f\circ g=g'$, so $g'$ is an $(F_C,F'_{C'})$-semi-inverse of $f$. Assume that there exists another $(F_C,F'_{C'})$-semi-inverse $h: F'C'\to FC$ in $\dd$ that satisfies conditions (i)-(ii) and such that $h=F\id_C\circ h$. Then, we have $h=F\id_{C}\circ h=g'\circ f\circ h=g'\circ F'\id_{C'}=g'$, thus the $(F_C,F'_{C'})$-semi-inverse of $f$ is unique.
\end{proof}
%
%
\begin{prop}\label{prop:semiiso-semisplit}
	Let $F:\cc\to\dd$, $F':\cc'\to\dd$ be semifunctors. A morphism $f:FC\to F'C'$ in $\dd$ is an $(F_C,F'_{C'})$-semi-isomorphism if and only if $f$ is both an $(F_C,F'_{C'})$-semisplit-mono and an $(F_C,F'_{C'})$-semisplit-epi.
\end{prop}
\begin{proof}
	If $f:FC\to F'C'$ is an $(F_C,F'_{C'})$-semi-isomorphism in $\dd$ with $(F_C,F'_{C'})$-semi-inverse $g'$, then it is trivially an $(F_C,F'_{C'})$-semisplit-mono and an $(F_C,F'_{C'})$-semisplit-epi. On the other hand, if $f$ is an $(F_C,F'_{C'})$-semisplit-mono, then $f\circ F\id_C=f$ and there exists a morphism $g:F'C'\to FC$ in $\dd$ such that $g\circ f=F\id_C$ and $g\circ F'\id_{C'}=g$. If $f$ is an $(F_C,F'_{C'})$-semisplit-epi, then $F'\id_{C'}\circ f=f$ and there exists a morphism $g':F'C'\to FC$ in $\dd$ such that $f\circ g'=F'\id_{C'}$ and $F\id_C \circ g'=g'$. Since $g=g\circ F'\id_{C'}=g\circ f\circ g'=F\id_C\circ g'=g'$, we have that $f$ is an $(F_C,F'_{C'})$-semi-isomorphism in $\dd$ with $(F_C,F'_{C'})$-semi-inverse $g=g'$.
\end{proof}
The following properties hold true for $(F_C,F'_{C'})$-semi-isomorphisms.
\begin{prop}\label{prop:iso}
	Given semifunctors $F:\cc\to\dd$, $F':\cc'\to\dd$, we have that:
	\begin{itemize}
		\item[(1)] for every object $C$ in $\cc$, the morphism $F\id_C:FC\to FC$ is an $(F_C,F_C)$-semi-isomorphism;
		\item[(2)] any $(F_C,F'_{C'})$-semi-isomorphism is both an $F_C$-semi-monomorphism and an $F'_{C'}$-semi-epimorphism;
		\item[(3)] if $f:FC\to F'C'$ is an $(F_C, F'_{C'})$-semisplit-mono and an $F'_{C'}$-semi-epimorphism in $\dd$, then it is an $(F_C,F'_{C'})$-semi-isomorphism;
		\item[(4)] if $f:FC\to F'C'$ is an $(F_C, F'_{C'})$-semisplit-epi and an $F_{C}$-semi-monomorphism in $\dd$, then it is an $(F_C,F'_{C'})$-semi-isomorphism.
	\end{itemize} 
\end{prop}
\begin{proof}
	(1). It is clear.\\ 
	(2). It follows from Proposition \ref{prop:semiiso-semisplit}, Remark \ref{rmk:CC'semisplit} and Proposition \ref{prop:semisect-retract}.\\
	(3). If $f:FC\to F'C'$ is an $(F_C,F'_{C'})$-semisplit-mono, then $f\circ F\id_C=f$ and there exists a morphism $g:F'C'\to FC$ in $\dd$ such that $g\circ f=F\id_C$ and $g\circ F'\id_{C'}=g$. Thus, we have $f\circ g\circ f=f\circ F\id_C=f$, hence, if $f$ is an $F'_{C'}$-semi-epimorphism, we get $f\circ g\circ F'\id_{C'}=F'\id_{C'}$, and then $f\circ g=F'\id_{C'}$, so $f$ is an $(F_C,F'_{C'})$-semi-isomorphism.\\
	(4). It is dual to (3).
\end{proof}

Moreover, we have the following. 
\begin{lem}
	Let $F:\cc\to\dd$, $F':\cc'\to\dd$ be semifunctors. Any semifunctor $H:\dd\to\e$ preserves $(F_C,F'_{C'})$-semisplit-monos, $(F_C,F'_{C'})$-semisplit-epis, $(F_C,F'_{C'})$-semi-isomorphisms.
\end{lem}
\begin{proof}
	Let $f:FC\to F'C'$ be an $(F_C,F'_{C'})$-semisplit-mono in $\dd$. Then, $f\circ F\id_C=f$ and there exists a morphism $g:F'C'\to FC$ in $\dd$ such that $g\circ f=F\id_C$ and $g\circ F'\id_{C'}=g$. We have that $Hf\circ HF\id_C=H(f\circ F\id_C)=Hf$, $Hg\circ Hf=H(g\circ f)=HF\id_{C}$ and $Hg\circ HF'\id_{C'}=H(g\circ F'\id_{C'})=Hg$, thus $Hf$ is an $(HF_C,HF'_{C'})$-semisplit-mono. If $f:FC\to F'C'$ is an $(F_C,F'_{C'})$-semisplit-epi in $\dd$, then $F'\id_{C'}\circ f=f$ and there exists a morphism $g:F'C'\to FC$ in $\dd$ such that $f\circ g=F'\id_{C'}$ and $F\id_C \circ g=g$. We have that $HF'\id_{C'}\circ Hf=Hf$, $Hf\circ Hg=H(f\circ g)=HF'\id_{C'}$ and $HF\id_C \circ Hg=Hg$, thus $Hf$ is an $(HF_C,HF'_{C'})$-semisplit-epi. If $f:FC\to F'C'$ is an $(F_C,F'_{C'})$-semi-isomorphism in $\dd$, then as in the previous cases $Hf$ is an $(HF_C,HF'_{C'})$-semi-isomorphism.
\end{proof}
Hereafter, in order to simplify the notation, when the semifunctors $F:\cc\to\dd$, $F':\cc'\to\dd$ are clear from the context, we will write $C$-semi-monomorphism (resp. $C$-semisplit-mono, $(C,C')$-semisplit-mono, $C$-semi-epimorphism, $C$-semisplit-epi, $(C,C')$-semisplit-epi, $(C,C')$-semi-isomorphism) instead of $F_C$-semi-monomorphism (resp. $F_C$-semisplit-mono, $(F_C,F'_{C'})$-semisplit-mono, $F_C$-semi-epimorphism, $F_C$-semisplit-epi, $(F_C,F'_{C'})$-semisplit-epi, $(F_C,F'_{C'})$-semi-isomorphism).
\begin{rmk}\label{rmk:seminat-split}
	Let $\alpha:F\to F'$  be a seminatural transformation of semifunctors $F, F':\cc\to\dd$. If $\alpha$ is a natural semisplit-mono (resp. natural semisplit-epi, natural semi-isomorphism), then every component morphism $\alpha_C:FC\to F'C$ is a $(C,C)$-semisplit-mono (resp. $(C,C)$-semisplit-epi, $(C,C)$-semi-isomorphism) in $\dd$.
\end{rmk} 
\begin{es}\label{es:semi-iso}
	\cite[See Section 2.4]{Ho93} Let $F:\cc\to\Set$ be a semifunctor and let $\overline{FX}=\{x \in F(X) | F(\id_X)(x)=x \}$ denote the subset of $FX$ of fixpoints of $F(\id_X)$. For any morphism $f:X\to Y$ in $\cc$, if $x\in \overline{FX}$, then $F(f)(x)\in \overline{FY}$, hence the function $F(f):FX\to FY$ restricts to a function $\overline{F(f)}:\overline{FX}\to\overline{FY}$. In fact, $F(\id_{Y})(F(f)(x))=F(\id_{Y}\circ f)(x)=F(f)(x)$, for every $x\in X$. Thus, we have a functor $\overline{F}:\cc\to\Set$, $X\mapsto\overline{FX}$, $f\mapsto \overline{F(f)}$, which is naturally semi-isomorphic to $F$. Indeed, let $\alpha:F\to \overline{F}$,  $\alpha=(\alpha_X:FX\to \overline{F}X)_{X\in \cc}$, be given by $\alpha_X(p)=F(\id_X)(p)$, for every $p\in FX$. Note that $F(\id_X)(p)\in \overline{FX}$ as $F(\id_X)(F(\id_X)(p))=F(\id_X\circ\id_X)(p)=F(\id_X)(p)$. We have that, for any morphism $f:X\to Y$ in $\cc$ and any $p\in FX$, $(\alpha_Y\circ Ff)(p)=(F\id_Y\circ Ff)(p)=Ff(p)=(Ff\circ F\id_X)(p)=(\overline{Ff}\circ\alpha_X)(p)$, thus $\alpha$ is a natural transformation. Since $\overline{F}$ is a functor, we have that $\alpha$ is seminatural. 
	Let $\beta:\overline{F}\to F$, $(\beta_X:\overline{F}X\to FX)_{X\in \cc}$, be given by the canonical inclusion $\beta_X(q)=q$, for every $q\in\overline{F}X=\overline{FX}$. We have that, for any $f:X\to Y$ in $\cc$, $(Ff\circ\beta_X)(q)=Ff(q)=(\beta_Y\circ\overline{F}f)(q)$ for every $q\in\overline{F}X$ and, since $\overline{F}$ is a functor, $\beta\circ \overline{F}\id =\beta$ holds true. Finally, $\alpha\circ\beta = \overline{F}\id$ and $\beta\circ\alpha = F\id$, as for every $X\in\cc$, $p\in FX$ and $q\in \overline{F}X$, we have $\alpha_X\beta_X(q)=\alpha_X(q)=F(\id_X)(q)=q=\overline{F}\id_X(q)$ and $\beta_X\alpha_X(p)=\beta_X (F(\id_X)(p))=F\id_X(p)$, respectively. Hence any component morphism $\alpha_X$ is an $(F_X,\overline{F}_X)$-semi-isomorphism in $\Set$, and any $\beta_X$ is an $(\overline{F}_X, F_X)$-semi-isomorphism in $\Set$. 
\end{es}

\section{The notion of semifull semifunctor}\label{sect:faithful-semifull}
Let $F: \cc \rightarrow \dd$ be a semifunctor and consider the associated natural transformation
\begin{equation}\label{nat_transf}
\f : \mathrm{Hom}_{\cc}(-,-)\rightarrow \Hom_{\dd}(F-, F-),
\end{equation}
given by $\f_{X,Y} : \mathrm{Hom}_{\cc}(X,Y)\rightarrow \Hom_{\dd}(FX, FY)$, $\f_{X,Y}(f)= F(f)$, for any morphism $f:X\rightarrow Y$ in $\cc$. Note that the codomain $\Hom_{\dd}(F-, F-)$ is a semifunctor, as $\Hom_{\dd}(F\id_X, F\id_Y)(h)=F\id_Y\circ h \circ F\id_X\neq h$ in general, while the domain $\Hom_{\cc}(-,-)$ is a functor, so $\f$ is actually a seminatural transformation. When needed we will denote $\f$ by $\f^F$ to make explicit the semifunctor $F$ we are considering.\medskip
\par 
As in the functorial case (see e.g. \cite[Definition 1.5.1]{Bor94}), if $\f_{X,Y}$ is injective, surjective, bijective for every pair of objects $X,Y\in\cc$, then $F$ is a \emph{faithful}, \emph{full}, \emph{fully faithful} semifunctor, respectively. In Proposition \ref{prop:full-semifull} we will show that a full (and hence a fully faithful) semifunctor is actually a functor. Here we investigate a weaker notion of fullness for semifunctors that we call \emph{semifullness}. We start with an example.
\begin{es}\label{es:Set-semifull-faithful}\cite[See Example 2.1]{Ho93}
	Let $\Set$ be the category of sets and functions, and consider the semifunctor $F:\Set\to\Set$, defined on objects $A$ by $F(A)=A\times A$, where $A\times A$ is the cartesian product of $A$ by itself, and on morphisms $f:A\to B$ by $F(f):A\times A\to B\times B$, $F(f)(\left( a,a'\right))=\left( f(a), f(a)\right) $, for every $a,a'\in A$. 
	In particular, if $a,a'\in A$ and $a\neq a'$, then $F(\id_A)(\left( a,a'\right))=\left( a,a\right)$, whereas $\id_{FA}(\left( a,a'\right))=\left( a,a'\right)$, hence $F$ is really a semifunctor. Note that $F$ is faithful as if $F(f)=F(f')$ for morphisms $f,f':A\to B$ in $\Set$, then for every $a\in A$ we get $\left( f(a), f(a)\right) =F(f)(\left( a,a'\right))=F(f')(\left( a,a'\right))=\left( f'(a), f'(a)\right)$, thus $f(a)=f'(a)$ for every $a\in A$, hence $f=f'$. Moreover, $F$ is not full, as there is no $f:A\to A$ in $\Set$ such that $F(f)=\id_{FA}$. Indeed, if such $f$ exists, then for all $a,a'\in A$ we have $F(f)(\left( a,a'\right))=\id_{FA}(\left( a,a'\right))=\left( a,a'\right)$, but this cannot happen if $a\neq a'$, as $F(f)(\left( a,a'\right))=(f(a),f(a))\neq (a,a')$. A deeper look at the semifunctor $F$ allows us to highlight the following property. Let $\psi_B:B\times B\to B$ be the canonical projection on the first factor of the cartesian product $B\times B$, and let $\Delta_A :A\to A\times A$ be the diagonal arrow of $A$, given by $\Delta_A(a)=\left( a,a\right)$, for every $a\in A$. For any morphism $f=\langle f_1,f_2\rangle :A\times A\to B\times B$ in $\Set$, where $f_1,f_2:A\times A\to B$, consider the morphism $$g:=\psi_B\circ f\circ\Delta_A=f_1\circ\Delta_A:A\to B$$ in $\Set$. We note that $F(g)=F\id_B\circ f\circ F\id_A$. Indeed, for all $a,a'\in A$ we have $F(g)\left( (a,a') \right)=(g(a), g(a))=(f_1((a,a)), f_1((a,a)))=F\id_B(f((a,a)))=(F\id_B\circ f\circ F\id_A)((a,a'))$. 
\end{es}
Motivated by this example, we introduce the notion of semifull semifunctor.
\begin{defn}\label{def:semifull}
	We say that a semifunctor $F:\cc\to\dd$ is \textbf{semifull} if for every morphism $f:FX\to FY$ in $\dd$ there exists a morphism $g:X\to Y$ in $\cc$ such that $F(g)=F\id_Y\circ f\circ F\id_X$.
\end{defn}
It is now clear that the semifunctor $F$ in Example \ref{es:Set-semifull-faithful} is semifull. The following result shows how semifullness is related to the traditional notion of full functor.
\begin{prop}\label{prop:full-semifull}
	Let $F:\cc\to\dd$ be a semifunctor. Then, $F$ is full if and only if it is semifull and $\id_F=F\id$.
	\end{prop}
\begin{rmk}\label{rmk:Fid}
	From $\id_F=F\id$ it follows that a full semifunctor is indeed a functor.
\end{rmk}	
\proof
If $F$ is full, then it is trivially semifull. Indeed, if for any $f:FX\to FY$ in $\dd$ there exists a morphism $g:X\to Y$ in $\cc$ such that $f=F(g)$, then $F(g)=F(\id_Y\circ g\circ \id_X)=F\id_Y\circ F(g)\circ F\id_X=F\id_Y\circ f\circ F\id_X$. In particular, for every $X\in\cc$, $\id_{FX}=F(h)$ for some $h:X\to X$ in $\cc$, hence $\id_{FX}=F(\id_X\circ h)=F\id_X\circ F(h)=F\id_X\circ \id_{FX}=F\id_X$. On the other hand, assume that $\id_F=F\id$. If $F$ is semifull, then for any morphism $f:FX\to FY$ in $\dd$ there exists $g:X\to Y$ in $\cc$ such that $F(g)=F\id_Y\circ f\circ F\id_X=\id_{FY}\circ f\circ \id_{FX}=f$, thus $F$ is full.
\endproof
Moreover, semifull semifunctors are stable under composition.
\begin{prop}\label{prop:semifull-comp}
	Let $F:\cc\to\dd$ and $G:\dd\to\e$ be semifunctors.
	\begin{itemize}
		\item[(i)] If $F$ and $G$ are semifull, then the semifunctor $G\circ F$ is semifull.
		\item[(ii)] If $G\circ F$ is semifull and $G$ is faithful, then $F$ is semifull.
	\end{itemize}
\end{prop}
\proof
	(i). Let $F$ and $G$ be semifull semifunctors. Then, for any morphism $f:GFX\rightarrow GFY$ in $\e$, since $G$ is semifull, there exists a morphism $g:FX\to FY$ in $\dd$ such that $G(g)=G\id_{FY}\circ f\circ G\id_{FX}$. Now, since $F$ is semifull, there exists a morphism $h:X\to Y$ in $\cc$ such that $F(h)=F\id_Y\circ g\circ F\id_X$. Then, we have that $GF(h)=GF\id_Y\circ G(g)\circ GF\id_X=GF\id_Y\circ G\id_{FY}\circ f\circ G\id_{FX}\circ GF\id_X=GF\id_Y\circ f\circ GF\id_X$, so $G\circ F$ is semifull.\\
	(ii). Assume that $G\circ F$ is semifull. Then, for any morphism $f:FX\to FY$ in $\dd$, there exists a morphism $h:X\to Y$ in $\cc$ such that $GF(h)=GF\id_Y\circ G(f)\circ GF\id_X$, so $G(F(h))=G(F\id_Y\circ f\circ F\id_X)$. If $G$ is faithful, we have that $F(h)=F\id_Y\circ f\circ F\id_X$, hence $F$ is semifull.\qedhere
\endproof 

We say that a semifunctor $F:\cc\to\dd$ is \textbf{semifully faithful} if $F$ is a faithful and semifull semifunctor. A fully faithful semifunctor, which is actually a functor by Remark \ref{rmk:Fid}, is in particular semifully faithful. As an instance, the semifunctor $F$ in Example \ref{es:Set-semifull-faithful} is faithful and semifull, thus semifully faithful, but not fully faithful. From Proposition \ref{prop:full-semifull} it follows that a semifunctor $F:\cc\to\dd$ is fully faithful if and only if it is semifully faithful and $\id_F=F\id$.
\medskip 

Now we see how the semifull and semifully faithful conditions can be derived from requirements on the natural transformation \eqref{nat_transf} associated with a semifunctor.
\begin{prop}\label{prop:trasfnat-semifull}
Let $F:\cc\to\dd$ be a semifunctor and consider the associated natural transformation $\f : \mathrm{Hom}_{\cc}(-,-)\rightarrow \Hom_{\dd}(F-, F-)$. 
\begin{itemize}
\item[(i)] If, for every $X, Y\in\cc$, $\f_{X,Y}$ is a $\Hom_{\dd}(F-, F-)_{(X,Y)}$-semisplit-epi (or $(X,Y)$-semisplit-epi for short), then $F$ is semifull.
\item[(ii)] If, for every $X, Y\in\cc$, $\f_{X,Y}$ is an $((X,Y), (X,Y))$-semi-isomorphism, then $F$ is semifully faithful.
\end{itemize}
\end{prop}
\begin{proof}
(i). If $\f_{X,Y}$ is an $(X,Y)$-semisplit-epi for every $X, Y\in\cc$, then there exists a map $\g_{X,Y}:\Hom_{\dd}(FX, FY)\to \Hom_{\cc}(X, Y)$ such that $\f_{X,Y}\circ \g_{X,Y}=\Hom_\dd(F\id_X, F\id_Y)$, i.e. for any morphism $g:FX\to FY$ in $\dd$, $(\f_{X,Y}\circ \g_{X,Y})(g)=F\id_Y\circ g\circ F\id_X$. Thus, for any morphism $g:FX\to FY$ in $\dd$ we have that $F(\g_{X,Y}(g))=F\id_Y\circ g\circ F\id_X$, where $\g_{X,Y}(g):X\to Y$ is a morphism in $\cc$, hence $F$ is semifull.\\
(ii). If $\f_{X,Y}$ is an $((X,Y),(X,Y))$-semi-isomorphism for every $X, Y\in\cc$, then there exists a map $\g_{X,Y}:\Hom_{\dd}(FX, FY)\to \Hom_{\cc}(X, Y)$ such that $\f_{X,Y}\circ \g_{X,Y}=\Hom_\dd(F\id_X, F\id_Y)$ and $\g_{X,Y}\circ \f_{X,Y}=\Hom_\cc(\id_X, \id_Y)$. By (i) $F$ is semifull. Moreover, for any morphism $h, k:X\to Y$ in $\cc$ such that $F(h)=F(k)$, we have that $\g_{X,Y}(F(h))=\g_{X,Y}(F(k))$, hence from $\g_{X,Y}\circ \f_{X,Y}=\Hom_\cc(\id_X, \id_Y)$ it follows that $h=k$. Thus, $F$ is semifully faithful.
\end{proof}
It is known that faithful functors reflect monomorphisms, epimorphisms and that fully faithful functors reflect isomorphisms. In the next propositions we show a similar behavior for faithful and semifully faithful semifunctors.
\begin{prop}\label{prop:refl-semimono-semiepi}
	Let $F:\cc\to\dd$ be a semifunctor and let $H:\dd\to\e$ be a faithful semifunctor. Then, $H$ reflects $C$-semi-monomorphisms and $C$-semi-epimorphisms.
\end{prop}
\begin{proof}
	Let $f:FC\to D$ be a morphism in $\dd$ and assume that $H(f)$ is a $C$-semi-monomorphism in $\e$. Suppose that $f\circ h=f\circ k$ for some pair of parallel morphisms $h,k:D'\to FC$ in $\dd$. Then, $H(f)\circ H(h)=H(f)\circ H(k)$, hence $HF\id_{C}\circ H(h)=HF\id_{C}\circ H(k)$. Since $H$ is faithful, we have that $F\id_C\circ h=F\id_C\circ k$, so $f$ is a $C$-semi-monomorphism. The case for $C$-semi-epimorphisms is similar.
\end{proof}	
\begin{prop}\label{prop:refl-semiiso}
	Let $F:\cc\to\dd$, $F':\cc'\to\dd$ be semifunctors and let $H:\dd\to\e$ be a semifully faithful semifunctor. Then, $H$ reflects $(F_C,F'_{C'})$-semisplit-monos, $(F_C,F'_{C'})$-semisplit-epis, $(F_C,F'_{C'})$-semi-isomorphisms.
\end{prop}
\begin{proof}
	Let $f:FC\to F'C'$ be a morphism in $\dd$ and assume that $H(f):HFC\to HF'C'$ is a $(HF_C,HF'_{C'})$-semisplit-mono in $\e$. Then, $H(f)\circ HF\id_C=H(f)$ and there exists a morphism $h:HF'C'\to HFC$ in $\e$ such that $h\circ H(f)=HF\id_{C}$ and $h\circ HF'\id_{C'}=h$. From $H(f\circ F\id_C)=H(f)$ we obtain $f\circ F\id_C=f$, as $H$ is faithful. Since $H$ is semifull, there is a morphism $g:F'C'\to FC$ in $\dd$ such that $H(g)=H\id_{FC}\circ h\circ H\id_{F'C'}$. Thus, from $h\circ H(f)=HF\id_{C}$ we have $H(g\circ f)=H(g)\circ H(f)=H\id_{FC} \circ h\circ H\id_{F'C'}\circ H(f)=H\id_{FC} \circ h\circ H(f)=H\id_{FC}\circ HF\id_{C}=HF\id_{C}$, and then, since $H$ is faithful, we get $g\circ f= F\id_C$. Moreover, we have $H(g\circ F'\id_{C'})=H(g)\circ HF'\id_{C'}=H\id_{FC}\circ h\circ H\id_{F'C'}\circ HF'\id_{C'}=H\id_{FC}\circ h\circ HF'\id_{C'} \circ H\id_{F'C'}=H\id_{FC}\circ h\circ H\id_{F'C'}=H(g)$, hence $g\circ F'\id_{C'}=g$ as $H$ is faithful. Then, $f$ is an $(F_C,F'_{C'})$-semisplit-mono in $\dd$. For $(F_C,F'_{C'})$-semisplit-epis and $(F_C,F'_{C'})$-semi-isomorphisms the proof is similar. \qedhere	
\end{proof}

Inspired by \cite[Proposition 2.5]{AMCM06}, we provide a characterization of faithfulness and semifullness for semifunctors that are part of a semiadjunction.
\begin{prop}\label{prop:char-faithful-semifull}
	Let $F\dashv_\mathrm{s} G:\dd\to \cc$ be a semiadjunction with unit $\eta$ and counit $\epsilon$. Then,
	\begin{itemize}
		\item[(1)] $F$ is faithful if and only if $\eta_C$ is a monomorphism in $\cc$ for every $C\in \cc$;
		\item[(2)] $F$ is semifull if and only if $\eta_C$ is a $C$-semisplit-epi in $\cc$ for every $C\in \cc$;
		\item[(3)] $G$ is faithful if and only if $\epsilon_D$ is an epimorphism in $\dd$ for every $D\in \dd$;
		\item[(4)] $G$ is semifull if and only if $\epsilon_D$ is a $D$-semisplit-mono in $\dd$ for every $D\in \dd$.
	\end{itemize}
\end{prop}
\begin{proof}
	We prove only $(1)$ and $(2)$, as $(3)$ and $(4)$ follow by duality. For any $C,C'$ in $\cc$ consider the composition
	$$\tau_{C,FC'}\circ\f^F_{C,C'}:\Hom_\cc(C,C')\to\Hom_\cc(C,GFC')$$
	where $\tau$ is the natural semi-isomorphism defined on components as in \eqref{eq:tau}, thus $\tau_{C,FC'}(Ff)=GFf\circ\eta_C=\eta_{C'}\circ f$, for any morphism $f:C\to C'$ in $\cc$.\\
	$(1)$. Assume that $F$ is faithful. Let $f,f':C\to C'$ be morphisms in $\cc$ such that $\eta_{C'}\circ f=\eta_{C'}\circ f'$, i.e. $\tau_{C,FC'}(Ff)=\tau_{C,FC'}(Ff')$. Then, by composing the latter equality with $\sigma_{C,FC'}$ defined as in \eqref{eq:sigma}, we get $\sigma_{C,FC'}\tau_{C,FC'}(Ff)=\sigma_{C,FC'}\tau_{C,FC'}(Ff')$, i.e. $Ff\circ F\id_C =Ff'\circ F\id_C$, so that $Ff=Ff'$. Since $F$ is faithful we have $f=f'$, thus $\eta_{C'}$ is a monomorphism. Conversely, suppose that $\eta_C$ is a monomorphism for every $C\in \cc$. Let $f,f': C\to C'$ be morphisms in $\cc$ such that $Ff=Ff'$. Then, $\eta_{C'}\circ f=GFf\circ\eta_C=GFf'\circ\eta_C=\eta_{C'}\circ f'$, thus $f=f'$ as $\eta_{C'}$ is a monomorphism. Hence $F$ is faithful.\\
	$(2)$. Assume that $F$ is semifull. Then, for any $f:FC\to FC'$ in $\dd$ there exists $g:C\to C'$ in $\cc$ such that $F(g)=F\id_{C'}\circ f\circ F\id_{C}$. In particular, for $\epsilon_{FC}:FGFC\to FC$, there exists $\nu_C: GFC\to C$ such that $F(\nu_C)=F\id_C\circ \epsilon_{FC}\circ F\id_{GFC}$. Then, for every $C\in \cc$, we have $\eta_C\circ\nu_C=GF\nu_C\circ\eta_{GFC}=\tau_{GFC, FC}(F\nu_C)=\tau_{GFC,FC}(F\id_C\circ \epsilon_{FC}\circ F\id_{GFC})=GF\id_C\circ G\epsilon_{FC}\circ GF\id_{GFC}\circ \eta_{GFC}=GF\id_C\circ G\epsilon_{FC}\circ\eta_{GFC}=GF\id_C\circ G\id_{FC}=GF\id_C$, thus $\eta_C$ is a $C$-semisplit-epi, for every $C$ in $\cc$. 
	Conversely, suppose that for every $C\in \cc$, $\eta_C$ is a $C$-semisplit-epi in $\cc$, i.e. there exists a morphism $\nu_C:GFC \to C$ in $\cc$ such that $\eta_C\circ \nu_C=GF\id_{C}$. Let $f:FC\to FC'$ be a morphism in $\dd$. Consider the composite morphism $\nu_{C'}\circ Gf\circ\eta_C:C\to C'$ in $\cc$. Then, we have $F(\nu_{C'}\circ Gf\circ\eta_C)=F(\id_{C'}\circ\nu_{C'}\circ Gf\circ\eta_C)=F\id_{C'}\circ F\nu_{C'}\circ FGf\circ F\eta_C=\epsilon_{FC'}\circ (F\eta_{C'}\circ F\nu_{C'})\circ FGf\circ F\eta_C=\epsilon_{FC'}\circ FGF\id_{C'}\circ FGf\circ F\eta_C=\epsilon_{FC'}\circ FG(F\id_{C'}\circ f)\circ F\eta_C=F\id_{C'}\circ f\circ \epsilon_{FC}\circ F\eta_C=F\id_{C'}\circ f\circ F\id_C$, thus $F$ is semifull.
\end{proof}	
\begin{rmk}\label{rmk:seminat-semisplit}
	Let $F:\cc\to\dd$, $G:\dd\to\cc$ be semifunctors. We observe that for any natural transformation $\alpha:\mathrm{Id}_\cc\to GF$ with domain the identity functor, which is indeed a seminatural transformation, any component morphism $\alpha_X:X\to GFX$ in $\cc$ is an $X$-semisplit-epi if and only if it is an $(X,X)$-semisplit-epi. Indeed, by Remark \ref{rmk:CC'semisplit} any $(X,X)$-semisplit-epi is an $X$-semisplit-epi. If $\alpha_X$ is an $X$-semisplit-epi, then there is $\beta_X:GFX\to X$ in $\cc$ such that $\alpha_X\circ\beta_X=GF\id_X$. Moreover,  $\id_X\circ\beta_X=\beta_X$ and from seminaturality of $\alpha$ it follows that $GF\id_X\circ\alpha_X=\alpha_X$. Analogously, for any seminatural transformation $\alpha:GF\to\mathrm{Id}_\cc$ with codomain the identity functor, any component morphism $\alpha_X:GFX\to X$ is an $X$-semisplit-mono if and only if it is an $(X,X)$-semisplit-mono. Thus, in Proposition \ref{prop:char-faithful-semifull} in the statement (2) $\eta_C$ is actually a $(C, C)$-semisplit-epi, and in the statement (4) $\epsilon_D$ is a $(D,D)$-semisplit-mono.
\end{rmk}

As a consequence, we have the following result for semifully faithful semifunctors, which is an analogue of \cite[Proposition 3.4.1]{Bor94} for semifunctors.
\begin{cor}\label{cor:semi-iso-ff}
	Let $F\dashv_\mathrm{s} G:\dd\to\cc$ be a semiadjunction with unit $\eta$ and counit $\epsilon$. Then,
	\begin{itemize}
		\item[(1)] $F$ is semifully faithful if and only if $\eta_C$ is a $(C,C)$-semi-isomorphism in $\cc$ for every $C\in \cc$;
		\item[(2)] $G$ is semifully faithful if and only if $\epsilon_D$ is a $(D,D)$-semi-isomorphism in $\dd$ for every $D\in \dd$.
	\end{itemize}
\end{cor}
\begin{proof}
	We show only $(1)$ as $(2)$ follows dually. If $F$ is semifully faithful, then by Proposition \ref{prop:char-faithful-semifull} $\eta_C$ is a monomorphism (hence a $C$-semi-monomorphism) and a $(C,C)$-semisplit-epi in $\cc$ (see Remark \ref{rmk:seminat-semisplit}) for every $C\in \cc$. Thus, by Proposition \ref{prop:iso} (4), $\eta_C$ is a $(C,C)$-semi-isomorphism. Conversely, if $\eta_C$ is a $(C,C)$-semi-isomorphism in $\cc$ for every $C\in \cc$, then by Proposition \ref{prop:semiiso-semisplit} $\eta_C$ is a $(C,C)$-semisplit-epi (hence a $C$-semisplit-epi) and a $(C,C)$-semiplit-mono (hence an ${\id_\cc}_{C}$-semi-monomorphism, i.e. a monomorphism) for every $C$ in $\cc$, thus again by Proposition \ref{prop:char-faithful-semifull} $F$ is semifull and faithful.
\end{proof}

\section{Separable semifunctors}\label{sect:separ-natsemifull}
Separable functors were introduced by C. N\v{a}st\v{a}sescu et al. in \cite{NVV89} and widely studied e.g. in \cite{CMZ02}. As a special case, which inspired the terminology, the restriction of scalars functor associated with a morphism of rings is a separable functor if and only if the corresponding ring extension is separable. In this section we study the property of separability for semifunctors. 
We recall that a functor $F: \cc \rightarrow \dd$, with associated natural transformation $\f$, is said to be \emph{separable} if there is a natural transformation $\p : \Hom_{\dd}(F-, F-)\rightarrow \Hom_{\cc}(-,-)$ such that $\p\circ\f = \id $. We define a semifunctor $F:\cc\to\dd$ to be separable by requiring the same condition on the associated natural transformation $\f$ given as in \eqref{nat_transf}. We discuss general properties, such as a Maschke-type Theorem and a Rafael-type Theorem for separable semifunctors.\\
\par
We say that a semifunctor $F:\cc\to\dd$ is \textbf{separable} if there is a natural transformation $\p : \Hom_{\dd}(F-, F-)\rightarrow \Hom_{\cc}(-,-)$ such that
\begin{equation}\label{def_P}
\p\circ\f = \Id_{\Hom_{\cc}(-,-)},
\end{equation}
i.e., for any morphism $f:X\rightarrow Y$ in $\cc$, one has $(\p_{X,Y}\circ\f_{X,Y})(f) = f$.\\
Note that $\p$ is actually a seminatural transformation. When needed we will denote $\p$ by $\p^F$.
\begin{rmk}\label{remark-separ}
\begin{itemize}
\item[(i)] A separable functor is a separable semifunctor.
\item[(ii)] A separable semifunctor is faithful.
\end{itemize}
\end{rmk}
\begin{es}\label{es:Set}
	Let $F:\Set\to\Set$ be the semifunctor considered in Example \ref{es:Set-semifull-faithful}. We define 
	\begin{equation}\label{eq:P-set}
	\begin{split}
	\p_{A,B}:\Hom_{\Set}(FA, FB)&=\Hom_{\Set}(A\times A, B\times B)\to \Hom_{\Set}(A, B)\\
	\p_{A,B}(g)&:=\psi_B\circ g\circ \Delta_A ,
	\end{split}	
\end{equation}
	for every map $g: A\times A\to B\times B$, where $\psi_B:B\times B\to B$ is the canonical projection on the first factor of the cartesian product $B\times B$, and $\Delta_A :A\to A\times A$, $\Delta_A(a)=\left( a,a\right)$, is the diagonal arrow of $A$. For any map $h:A\to B$, $g:B\times B\to C\times C$, $g(x)=\langle g_1(x),g_2(x)\rangle$, with $g_1, g_2:B\times B\to C$, and $k:C\to D$, by definition of $\p_{A,D}$, we have:\begin{displaymath}
	\begin{split}
	(\p_{A,D}(Fk\circ g\circ Fh))&(a)=(\psi_D\circ (Fk\circ g\circ Fh)\circ\Delta_A) (a)=(\psi_D\circ Fk\circ g\circ Fh) (\left( a, a\right) )\\
	&=(\psi_D\circ Fk\circ g) (\left( h(a), h(a)\right) )=(\psi_D\circ Fk) ( (g_1 ((h(a), h(a))), g_2 ((h(a),h(a)))) )\\
	&= \psi_D( (k(g_1 ((h(a), h(a)))), k(g_1 ((h(a),h(a))))) )= k( g_1 ((h(a), h(a))) )\\
	&=k( \psi_C ( (g_1 ((h(a), h(a))), g_2 ((h(a),h(a)))) ))=(k\circ \psi_C \circ g )(\left( h(a), h(a)\right) )\\
	&=( k\circ \psi_C\circ g\circ\Delta_B\circ h)(a)= (k\circ \p_{B,C} (g)\circ h )(a), 
	\end{split}
	\end{displaymath}
	for every $a\in A$, thus $\p:\Hom_{\Set}(F-, F-)\to \Hom_{\Set}(-,-)$ is a natural transformation. Moreover, for any morphism $f:A\to B$ in $\Set$ and for every $a\in A$, we have $(\p_{A,B} ( \f_{A,B} (f) ))(a)= (\p_{A,B} (Ff))(a) =(\psi_B\circ Ff\circ\Delta_A ) (a)=(\psi_B\circ Ff) (\left( a,a\right) )=\psi_B (\left( f(a),f(a)\right) )=f(a)$, hence $F$ results to be a separable semifunctor. 
\end{es}
Now we study some elementary properties of separable semifunctors. Their behavior with respect to composition is the same as in the functorial case, see e.g. \cite[Lemma 1.1]{NVV89}. 
\begin{lem}\label{lem:comp-sep}
Let $F:\cc\to\dd$ and $G:\dd\to\e$ be semifunctors.
\begin{itemize}
\item[(i)] If $F:\cc\to\dd$ and $G:\dd\to\e$ are separable, then so is the composite $G\circ F:\cc\to\e$.
\item[(ii)] If $G\circ F:\cc\to\e$ is separable, then so is $F$.
\end{itemize}
\end{lem}
 \begin{proof}
For (i) define $\p^{GF}_{X,Y}(g)=\p^F_{X,Y}\p^G_{FX,FY}(g)$, for any morphism $g$ in $\Hom_\e(GFX,GFY)$. For (ii) define $\p^F_{X,Y}(f)=\p^{GF}_{X,Y}(Gf)$, for every $f\in\Hom_\dd(FX,FY)$.	
\end{proof}
\begin{prop}\label{prop:sep-semiiso}
	A semifunctor naturally semi-isomorphic to a separable semifunctor is separable.
\end{prop}
\begin{proof}
Let $\alpha: F\to G$ be a natural semi-isomorphism of semifunctors, where $G:\cc\to\dd$ is separable with respect to $\p^G$. Consider $\varsigma_{X,Y}:\Hom_\dd (FX,FY)\to \Hom_\dd (GX,GY)$ defined by $\varsigma_{X,Y}(f)=\alpha_Y\circ f\circ \alpha_X^{-1}$. Then, $F$ results to be separable with respect to $\p^F_{X,Y}:=\p^G_{X,Y}\circ\varsigma_{X,Y}$. In fact, for any morphism $f:X\to Y$ in $\cc$ we have \begin{gather*}
(\p^F_{X,Y}\circ\f^F_{X,Y})(f)=\p^F_{X,Y}(Ff)=\p^G_{X,Y}(\varsigma_{X,Y}(Ff))=\p^G_{X,Y}(\alpha_Y\circ Ff\circ \alpha_X^{-1})\\ =\p^G_{X,Y}(Gf\circ\alpha_X \circ \alpha_X^{-1})=\p^G_{X,Y}(Gf\circ G\id_X)=\p^G_{X,Y}(G(f\circ\id_X))=\p^G_{X,Y}(Gf)= f. \qedhere
\end{gather*}
\end{proof}
Since separable functors satisfy a functorial version of Maschke Theorem \cite[Proposition 1.2]{NVV89}, see also \cite[Proposition 47 and Corollary 5]{CMZ02}, we wonder if a similar behavior holds for separable semifunctors. The Maschke Theorem for separable functors asserts that, given a separable functor $F:\cc\to\dd$ and a morphism $f:C\to C'$ in $\cc$, if $F(f)$ is a split-mono (resp. split-epi) in $\dd$, then so is $f$. In the next we prove a similar result for separable semifunctors. 
\begin{thm}(Maschke-type Theorem)\label{thm:maschke}
Let $F:\cc\to\dd$ be a separable semifunctor. For any morphism $f:C\to C'$ in $\cc$, consider the morphism $F(f):FC\to FC'$ in $\dd$.
\begin{itemize}
\item[(1)] If $F(f)$ is a $C$-semisplit-mono, then $f$ is a split-mono;
\item[(2)] if $F(f)$ is a $C'$-semisplit-epi, then $f$ is a split-epi;	
\item[(3)] if $F(f)$ is a $(C,C')$-semi-isomorphism, then $f$ is an isomorphism.	
\end{itemize}	
\end{thm} 
\begin{proof}
We show only (1), as (2) is analogous and (3) follows from $(1)+(2)$. Assume that $F(f)$ is a $C$-semisplit-mono, i.e. there exists a morphism $g: FC'\to FC$ such that $g\circ F(f)= F\id_{C}$, and that $F$ is separable through a natural transformation $\p$. Then, by naturality of $\p$, we get that $\p_{C',C}(g)\circ f=\p_{C,C}(g\circ F(f))=\p_{C,C}(F\id_{C})=\p_{C,C}\f_{C,C}(\id_C)=\id_C$, hence $f$ is a split-mono.\begin{invisible}
(2): Assume that $F(f)$ is a $F$-semisplit-epi, i.e. there exists a morphism $g: FC'\to FC$ such that $F(f)\circ g= F\id_{C'}$, and that $F$ is separable through the natural transformation $\p$. Then, by naturality of $\p$, we get that $f\circ \p_{C',C}(g)=\p_{C',C'}(F(f)\circ g)=\p_{C',C'}(F\id_{C'})=\p_{C',C'}\f_{C',C'}(\id_{C'})=\id_{C'}$, hence $f$ is a split-mono. (3) follows from $(1)+(2)$ as any $(F,F)$-semi-isomorphism is a $F$-semisplit-mono and a $F$-semisplit-epi morphism. The inverse of $f$ is $\p_{C',C}(g)$.
\end{invisible}
\end{proof}
\begin{invisible}
The following property for separable \cite[Proposition 48, page 93]{CMZ02} functors holds also for separable semifunctors. Recall that an object $P$ in a category $\cc$ is \emph{projective} if for any epimorphism $q: Q\to X$ and any morphism $f:P\to X$, there is a morphism $u:P\to Q$ such that $q\circ u=f$. \emph{Injective} objects can be defined dually.
\begin{prop}
Let $F:\cc\to\dd$ be a separable semifunctor. 
\begin{itemize}
	\item[(1)] If $F$ preserves epimorphisms, then $F$ reflects projective objects.
	\item[(2)] If $F$ preserves monomorphisms, then $F$ reflects injective objects.	
\end{itemize}
\end{prop}
Obeserve that if $F$ preserves epimorphisms (or monomorphisms), then $F(\id)$ is an epimorphism (monomorphism) and since $F(\id)\circ F(\id)=F(\id\circ\id)=F\id=\id_F\circ F\id$ ($F(\id)\circ F(\id)=F(\id\circ\id)=F\id=F\id\circ \id_F$), we get that $F(\id)=\id_F$, thus $F$ is a functor.
\begin{proof}
(1) Assume that $M$ is an object in $\cc$ such that $F(M)$ is projective. Let $g:C\to C'$ and $f:M\to C'$ be an epimorphism and a morphism in $\cc$, respectively. Then, by assumption $F(g):FC\to FC'$ is an epimorphism, hence there exists a morphism $u:FM\to FC$ such that $F(f)=F(g)\circ u$. Thus, if $F$ is separable through $\p$, then we have $f=\p (F(f))=\p(F(g)\circ u)=g\circ \p(u)$, so $M$ is projective. (2) follows similarly.
\end{proof}
\end{invisible}
\subsection{Separable semiadjoints}
In \cite[Theorem 1.2]{Raf90} a characterization of separability for functors which have an adjoint is provided and it is known as Rafael Theorem. Namely, given an adjunction $F\dashv G:\dd\to\cc$ with unit $\eta$ and counit $\epsilon$, then $F$ is separable if and only if there exists a natural transformation $\nu:GF\to \id _\cc$ such that $\nu\circ\eta =\id_{\id_\cc}$, while $G$ is separable if and only if there exists a natural transformation $\gamma:\id_\dd\to FG$ such that $\epsilon\circ\gamma=\id_{\id_\dd}$. Here we extend Rafael Theorem to semifunctors, in order to obtain a characterization of separability for semifunctors which are part of a semiadjunction.  
\begin{rmk}\label{rmk:seminat-alpha}
	Let $F:\cc\to\dd$, $G:\dd\to\cc$ be semifunctors. We observe that for any natural transformation $\alpha:GF\to\mathrm{Id}_\cc$ with codomain the identity functor (it is indeed a seminatural transformation), 
	we have $\alpha\circ G\id_F=\alpha\circ GF\id\circ G\id_F=\alpha\circ G(F\id\circ \id_F)=\alpha\circ GF\id=\alpha$. Analogously, for any (semi)natural transformation $\alpha:\mathrm{Id}_\cc\to GF$ with domain the identity functor, we have $G\id_F\circ\alpha=G\id_F\circ GF\id\circ \alpha=G(\id_F\circ F\id)\circ \alpha= GF\id\circ\alpha=\alpha$.
\end{rmk}	
\begin{thm}(Rafael-type Theorem for separable semifunctors)\label{Th.Rafael}
	Let $F\dashv_\mathrm{s} G:\dd\to \cc$ be a semiadjunction, with unit $\eta$ and counit $\epsilon$. Then,
	\begin{itemize}
		\item[(1)] $F$ is separable if and only if $\eta$ is a natural split-mono, i.e. there exists a seminatural transformation $\nu : GF\rightarrow \Id_{\cc}$ such that
		$\nu\circ\eta = \Id_{\id_{\cc}};$
		\item[(2)] $G$ is separable if and only if $\epsilon$ is a natural split-epi, i.e. there exists a seminatural transformation $\gamma : \Id_{\dd}\rightarrow FG$ such that $\epsilon\circ\gamma = \Id_{\id_{\dd}}.$
	\end{itemize}
\end{thm}
\proof
We prove only (1) as (2) follows by duality. Assume that $F$ is a separable semifunctor and let $\p$ be the associated natural transformation such that $\p\circ\f =\Id_{\Hom_{\cc}(-,-)}$. We define a seminatural transformation $\nu : GF\rightarrow \Id_{\cc}$ given for any object $X$ in $\cc$ by
\begin{equation}\label{def_nu}
\nu_{X}:=\p_{GFX,X}(\epsilon_{FX}): GFX\rightarrow X.
\end{equation}
It is indeed natural as for any morphism $f:X\to Y$ in $\cc$, by naturality of $\p$, we have $f\circ\nu_X=f\circ\p_{GFX,X}(\epsilon_{FX})=\p_{GFX,Y}(Ff\circ\epsilon_{FX})=\p_{GFX,Y}(\epsilon_{FY}\circ FGFf)=\p_{GFY,Y}(\epsilon_{FY})\circ GFf=\nu_Y\circ GFf$. It holds
\begin{equation*}
\begin{split}
\nu_X\circ\eta_X &= \p_{GFX,X}(\epsilon_{FX})\circ\eta_{X}=\p_{X,X}(\epsilon_{FX}\circ F\eta_{X})\\
&=\p_{X,X}(F\Id_{X})=\p_{X,X}\f_{X,X}(\Id_{X})=\id_X,
\end{split}
\end{equation*}
where the last equality follows from the separability of $F$. Moreover, for every $g:FX\to FY$ in $\dd$, we have 
\begin{equation*}
\begin{split}
\nu_Y\circ Gg\circ\eta_X &=\p_{GFY,Y}(\epsilon_{FY})\circ Gg\circ\eta_X=\p_{X,Y}(\epsilon_{FY}\circ FGg\circ F\eta_X)\\&=\p_{X,Y}(g\circ\epsilon_{FX}\circ F\eta_X)=\p_{X,Y}(g\circ F\id_X)=\p_{X,Y}(g)\circ\id_X=\p_{X,Y}(g).
\end{split}
\end{equation*} 
Conversely, suppose that there exists a seminatural transformation $\nu :GF\rightarrow \Id_{\cc}$ such that $\nu\circ\eta =\Id_{\id_{\cc}}$. Define for every $g:FX\to FY$ in $\dd$
\begin{equation}
\p_{X,Y}(g):=\nu_Y\circ Gg\circ\eta_X .
\end{equation}
From the naturality of $\eta$ and $\nu$, for any $h:X\to Y$ in $\cc$, $k:FY\to FZ$ in $\dd$, $l:Z\to T$ in $\cc$, we have $\p_{X,T}(Fl\circ k\circ Fh)=\nu_T\circ G(Fl\circ k\circ Fh)\circ\eta_X=(\nu_T\circ GFl )\circ Gk\circ (GFh\circ\eta_X)=l\circ (\nu_Z\circ Gk\circ \eta_{Y})\circ h=l\circ\p_{Y,Z}(k)\circ h$, thus $\p : \Hom_{\dd}(F-, F-)\rightarrow \Hom_{\cc}(-,-)$ is a natural transformation. 
It holds $\p_{GFX,X}(\epsilon_{FX})=\nu_X\circ G\epsilon_{FX}\circ\eta_{GFX}=\nu_X\circ G\id_{FX}=\nu_X$, where the last equality follows from Remark \ref{rmk:seminat-alpha}, hence the correspondence between $\p$ and $\nu$ is bijective. Moreover, for every $f\in \Hom_{\cc}(X,Y)$ we have
\begin{equation*}
\begin{split}
(\p_{X,Y}\circ\f_{X,Y}) (f)&=\p_{X,Y}(F(f))=\nu_Y\circ GF(f)\circ\eta_X=\nu_Y\circ\eta_{Y}\circ f=\id_Y\circ f=f,
\end{split}
\end{equation*}
that is, $F$ is separable.
\endproof
Given an idempotent (semi)natural transformation $e=(e_X)_{X\in\cc}:\id_\cc\to\id_\cc$ on a category $\cc$, consider the canonical semifunctor $E^e:\cc\to\cc$.
As a consequence of Theorem \ref{Th.Rafael} we have the following.
\begin{prop}\label{prop:E^e}
	Let $e=(e_X)_{X\in\cc}:\id_\cc\to\id_\cc$ be an idempotent (semi)natural transformation. Then, the canonical semifunctor $E^e:\cc\to\cc$ is separable if and only if $e_X=\id_X$, for every $X\in\cc$.
\end{prop}
\begin{proof}
From the proof of Proposition \ref{prop:E} $E^e$ is self-semiadjoint with unit and counit given on components by $e_X:X\to X$, for any $X\in\cc$. By Theorem \ref{Th.Rafael} $E^e$ is separable if and only if there exists a seminatural transformation $\nu=(\nu_X : X\rightarrow X)_{X\in\cc}$ such that $\nu_X\circ e_X= \Id_X$ for any $X\in\cc$. 
Then, $e_X=\id_X\circ e_X=\nu_X\circ e_X\circ e_X=\nu_X\circ e_X=\id_X$.
\end{proof}
\begin{rmk}\label{rmk:F'G'sep}
		By Corollary \ref{cor:F'G'} we know that, given a semiadjunction $F\dashv_\mathrm{s} G:\dd\to\cc$ and the canonical semifunctor $E^e:\cc\to\cc$, then $F':=FE^e:\cc\to\dd$ and $G':=E^eG:\dd\to\cc$ form a semiadjunction $F'\dashv_\mathrm{s} G'$. If $F'$ is separable, then by Lemma \ref{lem:comp-sep} (ii) $E^e$ is separable, hence by Proposition \ref{prop:E^e} $F'=F$, so $F$ is separable. If $G'$ is separable, then so is $G$ again by Lemma \ref{lem:comp-sep} (ii).
\end{rmk}

\section{The notion of natural semifullness}\label{sect:natsemifull}
As separable functors are naturally faithful, in a somehow dual way \emph{naturally full} functors have been defined in \cite{AMCM06}. Explicitly, a functor $F: \cc \rightarrow \dd$ with associated natural transformation $\f: \Hom_{\cc}(-, -)\rightarrow \Hom_{\dd}(F-,F-)$, $f\mapsto F(f)$, is called \emph{naturally full} \cite[Definition 2.1]{AMCM06} if there exists a natural transformation $\p : \Hom_{\dd}(F-, F-)\rightarrow \Hom_{\cc}(-,-)$ such that $\f\circ\p = \id$. Since a full semifunctor is actually a functor, if we require the same natural fullness condition (which implies fullness) on the natural transformation \eqref{nat_transf} associated with a semifunctor, that we then call \emph{naturally full}, we retrieve the notion of naturally full functor. By slightly modifying the previous condition, in this section we study a proper notion of natural semifullness.\\
\par
Let $F:\cc\to\dd$ be a semifunctor and consider its associated natural transformation $\f$.\\ We say that $F$ is a \textbf{naturally semifull} semifunctor if there is a natural transformation $\p : \Hom_{\dd}(F-, F-)\rightarrow \Hom_{\cc}(-,-)$ such that for every object $X$, $Y$ in $\cc$,
\begin{equation}
\f_{X,Y}\circ\p_{X,Y}=\Hom_\dd(F\id_X,F\id_Y)
\end{equation}
i.e., for any morphism $f:FX\rightarrow FY$ in $\dd$, one has $(\f_{X,Y}\circ\p_{X,Y})(f) = F\id_Y \circ f \circ F\id_X$.
\begin{rmk}
If $F:\cc\to\dd$ is a functor, then we recover the definition of naturally full functor.
\end{rmk}
\begin{lem}\label{lem:semifull}
	Let $F:\cc\to\dd$ be a semifunctor. If $F$ is naturally semifull, then it is semifull. 
	\end{lem}
\begin{proof}
	If $F$ is naturally semifull, then for any morphism $f:FX\to FY$ in $\dd$ there exists a morphism $\p_{X,Y}(f):X\to Y$ in $\cc$ such that $F(\p_{X,Y}(f))=F\id_Y\circ f\circ F\id_X$, hence $F$ is semifull.
\end{proof}
Similarly to Proposition \ref{prop:full-semifull}, the next result shows how the notions of naturally semifull semifunctor and naturally full functor are related.
\begin{prop}\label{prop:full-natfull} Let $F:\cc\to\dd$ be a semifunctor. Then, $F$ is naturally full if and only if $F$ is naturally semifull and $\id_F=F\id$.
\end{prop}
\proof
If $F$ is naturally full, then it is trivially full and by Proposition \ref{prop:full-semifull} it holds $\id_F=F\id$. Since $F\id_Y\circ f\circ F\id_X=\id_{FY}\circ f\circ\id_{FX}=f=\f_{X,Y}\p_{X,Y}(f)$,
	$F$ is also naturally semifull. Conversely, if $F$ is naturally semifull, then it is semifull by Lemma \ref{lem:semifull}. Thus, if $\id_F=F\id$, by Proposition \ref{prop:full-semifull} $F$ is full, so for any $f:FX\to FY$ in $\dd$ there exists a morphism $g:X\to Y$ in $\cc$ such that $f=F(g)$. Since $F$ is naturally semifull, we have $\f_{X,Y}(\p_{X,Y}(f))=F\id_Y\circ f\circ F\id_X=F\id_Y\circ Fg\circ F\id_X= Fg=f$, hence $F$ is naturally full.
\endproof
\begin{es}\label{es:Set-nat.full}
We come back to Example \ref{es:Set}, where the semifunctor $F:\Set\to\Set$, $A\mapsto F(A)=A\times A$, $[f:A\to B]\mapsto F(f):A\times A\to B\times B$, $F(f)(\left( a,a'\right))=\left( f(a), f(a)\right)$ is shown to be separable with respect to $\p_{A,B}:\Hom_{\Set}(FA, FB)=\Hom_{\Set}(A\times A, B\times B)\to \Hom_{\Set}(A, B)$, $\p_{A,B}(g)=\psi_B\circ g\circ \Delta_A $. Note that for any map $f=\langle f_1,f_2\rangle :A\times A\to B\times B$, with $f_1,f_2:A\times A\to B$, and for every $(a,a')\in A\times A$, we have \begin{displaymath}
\begin{split}
(\f_{A,B}(\p_{A,B}(f)))((a,a'))&=F(\p_{A,B} (f))((a,a'))=F(\psi_B\circ f\circ\Delta_A)((a,a'))\\
&=(F(\psi_B)\circ F(f)\circ F(\Delta_A))((a,a'))=F(\psi_B)( F(f)( (a,a),(a,a)))\\
&=F(\psi_B)(f((a,a)),f((a,a)))=(\psi_B f((a,a)),\psi_Bf((a,a)))\\
&=(f_1((a,a)), f_1((a,a)))=F(\id_B)((f_1((a,a)), f_2((a,a))))\\
&=F(\id_B)(f((a,a)))=(F\id_B\circ f)((a,a))=(F\id_B\circ f\circ F\id_{A})((a,a')),
\end{split}
\end{displaymath}
hence $F$ is also naturally semifull with respect to such $\p$.
\end{es}

In the next proposition we describe the behavior of naturally semifull semifunctors with respect to composition, cf. \cite[Proposition 2.3]{AMCM06} for the naturally full functor case. 
\begin{prop}\label{prop:comp}
	Let $F:\cc\to\dd$ and $G:\dd\to\e$ be semifunctors.
	\begin{itemize}
		\item[(i)] If $F$ and $G$ are naturally semifull, then the semifunctor $G\circ F$ is naturally semifull.
		\item[(ii)] If $G\circ F$ is naturally semifull and $G$ is faithful, then $F$ is naturally semifull.
	\end{itemize}
\end{prop}
\proof
(i). Let $F$ and $G$ be naturally semifull semifunctors with respect to $\p^F$ and $\p^G$, respectively. Then, $G\circ F$ is naturally semifull with respect to $\p^{G F}_{X,Y}:=\p^F_{X,Y}\circ\p^G_{FX,FY}$. Indeed, for any $g:GFX\rightarrow GFY$ in $\e$, we have
	\begin{displaymath}
	\begin{split}
	(\f_{X,Y}^{GF}\circ\p^{G F}_{X,Y})(g)&=(\f_{FX,FY}^G\circ\f_{X,Y}^F\circ\p_{X,Y}^F\circ\p_{FX,FY}^G)(g) 
	= \f_{FX,FY}^G(\f_{X,Y}^F(\p_{X,Y}^F(\p_{FX,FY}^G(g))))\\
	&=\f_{FX,FY}^G(F\id_Y\circ \p^G_{FX,FY}(g)\circ F\id_X)
	=GF\id_Y\circ G\p^G_{FX,FY}(g)\circ GF\id_X \\
	&=GF\id_Y\circ (G\id_{FY}\circ g\circ G\id_{FX})\circ GF\id_X \\&
	=G(F\id_Y\circ \id_{FY})\circ g\circ G(\id_{FX}\circ F\id_X)
	=GF\id_Y\circ g\circ GF\id_X . 
	\end{split}
	\end{displaymath}
(ii). Assume that $G\circ F$ is naturally semifull with respect to $\p^{G F}$. Then, for any $f:FX\to FY$ in $\dd$, we have $\f_{FX,FY}^G(\f_{X,Y}^F(\p^{G F}_{X,Y}(\f_{FX,FY}^G(f))))= GF\p^{G F}_{X,Y}(Gf)= GF\id_Y\circ Gf\circ GF\id_X=G(F\id_Y\circ f\circ F\id_X)=\f_{FX,FY}^G(F\id_Y\circ f\circ F\id_X)$, and if $G$ is faithful, it follows that $(\f_{X,Y}^F\circ\p^{G F}_{X,Y}\circ\f_{FX,FY}^G) (f)=F\id_Y\circ f\circ F\id_X$, thus $F$ is naturally semifull with respect to $\p^F:=\p^{G F}\circ \f^G$.\qedhere
\endproof
\begin{rmk}
	If $F:\cc\to\dd$ is a full functor which is not naturally full (see e.g. \cite[Example 3.3]{AMCM06}) and $G:\dd\to\e$ is a semifully faithful semifunctor, then the composite $G\circ F:\cc\to \e$ is a semifull semifunctor which is not naturally semifull. In fact, by Proposition \ref{prop:semifull-comp} (i) $G\circ F$ is semifull. If $G\circ F$ were naturally semifull, then by Proposition \ref{prop:comp} (ii) it would follow that $F$ is a naturally semifull functor, i.e. a naturally full functor, and this contradicts our assumption.
\end{rmk}
Similarly to Proposition \ref{prop:sep-semiiso}, we have the following.
\begin{prop}\label{prop:natsemif-semiiso}
	A semifunctor naturally semi-isomorphic to a naturally semifull semifunctor is naturally semifull.
\end{prop}
\begin{proof}
	Consider a natural semi-isomorphism $\alpha: F\to G$ of semifunctors, where the semifunctor $G:\cc\to\dd$ is naturally semifull with respect to $\p^G$, and $\varsigma_{X,Y}:\Hom_\dd (FX,FY)\to \Hom_\dd (GX,GY)$ defined by $\varsigma_{X,Y}(f)=\alpha_Y\circ f\circ \alpha_X^{-1}$. We show that $F$ results to be naturally semifull with respect to $\p^F_{X,Y}:=\p^G_{X,Y}\circ\varsigma_{X,Y}$. Recall that $\alpha^{-1}_X\circ\alpha_X=F\id_X$. Then, for any $f:FX\to FY$ in $\dd$, we have 
	\begin{displaymath}
	\begin{split}
	(\f^F_{X,Y}\circ\p^F_{X,Y})(f)&=\f^F_{X,Y}(\p^G_{X,Y}(\varsigma_{X,Y}(f)))=F(\p^G_{X,Y}(\alpha_Y\circ f\circ \alpha_X^{-1}))\\&=F(\p^G_{X,Y}(\alpha_Y\circ f\circ \alpha_X^{-1})\circ\id_X)=F(\p^G_{X,Y}(\alpha_Y\circ f\circ \alpha_X^{-1}))\circ F\id_X \\&=F(\p^G_{X,Y}(\alpha_Y\circ f\circ \alpha_X^{-1}))\circ\alpha_X^{-1}\circ\alpha_X=\alpha^{-1}_Y\circ G(\p^G_{X,Y}(\alpha_Y\circ f\circ \alpha_X^{-1}))\circ\alpha_X\\&=\alpha^{-1}_Y\circ (G\id_Y\circ\alpha_Y\circ f\circ\alpha_X^{-1}\circ G\id_X)\circ\alpha_X=\alpha_Y^{-1}\circ\alpha_Y\circ f\circ\alpha^{-1}_X\circ\alpha_X\\&=F\id_Y\circ f\circ F\id_X. \qedhere
	\end{split}
\end{displaymath}
\end{proof}
It is known that a functor is fully faithful if and only if it is separable and naturally full, see \cite[Remark 2.2 (3)]{AMCM06}. A similar characterization in terms of separable and naturally semifull semifunctors holds for semifully faithful semifunctors.
\begin{prop}\label{prop:char-sff}
	Let $F:\cc\to\dd$ be a semifunctor. Then, $F$ is semifully faithful if and only if it is separable and naturally semifull.
\end{prop}
\proof
If $F$ is separable and naturally semifull, then it is trivially semifully faithful by Remark \ref{remark-separ} (ii) and Lemma \ref{lem:semifull}. Conversely, assume that $F$ is semifully faithful. Since $F$ is semifull, for any morphism $f:FX\to FY$ in $\dd$ there exists a morphism $g:X\to Y$ in $\cc$ such that $F(g)=F\id_Y\circ f\circ F\id_X$, and by faithfulness of $F$, $g$ is unique. This assignment defines a mapping \begin{equation*}
	\p : \Hom_{\dd}(F-, F-)\rightarrow \Hom_{\cc}(-,-)
\end{equation*}
such that for any $f:FX\rightarrow FY$ in $\dd$, with $X$, $Y$ in $\cc$, $\p_{X,Y}(f)=g$, where $g:X\to Y$ in $\cc$ is such that $F(g)=F\id_Y\circ f\circ F\id_X$. We show that such $\p$ is actually a natural transformation. For any $h:X\to Y$ in $\cc$, $k:FY\to FZ$ in $\dd$, $l:Z\to T$ in $\cc$, we have that there is a morphism $g:X\to T$ in $\cc$ such that $F(g)=F\id_T\circ Fl\circ k\circ Fh\circ F\id_X=Fl\circ k\circ Fh$ and $\p_{X,T}(Fl\circ k\circ Fh)=g$. Then, we get $\f_{X,T}(\p_{X,T}(Fl\circ k\circ Fh))=F(g)= Fl\circ k\circ Fh=Fl\circ (F\id_Z\circ k\circ F\id_Y) \circ Fh=Fl\circ F(\p_{Y,Z}(k))\circ Fh=\f_{X,T}(l\circ \p_{Y,Z}(k)\circ h)$, hence since $F$ is faithful it follows that $\p_{X,T}(Fl\circ k\circ Fh)=l\circ \p_{Y,Z}(k)\circ h$ and so $\p$ is a natural transformation. Since for any $f:FX\rightarrow FY$ in $\dd$ there is a morphism $g:X\to Y$ in $\cc$ such that $F(g)=F\id_Y\circ f\circ F\id_X$, we have that $F(\p_{X,Y}(f))=F(g)=F\id_Y\circ f\circ F\id_X$, thus $F$ is naturally semifull. Moreover, for any $f:X\to Y$ in $\cc$, $\p_{X,Y}(F(f))=h$, for some $h:X\to Y$ in $\cc$ such that $F(h)=F\id_Y\circ F(f)\circ F\id_X=F(f)$, but since $F$ is faithful we achieve $h=f$, and hence $F$ is separable as $\p_{X,Y}(F(f))=f$.
\endproof

\subsection{Naturally semifull semiadjoints} The next result is a Rafael-type Theorem for naturally semifull semifunctors. In the functor case \cite[Theorem 2.6]{AMCM06}, given an adjunction $F\dashv G:\dd\to\cc$ with unit $\eta$ and counit $\epsilon$, $F$ is naturally full if and only if there exists a natural transformation $\nu:GF\to \id _\cc$ such that $\eta\circ\nu =\id_{GF}$, while $G$ is naturally full if and only if there exists a natural transformation $\gamma:\id_\dd\to FG$ such that $\gamma\circ\epsilon=\id_{FG}$. A similar characterization of natural semifullness can be given for semifunctors that are part of a semiadjunction in terms of semisplitting properties for the unit and the counit.

\begin{thm}\label{thm:Raf-nat.full}
Let $F\dashv_\mathrm{s} G:\dd\to \cc$ be a semiadjunction with unit $\eta$ and counit $\epsilon$. Then,
\begin{itemize}
\item[(1)] $F$ is naturally semifull if and only if $\eta$ is a natural semisplit-epi, i.e. there exists a seminatural transformation $\nu : GF\rightarrow \Id_{\cc}$ such that $\eta\circ\nu = GF\Id $; 
\item[(2)] $G$ is naturally semifull if and only if $\epsilon$ is a natural semisplit-mono, i.e. there exists a seminatural transformation $\gamma : \Id_{\dd}\rightarrow FG$ such that $\gamma\circ\epsilon = FG\Id$.
\end{itemize}
\end{thm}
\proof
We prove only (1) as (2) follows by duality. Assume that $F$ is a naturally semifull semifunctor and let $\p : \Hom_{\dd}(F-, F-)\rightarrow \Hom_{\cc}(-,-)$ be the associated natural transformation such that for any $f:FX\rightarrow FY$ in $\dd$, $$(\f_{X,Y}\circ\p_{X,Y})(f) = F\id_Y\circ f \circ F\id_X .$$ We define a seminatural transformation $\nu : GF\rightarrow \Id_{\cc}$ given for any object $X$ in $\cc$ by $\nu_{X}:=\p_{GFX,X}(\epsilon_{FX}): GFX\rightarrow X$. It is indeed natural as for any morphism $f:X\to Y$ in $\cc$, we have $f\circ\nu_X=f\circ\p_{GFX,X}(\epsilon_{FX})=\p_{GFX,Y}(Ff\circ\epsilon_{FX})=\p_{GFX,Y}(\epsilon_{FY}\circ FGFf)=\p_{GFY,Y}(\epsilon_{FY})\circ GFf=\nu_Y\circ GFf$. Then, for every $X\in \cc$, by naturality of $\eta$, we get
\begin{equation*}
\begin{split}
\eta_X\circ\nu_X &=GF\nu_X\circ\eta_{GFX}=GF\p_{GFX,X}( \epsilon_{FX})\circ \eta_{GFX}=G(F\id_X\circ\epsilon_{FX}\circ F\id_{GFX})\circ\eta_{GFX}\\
&=GF\id_X\circ G\epsilon_{FX}\circ GF\id_{GFX}\circ\eta_{GFX}=GF\id_X\circ G\epsilon_{FX}\circ \eta_{GFX}\\
&=GF\id_X\circ G\id_{FX}=G(F\id_X\circ\id_{FX})=GF\id_X ,
\end{split}
\end{equation*}
hence $\eta\circ\nu=GF\id$. Moreover, for every $g:FX\to FY$ in $\dd$ we have
\begin{equation*}
\begin{split}
\nu_Y\circ Gg\circ\eta_X &=\p_{GFY,Y}(\epsilon_{FY})\circ Gg\circ\eta_X=\p_{X,Y}(\epsilon_{FY}\circ FGg\circ F\eta_X)\\&=\p_{X,Y}(g\circ\epsilon_{FX}\circ F\eta_X)=\p_{X,Y}(g\circ F\id_X)=\p_{X,Y}(g)\circ\id_X=\p_{X,Y}(g).
\end{split}
\end{equation*} 
Conversely, suppose that there exists a seminatural transformation $\nu :GF\rightarrow \Id_{\cc}$ such that $\eta\circ\nu =GF\Id$. Define for any $f\in \Hom_{\dd}(FX,FY)$, $\p_{X,Y}(f):=\nu_Y\circ Gf\circ\eta_X$. By naturality of $\eta$ and $\nu$, for any $h:X\to Y$ in $\cc$, $k:FY\to FZ$ in $\dd$, $l:Z\to T$ in $\cc$, we have $\p_{X,T}(Fl\circ k\circ Fh)=\nu_T\circ G(Fl\circ k\circ Fh)\circ\eta_X=(\nu_T\circ GFl )\circ Gk\circ (GFh\circ\eta_X)=l\circ (\nu_Z\circ Gk\circ \eta_{Y})\circ h=l\circ\p_{Y,Z}(k)\circ h$, thus $\p : \Hom_{\dd}(F-, F-)\rightarrow \Hom_{\cc}(-,-)$ is a natural transformation. 
Since $\p_{GFX,X}(\epsilon_{FX})=\nu_X\circ G\epsilon_{FX}\circ\eta_{GFX}=\nu_X\circ G\id_{FX}=\nu_X$ (the last equality follows from Remark \ref{rmk:seminat-alpha}), the correspondence between $\p$ and $\nu$ is bijective. For any $f\in \Hom_{\dd}(FX,FY)$, we have
\begin{equation*}
\begin{split}
(\f_{X,Y}\circ\p_{X,Y}) (f)&=F(\p_{X,Y}(f))=F(\nu_Y\circ Gf\circ\eta_X)=F(\id_Y\circ\nu_Y\circ Gf\circ\eta_X)\\&=F\id_Y\circ F\nu_Y\circ FGf\circ F\eta_{X}=\epsilon_{FY}\circ (F\eta_Y\circ F\nu_Y) \circ FGf\circ F\eta_{X}\\&=\epsilon_{FY}\circ FGF\id_{Y}\circ FGf\circ F\eta_{X}=\epsilon_{FY}\circ FG(F\id_{Y}\circ f)\circ F\eta_{X}\\&=F\id_Y\circ f\circ \epsilon_{FX}\circ F\eta_X=F\id_Y\circ f\circ F\id_X ,
\end{split}
\end{equation*}
so $F$ is a naturally semifull semifunctor.
\endproof

\begin{invisible}
\noindent\textbf{Semiadjoint triples.} We say that $F\dashv_\mathrm{s} G\dashv_\mathrm{s} H:\cc\to\dd$ is a \emph{semiadjoint triple} if it is a triple of semifunctors $F,H:\cc\to\dd$ and $G:\dd\to\cc$ such that $F\dashv_\mathrm{s} G$ and  $G\dashv_\mathrm{s} H$. The following is a semi-analogue of \cite[Proposition 2.19]{AB22} for semiadjoint triples. In particular the semifully faithful case is a semi-analogue of \cite[Proposition 3.4.2]{Bor94}. 
\begin{prop}
Let $G:\dd\to\cc$ be a semifunctor and $F\dashv_\mathrm{s} G\dashv_\mathrm{s} H:\cc\to\dd$ a semiadjoint triple. Then, $F$ is separable (resp. naturally semifull, semifully faithful) if and only if so is $H$.
\end{prop} 
\proof
We denote by $\eta:\id_\cc\to GF$ and $\epsilon:FG\to\id_\dd$ the unit and the counit of the semiadjunction $F\dashv_\mathrm{s}G$, respectively, and by $\beta:\id_\dd\to HG$ and $\alpha:GH\to\id_\cc$ the unit and the counit of the semiadjunction $G\dashv_\mathrm{s}H$, respectively. We only prove that if $H$ is separable (resp. naturally semifull, semifully faithful), so is $F$, as the converse implication follows by duality.\par
Assume that $H$ is separable. Then, by Theorem \ref{Th.Rafael}, there exists a seminatural transformation $\alpha':\id_\cc\to GH$ such that $\alpha\circ\alpha'=\id_{\id_\cc}$. Set $\eta'$ the composite 
$$\xymatrix{GF\ar[r]_-{GF\alpha'}&GFGH\ar[r]_-{G\epsilon H}&GH\ar[r]_-{\alpha}&\id_{\cc}.}$$ We have that 
\begin{gather*}
\eta'\circ\eta=\alpha\circ G\epsilon_H\circ GF\alpha'\circ\eta=\alpha\circ G\epsilon H\circ\eta GH\circ\alpha'=\alpha\circ G\id_H\circ \alpha'=\alpha\circ\alpha'=\id_{\id_\cc} .
\end{gather*}
Thus, $\eta'$ is the desired seminatural transformation such that $\eta'\circ\eta =\id_{\id_\cc}$. By Theorem \ref{Th.Rafael} $F$ is separable. Now, assume that $H$ is naturally semifull. Then, by Theorem \ref{thm:Raf-nat.full} there is a seminatural transformation $\alpha'':\id_\cc\to GH$ such that $\alpha''\circ\alpha=GH\id$. Set $\eta'':=\alpha\circ G\epsilon H\circ GF\alpha'' :GF\to\id_\cc$. 
Note that from $\alpha''\circ\alpha=GH\id$ and $\alpha G\circ G\beta =G\id$ it follows that $G\beta =\alpha'' G$. Indeed, $G\beta=G(HG\id\circ\beta)=G(H\id_G\circ HG\id\circ\beta)=G(H\id_G\circ\beta)=GH\id_G\circ G\beta=\alpha''G\circ \alpha G\circ G\beta =\alpha'' G\circ G\id=GHG\id\circ\alpha''G=\alpha''G$. Then, by naturality of $\alpha$, $\epsilon$ and $\alpha''$, we get
\begin{gather*}
\eta\circ\eta''=\eta\circ\alpha\circ G\epsilon H\circ GF\alpha''=\alpha GF\circ GH\eta\circ G\epsilon H\circ GF\alpha''=\alpha GF\circ G\epsilon HGF\circ GFGH\eta\circ GF\alpha''\\=\alpha GF\circ G\epsilon HGF\circ GF(\alpha'' GF\circ \eta) =\alpha GF\circ G\epsilon HGF\circ GFG\beta F\circ GF\eta \\=\alpha GF\circ G\beta F\circ G\epsilon F \circ GF\eta = G\id_F\circ GF\id =GF\id ,
\end{gather*}
hence by Theorem \ref{thm:Raf-nat.full} $F$ is naturally semifull. The semifully faithful case follows from Proposition \ref{prop:char-sff} by combining the previous separable and naturally semifull cases.\qedhere
\endproof
\end{invisible}

\section{Semiseparable semifunctors}\label{sect:semisep}
Recently, in \cite{AB22} semiseparable functors have been introduced in order to treat separability and natural fullness in a unified way. Semiseparability allows to describe separable and naturally full functors in terms of faithful and full functors, respectively. In this section we explore the same notion for semifunctors in order to complete the overview on semisplitting properties for the hom-set components associated with a semifunctor. 
In particular, in Proposition \ref{prop:sep-nat.full} we characterize separable and naturally semifull semifunctors in terms of faithful and semifull semifunctors, respectively.\\
\par
We say that a semifunctor $F: \cc \rightarrow \dd$ is \textbf{semiseparable} if there exists a natural transformation $\p : \Hom_{\dd}(F-, F-)\rightarrow \Hom_{\cc}(-,-)$ such that
\begin{equation}\label{eq_semi-sep}
\f\circ\p\circ\f = \f .
\end{equation}

\begin{prop}\label{prop:sep-nat.full} Let $F: \cc \rightarrow \dd$ be a semifunctor. Then,
	\begin{itemize}
		\item[(i)]$F$ is separable if and only if $F$ is semiseparable and faithful; 
		\item[(ii)]$F$ is naturally semifull if and only if $F$ is semiseparable and semifull. 
	\end{itemize}
\end{prop}
\proof
(i). The proof is analogous to that for functors, see \cite[Proposition 1.3]{AB22}.\\
	\begin{invisible}
		Assume that $F$ is separable. From $\p\circ\f = \id $ it follows that $\f\circ \p\circ\f = \f\circ\id=\f$, i.e. $F$ is semiseparable, and that for all $C,C'\in\cc$, the map $\f _{C,C'}$ is injective, i.e. $F$ is faithful. Conversely, if $F$ is semiseparable, we have that there exists a natural transformation $\p $ such that $\f\circ\p\circ\f = \f$, hence, if $F$ is faithful, $\p\circ\f = \id $, as $\f$ is injective.
	\end{invisible}
(ii). If $F$ is naturally semifull, then for any $f:X\to Y$ in $\cc$, we have that $(\f_{X,Y}\circ\p_{X,Y}\circ \f_{X,Y})(f)=\f_{X,Y}(\p_{X,Y}(Ff))=F\id_Y\circ Ff\circ F\id_X=F(\id_Y\circ f\circ \id_X)=Ff=\f_{X,Y}(f)$, hence $F$ is semiseparable. Moreover, a naturally semifull semifunctor is semifull by Lemma \ref{lem:semifull}. Conversely, assume that $F$ is semiseparable. Then there is a natural transformation $\p : \Hom_{\dd}(F-, F-)\rightarrow \Hom_{\cc}(-,-)$ such that $\f\circ\p\circ\f = \f$. If $F$ is semifull, then for any $f:FX\to FY$ in $\dd$ there exists $h:X\to Y$ in $\cc$ such that $F(h)=F\id_Y\circ f\circ F\id_X:FX\to FY$, so by naturality of $\p_{X,Y}$ we get that $\f_{X,Y}( \p_{X,Y}(f))=F(\id_Y\circ \p_{X,Y}(f)\circ\id_X)=F\p_{X,Y}(F\id_Y\circ f\circ F\id_X)=F\p_{X,Y}(F(h))=F(h)=F\id_Y\circ f\circ F\id_X$, hence $F$ is a naturally semifull semifunctor. \qedhere
\endproof
\begin{prop}\label{prop:semisep-semiiso}
	A semifunctor naturally semi-isomorphic to a semiseparable semifunctor is semiseparable.
\end{prop}
\begin{proof}
	Consider a natural semi-isomorphism $\alpha: F\to G$ of semifunctors, where the semifunctor $G:\cc\to\dd$ is semiseparable with respect to $\p^G$, and
define $\varsigma:\Hom_{\dd}(F-,F-)\to\Hom_\dd(G-,G-)$ by $\varsigma_{X,Y}(f)=\alpha_Y\circ f\circ \alpha_X^{-1}$, for every $f:FX\to FY$ in $\dd$. It is a natural semi-isomorphism with semi-inverse $\varsigma^{-1}:\Hom_{\dd}(G-,G-)\to\Hom_\dd(F-,F-)$ given by $\varsigma^{-1}_{X,Y}(g)=\alpha^{-1}_Y\circ g\circ\alpha_X$, for every $g:GX\to GY$ in $\dd$. \begin{invisible}Indeed, $(\varsigma_{X,Y}\circ\Hom_\dd(F\id_X, F\id_Y))(f)=\varsigma_{X,Y}(F\id_Y\circ f\circ F\id_X)=\alpha_Y\circ F\id_Y\circ f\circ F\id_X\circ\alpha^{-1}_X=\alpha_Y\circ f\circ\alpha^{-1}_X=\varsigma_{X,Y}(f)$ and $(\varsigma^{-1}_{X,Y}\circ\Hom_\dd(G\id_X, G\id_Y))(g)=\varsigma^{-1}_{X,Y}(G\id_Y\circ g\circ G\id_X)=\alpha^{-1}_Y\circ G\id_Y\circ g\circ G\id_X\circ\alpha_X=\alpha^{-1}_Y\circ g\circ\alpha_X=\varsigma^{-1}_{X,Y}(g)$. Moreover, $\varsigma^{-1}_{X,Y}(\varsigma_{X,Y}(f))=\alpha^{-1}_Y\circ\varsigma_{X,Y}(f)\circ\alpha_X=\alpha^{-1}_Y\circ\alpha_Y\circ f\circ\alpha^{-1}_X\circ\alpha_X=F\id_Y\circ f\circ F\id_X=\Hom_\dd(F\id_X,F\id_Y)(f)$ and $\varsigma_{X,Y}(\varsigma^{-1}_{X,Y}(g))=\alpha_Y\circ\varsigma^{-1}_{X,Y}(g)\circ\alpha^{-1}_X=\alpha_Y\circ\alpha^{-1}_Y\circ g\circ\alpha_X\circ\alpha^{-1}_X=G\id_Y\circ g\circ G\id_X=\Hom_\dd(G\id_X,G\id_Y)(g)$.\end{invisible} 
Note that $\f^F_{X,Y}=\varsigma^{-1}_{X,Y}\circ\f^G_{X,Y}$ and $\f^G_{X,Y}=\varsigma_{X,Y}\circ\f^F_{X,Y}$. In fact, $\f^F_{X,Y}(f)=Ff=F\id_Y\circ Ff=\alpha^{-1}_Y\circ\alpha_Y\circ Ff=\alpha^{-1}_Y\circ Gf\circ\alpha_X=\varsigma^{-1}_{X,Y}(Gf)=(\varsigma^{-1}_{X,Y}\circ\f^G_{X,Y})(f)$ and $\f^G_{X,Y}(f)=Gf=Gf\circ G\id_X=Gf\circ\alpha_X\circ\alpha^{-1}_X=\alpha_Y\circ Ff\circ\alpha^{-1}_X =\varsigma_{X,Y}(Ff)=(\varsigma_{X,Y}\circ\f^F_{X,Y})(f)$, for every $f:X\to Y$ in $\cc$. We show that $F$ results to be semiseparable with respect to $\p^F_{X,Y}:=\p^G_{X,Y}\circ\varsigma_{X,Y}$. Indeed, we have $\f^F_{X,Y}\circ\p^F_{X,Y}\circ\f^F_{X,Y}=\varsigma^{-1}_{X,Y}\circ\f^G_{X,Y}\circ\p^G_{X,Y}\circ\varsigma_{X,Y}\circ\f^F_{X,Y}
=\varsigma^{-1}_{X,Y}\circ\f^G_{X,Y}\circ\p^G_{X,Y}\circ\f^G_{X,Y}=\varsigma^{-1}_{X,Y}\circ\f^G_{X,Y}=\f^F_{X,Y}$.
\end{proof}
Semiseparable semifunctors satisfy the following properties with respect to composition, as in the functor case, cf. \cite[Lemma 1.12, Lemma 1.13]{AB22}.
\begin{lem}\label{lem:comp}
	Let $F: \cc \rightarrow \dd$ and $G:\dd\rightarrow\e$ be semifunctors and consider the composite semifunctor $G\circ F:\cc\rightarrow \e$.
	\begin{itemize}
		\item[(i)] If $F$ is semiseparable and $G$ is separable, then $G\circ F$ is semiseparable.
		\item[(ii)] If $F$ is naturally semifull and $G$ is semiseparable, then $G\circ F$ is semiseparable.
		\item[(iii)] If $G\circ F$ is semiseparable and $G$ is faithful, then $F$ is semiseparable.
	\end{itemize}
\end{lem}
\proof
(i). If $F$ is semiseparable with respect to $\p^F$ and $G$ is separable with respect to $\p^G$, then $G\circ F$ is semiseparable with respect to $\p^{GF}_{X,Y}:=\p^F_{X,Y}\p^G_{FX,FY}$, for every $X, Y$ in $\cc$, and the proof is the same of that for functors.\\
(ii). If $F$ is naturally semifull with respect to $\p^F$ and $G$ is semiseparable with respect to $\p^G$, then for every $f:X\to Y$ in $\cc$ we have
	\begin{gather*}
		\f^{GF}_{X,Y}\p^F_{X,Y}\p^G_{FX,FY}\f^{GF}_{X,Y}(f) =\f^G_{FX,FY}(\f^F_{X,Y}\p^F_{X,Y}(\p^G_{FX,FY}(GFf)))\\
		=\f^G_{FX,FY}(F\id_Y\circ\p^G_{FX,FY}(GFf)\circ F\id_X)=GF\id_Y\circ G\p^G_{FX,FY}(GFf)\circ GF\id_X\\=GF\id_Y\circ GFf\circ GF\id_X= GF(\id_Y\circ f\circ\id_X)=GFf=\f^{GF}_{X,Y}(f),
	\end{gather*}
	hence $G\circ F$ is semiseparable with respect to $\p^{GF}_{X,Y}:=\p^F_{X,Y}\p^G_{FX,FY}$.\\
(iii). If $G\circ F$ is semiseparable through $\p^{GF}$, then $\f^{GF}_{X,Y}\circ\p^{GF}_{X,Y}\circ\f^{GF}_{X,Y}=\f^{GF}_{X,Y}$, i.e. $\f^G_{FX,FY}\circ \f^F_{X,Y}\circ\p^{GF}_{X,Y}\circ \f^G_{FX,FY}\circ \f^F_{X,Y}=\f^G_{FX,FY}\circ\f^F_{X,Y}$, for every $X,Y\in\cc$. Since $G$ is faithful, we have that $\f^F_{X,Y}\circ \p^{GF}_{X,Y}\circ \f^G_{FX,FY}\circ \f^F_{X,Y}=\f^F_{X,Y}$, for every $X, Y$ in $\cc$, so $F$ is semiseparable through $\p^F_{X,Y}:=\p^{GF}_{X,Y}\circ\f^G_{FX,FY}$.
\qedhere
\endproof
\begin{rmk}
If $F:\cc\to\dd$ is a semiseparable functor which is neither separable, nor naturally full (see e.g. \cite[Example 3.2]{AB22}) and $G:\dd\to\e$ is a separable semifunctor, then the composite $G\circ F:\cc\to \e$ is a semiseparable semifunctor by Lemma \ref{lem:comp} (i), which is not separable, neither naturally semifull. In fact, if $G\circ F$ were separable, then by Lemma \ref{lem:comp-sep} (ii) $F$ would be separable, and if $G\circ F$ were naturally semifull, then by Proposition \ref{prop:comp} (ii) $F$ would be a naturally semifull functor, i.e. a naturally full functor, contradicting our assumptions.
\end{rmk}
Similarly to \cite[Proposition 1.4]{AB22}, we can attach to any semiseparable semifunctor a canonical idempotent (semi)natural transformation, that we call \textit{the associated idempotent}, which controls when the semifunctor is separable.
\begin{prop}\label{prop:idempotent}
	Let $F:\cc\rightarrow \dd$ be a semiseparable semifunctor. Then,
	there is a unique idempotent (semi)natural transformation $e:\id_{\cc%
	}\rightarrow \id_{\cc}$ such that $Fe=F\id$ with the following universal property: if $f,g:X\to Y$ are morphisms in $\cc$, then $Ff=Fg$ if and only if $e_Y\circ f=e_Y\circ g$. Moreover, $e = \id:\id_{\cc%
}\rightarrow \id_{\cc}$ if and only if $F$ is separable.
\end{prop}
\begin{proof}
	Since $F$ is semiseparable, there is a natural transformation $\mathcal{P}$
	such that $\mathcal{F}\circ \mathcal{P}\circ \mathcal{F}=\mathcal{F}$. Set $%
	e_{X}:=\mathcal{P}_{X,X}\left( F\id_{X}\right) $, for every $X\in\cc$. Note that for all $X\in\cc$, $Fe_{X}=F%
	\mathcal{P}_{X,X}\left( F\id_{X}\right) =\mathcal{F}_{X,X}\mathcal{P}%
	_{X,X}\mathcal{F}_{X,X}\left( \id_{X}\right) =\mathcal{F}%
	_{X,X}\left( \id_{X}\right) =F\id_{X}$. Then, by naturality of $\mathcal{P}$, we have $e_{X}\circ
	e_{X}=\mathcal{P}_{X,X}\left( F\id_{X}\right) \circ e_{X}=\mathcal{P}%
	_{X,X}\left( F\id_{X}\circ Fe_{X}\right) =\mathcal{P}_{X,X}\left(
	Fe_{X}\right) =\mathcal{P}%
	_{X,X}\left( F\id_{X}\right)=e_{X}$ and hence $e_{X}$ is idempotent. Moreover,
	for every morphism $f:X\rightarrow Y$ in $\cc$ we have $f\circ e_{X}=f\circ \mathcal{P%
	}_{X,X}\left( F\id_{X}\right) =\mathcal{P}_{X,Y}\left( Ff\circ
	F\id_{X}\right) =\mathcal{P}_{X,Y}\left( F\id_{Y}\circ
	Ff\right) =\mathcal{P}_{Y,Y}\left( F\id_{Y}\right) \circ
	f=e_{Y}\circ f$, so that $f\circ e_{X}=e_{Y}\circ f$. Thus, $e=\left(
	e_{X}\right) _{X\in \cc}:\id_{\cc}\rightarrow
	\id_{\cc}$ is an idempotent (semi)natural transformation such that
	$Fe=F\id$. Now, consider morphisms $f,g:X\to Y$ in $\cc$. If $Ff=Fg$, then $\mathcal{P}%
	_{X,Y}\left( Ff\right) =\mathcal{P}_{X,Y}\left( Fg\right)$, i.e. $\mathcal{P}%
	_{Y,Y}\left( F\id_{Y}\right) \circ f=\mathcal{P}%
	_{Y,Y}\left( F\id_{Y}\circ Ff \right)=\mathcal{P}_{Y,Y}\left( F\id_Y\circ Fg\right)=\mathcal{P}_{Y,Y}\left( F\id_Y\right) \circ g$, i.e. $e_{Y}\circ f=e_{Y}\circ g$. Conversely, from $e_{Y}\circ f=e_{Y}\circ g$ we get $Fe_{Y}\circ Ff=Fe_{Y}\circ Fg$ and hence $Ff=Fg$ as $Fe_Y=F\id_Y$. Finally, let $e':\id_{\cc%
	}\rightarrow \id_{\cc}$ be an idempotent (semi)natural transformation such that, if $f,g:X\to Y$ are morphisms in $\cc$, then $Ff=Fg$ if and only if $e'_Y\circ f=e'_Y\circ g$. From $e'_X\circ e'_X=e'_X\circ \id_X$ we get $Fe'_X=F\id_X$, hence $Fe'=F\id$. From the universal property of $e$ we have $e_X\circ e'_X=e_X\circ \id_X$, i.e. $e_X\circ e'_X=e_X$. By interchanging the roles of $e$ and $e'$, similarly we get $e'_X\circ e_X=e'_X$, and by naturality we have $e_X\circ e'_X=e'_X\circ e_X$, hence $e_X=e'_X$, i.e. $e=e'$.

Now, if $F$ is separable then there is a natural transformation $\mathcal{P}$
such that $\mathcal{P}\circ \mathcal{F}=\id$ and hence $e_X=\mathcal{P}_{X,X}(F\id_{X})=\mathcal{P}_{X,X}\mathcal{F}_{X,X}(\id_{X})=\id_{X}$ for all $X\in\cc$. Conversely, let $F:\cc\to\dd$ be a semiseparable semifunctor through $\p$ and suppose $e=\id$ is the associated idempotent. Then, for every $f:X\to Y$ in $\cc$, we have $\mathcal{P}_{X,Y}(Ff)=\mathcal{P}_{X,Y}(Ff\circ F\id_{X})=f\circ \mathcal{P}_{X,X}(F\id_{X})=f\circ e_X=f\circ \id_X=f$ so that $\mathcal{P}\circ \mathcal{F}=\id$, hence $F$ is separable.
\end{proof}

The following Rafael-type Theorem characterizes semiseparability for semifunctors that are part of a semiadjunction, cf. \cite[Theorem 2.1]{AB22} for the case of functors.
\begin{thm}\label{thm:rafael-semisep}
	Let $F\dashv_\mathrm{s} G:\dd\to \cc$ be a semiadjunction, with unit $\eta$ and counit $\epsilon$. Then,
	\begin{itemize}
		\item[(1)] $F$ is semiseparable if and only if there exists a natural transformation $\nu : GF\rightarrow \id _{\cc}$ which satisfies one of the following equivalent conditions:
		\begin{enumerate}[(1)]
			\item[(i)] $\eta\circ\nu\circ\eta =\eta$;
			\item[(ii)] $F\nu\circ F\eta = F\id$.
		\end{enumerate}
		\item[(2)] $G$ is semiseparable if and only if there exists a natural transformation $\gamma : \id _{\dd}\rightarrow FG$ which satisfies one of the following equivalent conditions:
		\begin{enumerate}[(1)]
			\item[(i)]  $\epsilon\circ\gamma\circ\epsilon = \epsilon$;
			\item[(ii)] $G\epsilon\circ G\gamma = G\id$.
		\end{enumerate}
	\end{itemize}
\end{thm}
\proof
(1). Assume that $F$ is semiseparable and let $\p$ be the associated natural transformation such that $\f\circ\p\circ\f =\f$. We define $\nu : GF\rightarrow \id _{\cc}$ by $\nu_{X}:=\p_{GFX,X}(\epsilon_{FX}): GFX\rightarrow X$. It is indeed natural as for any morphism $f:X\to Y$ in $\cc$, we have $f\circ\nu_X=f\circ\p_{GFX,X}(\epsilon_{FX})=\p_{GFX,Y}(Ff\circ\epsilon_{FX})=\p_{GFX,Y}(\epsilon_{FY}\circ FGFf)=\p_{GFY,Y}(\epsilon_{FY})\circ GFf=\nu_Y\circ GFf$. For every $g:FX\to FY$ in $\dd$, by naturality of $\p$ we have that
	\begin{equation*}
	\begin{split}
	\nu_Y\circ Gg\circ\eta_X &=\p_{GFY,Y}(\epsilon_{FY})\circ Gg\circ\eta_X=\p_{X,Y}(\epsilon_{FY}\circ FGg\circ F\eta_X)\\
	&=\p_{X,Y}(g\circ\epsilon_{FX}\circ F\eta_X)=\p_{X,Y}(g\circ F\id_X)=\p_{X,Y}(g)\circ\id_X=\p_{X,Y}(g).
	\end{split}
	\end{equation*} 
	Moreover, it holds
	\begin{equation*}
	\begin{split}
	\eta_X\circ\nu_X\circ\eta_X &= \eta_X\circ\p_{GFX,X}(\epsilon_{FX})\circ\eta_{X}
	=\eta_{X}\circ\p_{X,X}(\epsilon_{FX}\circ F\eta_{X})
	=\eta_X\circ \p_{X,X}(F\id _{X})\\
	&=GF\p_{X,X}(F\id _{X})\circ\eta_X
	=G((\f_{X,X}\circ\p_{X,X}\circ\f_{X,X})(\id _X))\circ\eta_X\\
	&\overset{\eqref{eq_semi-sep}}{=}G\f_{X,X}(\id _X)\circ\eta_X=GF(\id _X)\circ\eta_X
	=\eta_X ,
	\end{split}
	\end{equation*}
	hence condition (i) is satisfied. Conversely, assume that there exists a natural transformation $\nu :GF\rightarrow \id _{\cc}$ such that $\eta\circ\nu\circ\eta =\eta$, and for any $f\in \Hom_{\dd}(FX,FY)$ define $\p_{X,Y}(f):=\nu_Y\circ Gf\circ\eta_X$.	From the naturality of $\eta$ and $\nu$, for any $h:X\to Y$ in $\cc$, $k:FY\to FZ$ in $\dd$, $l:Z\to T$ in $\cc$, we have $\p_{X,T}(Fl\circ k\circ Fh)=\nu_T\circ G(Fl\circ k\circ Fh)\circ\eta_X=(\nu_T\circ GFl )\circ Gk\circ (GFh\circ\eta_X)=l\circ (\nu_Z\circ Gk\circ \eta_{Y})\circ h=l\circ\p_{Y,Z}(k)\circ h$, thus $\p : \Hom_{\dd}(F-, F-)\rightarrow \Hom_{\cc}(-,-)$ is a natural transformation. 
	Since $\p_{GFX,X}(\epsilon_{FX})=\nu_X\circ G\epsilon_{FX}\circ\eta_{GFX}=\nu_X\circ G\id_{FX}=\nu_X$, the correspondence between $\p$ and $\nu$ is bijective. For every $f\in \Hom_{\cc}(X,Y)$ we have
	\begin{equation*}
	\begin{split}
	(\f_{X,Y}\circ\p_{X,Y}&\circ\f_{X,Y}) (f)= F(\p_{X,Y}(F(f)))=F(\nu_Y\circ GF(f)\circ\eta_X)=F(\nu_Y\circ\eta_{Y}\circ f)\\&
	= F(\id_Y\circ\nu_Y\circ\eta_{Y}\circ f)=F\id_Y\circ F(\nu_Y\circ\eta_{Y}\circ f)=\epsilon_{FY}\circ F\eta_Y\circ F(\nu_Y\circ\eta_{Y}\circ f)\\
	&=\epsilon_{FY}\circ F(\eta_Y\circ\nu_Y\circ\eta_{Y}\circ f)=\epsilon_{FY}\circ F(\eta_{Y}\circ f)=\epsilon_{FY}\circ F\eta_Y\circ Ff
	\\&=F\id _{Y}\circ Ff	=Ff=\f_{X,Y}(f),
	\end{split}
	\end{equation*}
	so $F$ is semiseparable. Finally, we prove that (i) and (ii) are equivalent.\\
	(i) $\Rightarrow$ (ii): $F\nu \circ F\eta = F(\id\circ\nu \circ \eta) =F\id\circ F\nu \circ F\eta =\epsilon F\circ F\eta\circ F\nu\circ F\eta \overset{\text{(i)}}{=}\epsilon F\circ F\eta = F\id $.\\
	(ii) $\Rightarrow$ (i): By naturality of $\eta$, we have $\eta\circ\nu\circ\eta =\eta\circ (\nu \circ\eta ) = GF(\nu\circ\eta )\circ\eta = G(F\nu\circ F\eta )\circ \eta = GF\id\circ\eta=\eta$.\\
(2). It follows by duality. \qedhere
\endproof

\noindent\textbf{Idempotent completion and semiadjoint triples.} In Subsection \ref{subsect:idempcompl} we have observed that any semifunctor $F:\cc\rightarrow \dd$ induces a functor $%
F^{\natural }:\cc^{\natural }\rightarrow \dd^{\natural }$
such that $F^{\natural }\left( C,c\right) =\left( FC,Fc\right) $ and $%
F^{\natural }f=Ff$. 
Now, we show how the notions introduced in this paper for a semifunctor $F$ are related to the corresponding functorial notions for its completion $F^\natural$. The following result is a semifunctorial version of \cite[Proposition 2.1 and Corollary 2.2]{AB22-II}.
\begin{prop}\label{prop:completion-ffs}
	Let $F:\cc\to\dd$ be a semifunctor. Then,
	\begin{itemize}
		\item[(i)] $F$ is faithful if and only if $F^\natural$ is a faithful functor;
		\item[(ii)] $F$ is semifull if and only if $F^\natural$ is a full functor; 
		\item[(iii)] $F$ is semiseparable if and only if $F^\natural$ is a semiseparable functor. 
	\end{itemize}
\end{prop}
\proof
(i). It follows from the fact that $F^{\natural }f=Ff$, for any morphism $f:(C,c)\to (C',c')$ in $\cc^\natural$, i.e. for any morphism $f:C\to C'$ in $\cc$ such that $f=c'\circ f\circ c$. \begin{invisible}
		Assume that the semifunctor $F$ is faithful. Then, let $f,g:(C,c)\to (C',c')$ be morphisms in $\cc^\natural$ such that $F^\natural (f)=F^\natural (g)$. Since for all $C,C'$ in $\cc$ the map $\f^F_{C,C'}: \Hom_{\cc}(C,C')\rightarrow \Hom_{\dd}(FC, FC')$ is injective, from $F(f)=F^\natural (f)=F^\natural (g)=F(g)$ it follows that $f=g$. Therefore, also the map $$\f^{F^\natural} _{(C,c),(C',c')}: \Hom_{\cc^\natural}((C,c),(C',c'))\rightarrow \Hom_{\dd^\natural}((FC,Fc) ,(FC',Fc'))$$ is injective, so $F^\natural$ is faithful. Conversely, assume that $F^\natural$ is faithful. Then, let $f,g:C\to C'$ be morphisms in $\cc$ such that $F (f)=F (g)$. Since for all $(C,c),(C',c')$ in $\cc^\natural$ the map $\f^{F^\natural}_{(C,c),(C',c')}: \Hom_{\cc^\natural}((C,c),(C',c'))\rightarrow \Hom_{\dd^\natural}((FC, Fc), (FC',Fc'))$ is injective, by considering $f$, $g$ as morphisms $f:(C,\id_C)\to(C',\id_{C'})$ in $\cc^\natural$, from $F^\natural (f)=F (f)=F (g)=F^\natural (g)$ it follows that $f=g$. Therefore, also the map $$\f^F_{C,C'}: \Hom_{\cc}(C,C')\rightarrow \Hom_{\dd}(FC, FC')$$ is injective, thus $F$ is faithful.
	\end{invisible}\\ 
(ii). Assume that $F$ is semifull. 
	Let $f:F^{\natural }\left( C,c\right) \to F^{\natural }\left( C',c'\right) $ be a morphism in $\dd^\natural$, i.e. a morphism $f:FC\to FC'$ in $\dd$ such that $f=Fc'\circ f\circ Fc$. Since $F$ is semifull, there exists a morphism $g:C\to C'$ in $\cc$ such that $F(g)=F\id_{C'}\circ f\circ F\id_{C}$. Set $g':=c'\circ g\circ c : C\to C'$. Note that $c'g' c= c'(c'g c)c=c'gc=g'$, hence $g':(C,c)\to (C',c')$ is a morphism in $\cc^\natural$. Then, $F^\natural (g')=F(g')=F(c'gc)=F(c')\circ F(g)\circ F(c)=Fc'\circ F\id_{C'}\circ f\circ F\id_C\circ Fc=Fc'\circ f\circ Fc=f$. Thus, $F^\natural$ is full.	Conversely, assume that $F^\natural$ is full.  Let $f:FC\to FC'$ be a morphism in $\dd$. Consider the morphism $f':(FC, F\id_C)\to (FC',F\id_{C'})$ in $\dd^\natural$ given by $f'=F\id_{C'}\circ f\circ F\id_C$.\begin{invisible}
		Note that $F\id_{C'}\circ f'\circ F\id_C=F\id_{C'}F\id_{C'}\circ f\circ F\id_C F\id_C=F\id_{C'}\circ f\circ F\id_C=f'$.
	\end{invisible}
	Then, there is a morphism $g:(C,\id_C)\to (C',\id_{C'})$ in $\cc^\natural$ (i.e. a morphism $g:C\to C'$ in $\cc$) such that $F^\natural g=f'$. Thus, we have $F(g)=F^\natural(g)=f'=F\id_{C'}\circ f\circ F\id_C$, hence $F$ is semifull.\\
(iii). If $F$ is semiseparable, then the proof of the fact that $F^\natural$ is semiseparable is the same as in \cite[Corollary 2.2]{AB22-II} for the functorial case. \begin{invisible}
		Then, there is a natural transformation $\p^F: \Hom_{\dd}(F-, F-)\rightarrow \Hom_{\cc}(-,-)$ such that $\f^F\p^F\f^F =\f^F$. Define $\p^{F^\natural} : \Hom_{\dd^\natural}(F^\natural -, F^\natural -)\rightarrow \Hom_{\cc^\natural}(-,-)$ by $\p^{F^\natural}_{C,C'} (g)=\p^F_{C,C'}(g)$, for every $g: (FC,Fc)\to (FC',Fc')$ in $\dd^\natural$. Since $g=Fc'\circ g\circ Fc$, by naturality of $\p^F$ it follows that $c'\circ \p^{F^\natural}_{C,C'}(g)\circ c =c'\circ \p^F_{C,C'}(g)\circ c =\p^F_{C,C'}(Fc'\circ g\circ Fc)= \p^F_{C,C'}(g)=\p^{F^\natural}_{C,C'} (g)$, hence $\p^{F^\natural}_{C,C'} (g)$ is a morphism in $\cc^\natural$. Moreover, $\p^{F^\natural}$ is a natural transformation as so is $\p^F$     for any $h:(X,x)\to (Y,y)$, $k:F^\natural (Y,y)\to F^\natural (Z,z)$, and $l:(Z,z)\to (T,t)$ we have $\p^{F^\natural}_{X,T}(F^\natural l\circ k\circ F^\natural h)=\p^F_{X,T}(Fl\circ k\circ Fh)=l\circ\p^F_{Y,Z}(k)\circ h=l\circ\p^{F^\natural}_{Y,Z}(k)\circ h$, and for any $f: (C,c)\to (C',c')$ in $\cc^\natural$ it holds $\f^{F^\natural}_{C,C'}\p^{F^\natural}_{C,C'}\f^{F^\natural}_{C,C'}(f)=\f^F_{C,C'}\p^F_{C,C'}\f^F_{C,C'}(f)=\f^F_{C,C'}(f)=\f^{F^\natural}_{C,C'}(f)$.
	\end{invisible}
Conversely, assume that $F^\natural$ is semiseparable. Then, there is a natural transformation $ \p^{F^\natural} : \Hom_{\dd^\natural}(F^\natural -, F^\natural -)\rightarrow \Hom_{\cc^\natural}(-,-)$ such that $\f^{F^\natural}\p^{F^\natural}\f^{F^\natural} =\f^{F^\natural}$. Define $\p^{F}: \Hom_{\dd}(F-, F-)\rightarrow \Hom_{\cc}(-,-)$ by $\p^F_{C,C'}(g)=\p^{F^\natural}_{C,C'} (g)$, for every $g: FC\to FC'$ in $\dd$, i.e. for every $g: (FC, \id_{FC})\to (FC',\id_{FC'})$ in $\dd^\natural$. Thus, $\p^F_{C,C'}(g)$ is a morphism in $\cc$ and, by naturality of $\p^{F^\natural}$, also $\p^{F}$ is a natural transformation. \begin{invisible}Indeed, for any $h:X\to Y$, $k:FY\to FZ$, and $l:Z\to T$ we have $\p^F_{X,T}(Fl\circ k\circ Fh)=\p^{F^\natural}_{X,T}(F^\natural l\circ k\circ F^\natural h)=l\circ\p^{F^\natural}_{Y,Z}(k)\circ h=l\circ\p^F_{Y,Z}(k)\circ h$.
	\end{invisible}
	Moreover, for any $f: C\to C'$ in $\cc$, we have $\f^F_{C,C'}\p^F_{C,C'}\f^F_{C,C'}(f)=\f^{F^\natural}_{C,C'}\p^{F^\natural}_{C,C'}\f^{F^\natural}_{C,C'}(f)=\f^{F^\natural}_{C,C'}(f)=\f^F_{C,C'}(f)$.\qedhere	
\endproof
\begin{cor}\label{cor:sep-natsful-semiff-compl}
	Let $F:\cc\to\dd$ be a semifunctor. Then,
	\begin{itemize}
		\item[(i)] $F$ is separable if and only if $F^\natural$ is a separable functor;
		\item[(ii)] $F$ is naturally semifull if and only if $F^\natural$ is a naturally full functor; 
		\item[(iii)] $F$ is semifully faithful if and only if $F^\natural$ is a fully faithful functor. 
	\end{itemize}
\end{cor}	
\proof It follows from Proposition \ref{prop:completion-ffs} and Proposition \ref{prop:sep-nat.full}. 
\endproof
\begin{rmk}
	Proposition \ref{prop:char-sff} can be seen as a consequence of Corollary \ref{cor:sep-natsful-semiff-compl}. In fact, by \cite[Remark 2.2 (3)]{AMCM06} a functor is fully faithful if and only if it is separable and naturally full, so by Corollary \ref{cor:sep-natsful-semiff-compl} a semifunctor is semifully faithful if and only if it is separable and naturally semifull.
\end{rmk}
We say that $F\dashv_\mathrm{s} G\dashv_\mathrm{s} H:\cc\to\dd$ is a \emph{semiadjoint triple} if it is a triple of semifunctors $F,H:\cc\to\dd$ and $G:\dd\to\cc$ such that $F\dashv_\mathrm{s} G$ and  $G\dashv_\mathrm{s} H$ are semiadjunctions. The following is a semifunctorial analogue of \cite[Proposition 2.19]{AB22} for semiadjoint triples. In particular the semifully faithful case is a semifunctorial analogue of \cite[Proposition 3.4.2]{Bor94}. 
\begin{prop}\label{prop:semiadj-triple}
	Let $F\dashv_\mathrm{s} G\dashv_\mathrm{s} H:\cc\to\dd$ be a semiadjoint triple of semifunctors. Then, $F$ is semiseparable (resp. separable, naturally semifull, semifully faithful) if and only if so is $H$.
\end{prop} 
\proof
Given a semiadjoint triple $F\dashv_\mathrm{s} G\dashv_\mathrm{s} H:\cc\to\dd$, by \cite[Theorem 3.5]{Ho93} we obtain an adjoint triple $F^\natural\dashv G^\natural\dashv H^\natural:\cc^\natural\to\dd^\natural$ of functors. 
By \cite[Proposition 2.19]{AB22} we know that $F^\natural$ is semiseparable (resp. separable, naturally full) if and only if so is $H^\natural$. Thus, by Proposition \ref{prop:completion-ffs} and Corollary \ref{cor:sep-natsful-semiff-compl} we have that the semifunctor $F$ is semiseparable (resp. separable, naturally semifull) if and only if so is $H$. As a consequence of Proposition \ref{prop:char-sff}, by combining the previous separable and naturally semifull cases, we get that $F$ is semifully faithful if and only if so is $H$.\qedhere
\begin{invisible}
We denote by $\eta:\id_\cc\to GF$ and $\epsilon:FG\to\id_\dd$ the unit and the counit of the semiadjunction $F\dashv_\mathrm{s}G$, respectively, and by $\beta:\id_\dd\to HG$ and $\alpha:GH\to\id_\cc$ the unit and the counit of the semiadjunction $G\dashv_\mathrm{s}H$, respectively. We only prove that if $H$ is separable (resp. naturally semifull, semifully faithful), so is $F$, as the converse implication follows by duality.\par
Assume that $H$ is separable. Then, by Theorem \ref{Th.Rafael}, there exists a seminatural transformation $\alpha':\id_\cc\to GH$ such that $\alpha\circ\alpha'=\id_{\id_\cc}$. Set $\eta'$ the composite 
$$\xymatrix{GF\ar[r]_-{GF\alpha'}&GFGH\ar[r]_-{G\epsilon H}&GH\ar[r]_-{\alpha}&\id_{\cc}.}$$ We have that 
\begin{gather*}
	\eta'\circ\eta=\alpha\circ G\epsilon_H\circ GF\alpha'\circ\eta=\alpha\circ G\epsilon H\circ\eta GH\circ\alpha'=\alpha\circ G\id_H\circ \alpha'=\alpha\circ\alpha'=\id_{\id_\cc} .
\end{gather*}
Thus, $\eta'$ is the desired seminatural transformation such that $\eta'\circ\eta =\id_{\id_\cc}$. By Theorem \ref{Th.Rafael} $F$ is separable. Now, assume that $H$ is naturally semifull. Then, by Theorem \ref{thm:Raf-nat.full} there is a seminatural transformation $\alpha'':\id_\cc\to GH$ such that $\alpha''\circ\alpha=GH\id$. Set $\eta'':=\alpha\circ G\epsilon H\circ GF\alpha'' :GF\to\id_\cc$. 
Note that from $\alpha''\circ\alpha=GH\id$ and $\alpha G\circ G\beta =G\id$ it follows that $G\beta =\alpha'' G$. Indeed, $G\beta=G(HG\id\circ\beta)=G(H\id_G\circ HG\id\circ\beta)=G(H\id_G\circ\beta)=GH\id_G\circ G\beta=\alpha''G\circ \alpha G\circ G\beta =\alpha'' G\circ G\id=GHG\id\circ\alpha''G=\alpha''G$. Then, by naturality of $\alpha$, $\epsilon$ and $\alpha''$, we get
\begin{gather*}
	\eta\circ\eta''=\eta\circ\alpha\circ G\epsilon H\circ GF\alpha''=\alpha GF\circ GH\eta\circ G\epsilon H\circ GF\alpha''=\alpha GF\circ G\epsilon HGF\circ GFGH\eta\circ GF\alpha''\\=\alpha GF\circ G\epsilon HGF\circ GF(\alpha'' GF\circ \eta) =\alpha GF\circ G\epsilon HGF\circ GFG\beta F\circ GF\eta \\=\alpha GF\circ G\beta F\circ G\epsilon F \circ GF\eta = G\id_F\circ GF\id =GF\id ,
\end{gather*}
hence by Theorem \ref{thm:Raf-nat.full} $F$ is naturally semifull. The semifully faithful case follows from Proposition \ref{prop:char-sff} by combining the previous separable and naturally semifull cases.\qedhere
\end{invisible}
\endproof

\section{Examples and applications}\label{sect:examples}

In this section we provide examples of semifull, naturally semifull, (semi)separable and semifully faithful semifunctors. 

\begin{es}\label{es:iotaequtr}
\textbf{The forgetful semifunctor.}\\ 
See \cite[Example 2.13]{AB22-II}. Let $\cc$ be a category with idempotent completion $\cc^\natural$. Consider the canonical functor $\iota_{\cc}:\cc\rightarrow\cc^{\natural }$ given by $X\mapsto (X,\id_X)$, $[f:X\to X']\mapsto [f:(X,\id_X)\to (X',\id_{X'})]$ and the forgetful semifunctor $\upsilon_\cc : \cc^\natural \to\cc$ which maps an object $%
	\left( X,e\right)\in\cc^\natural$ to the underlying object $X$ and a morphism $f:(X,e)\to (X',e') $ to the underlying morphism $\upsilon_\cc f:X\rightarrow X'$ in $\cc$ such that $e'\circ\upsilon_\cc f\circ e=\upsilon_\cc f$. It is indeed a semifunctor as $\upsilon_\cc (\id_{(X,e)})=e\neq \id_X $ in general. As mentioned in Subsection \ref{subsect:idempcompl}, by \cite[Theorem 2.10]{Ho93} the Karoubi envelope functor $\kappa:\Cat_\mathrm{s}\to\Cat$, defined by $\kappa (\cc)=\cc^\natural$, $\kappa (F)=F^\natural$, for any category $\cc$ and any semifunctor $F:\cc\to\dd$, is the right adjoint of the inclusion functor $i:\Cat\to\Cat_\mathrm{s}$. Then, $\iota _{\cc}$ and $\upsilon_\cc$ result to be the $\cc$-components of the unit and of the counit for the adjunction $i\dashv \kappa$, respectively. Moreover, $\upsilon_\cc\dashv_\mathrm{s}\iota _{\cc}$ and $\iota _{\cc}\dashv_\mathrm{s}\upsilon_\cc$ are semiadjunctions, cf. \cite[Example 6]{Ho90}. The component $(\eta_\cc)_{\left( C,c\right) }:\left( C,c\right) \rightarrow \iota _{\cc%
	}\upsilon _{\cc}\left( C,c\right) =\left( C,\id_{C}\right) $ of the unit $\eta_\cc$ of $\upsilon_\cc\dashv_\mathrm{s}\iota _{\cc}$ is defined by setting $\upsilon _{\cc}((\eta_\cc) _{\left( C,c\right)
	}):=c:C\rightarrow C$ while the counit is $\epsilon _{\cc} :=\id_{\id_{\cc}}:\upsilon _{\cc}\iota _{\cc}=%
	\id_{\cc}\rightarrow \id_{\cc}$. The unit of $\iota _{\cc}\dashv_\mathrm{s}\upsilon_\cc$ is $\epsilon _{\cc} :=\id_{\id_{\cc}}:\id_\cc\to\id_\cc=\upsilon _{\cc}\iota _{\cc}$ while the component $(\nu_\cc) _{\left( C,c\right) }:\iota _{\cc}\upsilon _{\cc}\left( C,c\right) =\left( C,\id_{C}\right) \rightarrow \left(
	C,c\right) $ of the counit is given by $\upsilon _{\cc}((\nu_\cc) _{\left( C,c\right)
	}):=c:C\rightarrow C$. Note that
	\begin{equation*}
	\upsilon _{\cc}\left( (\nu_{\cc}) _{\left( C,c\right) }\circ (\eta_{\cc}) _{\left(
		C,c\right) }\right) =\upsilon _{\cc}((\nu_{\cc}) _{\left( C,c\right) })\circ
	\upsilon _{\cc}((\eta_{\cc}) _{\left( C,c\right) })=c\circ c=c=\upsilon _{%
		\cc}( \id_{\left( C,c\right) }),
	\end{equation*}%
	hence $\nu_\cc \circ \eta_\cc =\id_{\id_{\cc^{\natural}}}$, thus by Theorem \ref{Th.Rafael} it follows that $\upsilon_\cc$ is a separable semifunctor. Moreover, $$\upsilon _{\cc}\left( \left( \eta _{\cc}\right)
	_{\left( C,c\right) }\circ (\nu_{\cc}) _{\left( C,c\right) }\right) =\upsilon _{%
		\cc}(\left( \eta _{\cc}\right) _{\left( C,c\right) })\circ
	\upsilon _{\cc}((\nu_{\cc}) _{\left( C,c\right) })=c\circ c=c=\upsilon _{%
		\cc}\iota _{\cc}\upsilon _{\cc}\id_{\left(
		C,c\right) },$$ so $\left( \eta _{\cc}\right) _{\left( C,c\right)
	}\circ \left( \nu _{\cc}\right) _{\left( C,c\right) }=\iota _{%
		\cc}\upsilon _{\cc}\id_{\left( C,c\right) }$ and
	hence it holds $\eta _{\cc}\circ \nu _{\cc}=\iota _{\cc%
	}\upsilon _{\cc}\id.$ By Theorem \ref{thm:Raf-nat.full} it follows that $\upsilon_\cc$ is also a naturally semifull semifunctor, hence semifully faithful by Proposition \ref{prop:char-sff}.
\end{es}

\begin{es}\label{es:semiproduct}\textbf{Semi-product semifunctor.}\\
	Recall from \cite[Definition 4.3]{Ho93} that a binary \emph{semi-product} of objects $A, B$ in a category $\cc$ consists of an object $A\times B$ in $\cc$, an arrow $\pi_A:A\times B\to A$, an arrow $\pi_B:A\times B\to B$, an arrow $\langle f,g \rangle :C\to A\times B$, for each $f:C\to A$, $g:C\to B$ in $\cc$, such that
	\begin{itemize}
		\item[(1)] $\pi_A\circ \langle f,g\rangle = f$,
		\item[(2)] $\pi_B\circ \langle f,g\rangle = g$,
		\item[(3)] $\langle f\circ h, g\circ h\rangle = \langle f,g\rangle\circ h$, for any morphism $h: D\to C$ in $\cc$.
	\end{itemize}
	A binary semi-product is a binary product if and only if $\langle \pi_A,\pi_B \rangle=\id_{A\times B}$. If $\cc$ is a category with semi-products for all pairs of objects $A, B$, let $\times_\cc:\cc\times\cc\to\cc$ be the semifunctor given by 
	$$(A,B)\mapsto A\times B\quad \text{and}\quad (f,g)\mapsto f\times g:= \langle f\circ\pi_A, g\circ\pi_B\rangle .$$
	It is indeed a semifunctor as $\id_A\times\id_B=\langle\pi_A,\pi_B\rangle$. 
	Moreover, it is a right semiadjoint of the functor $\Delta_\cc:\cc\to\cc\times \cc$, $A\mapsto (A,A)$, $f\mapsto (f,f)$, for any morphism $f:A\to A'$ in $\cc$. They actually form a semiadjunction $\Delta_\cc\dashv_\mathrm{s}\times_\cc$ with unit the seminatural transformation $\eta:\id_{\cc}\to\times_\cc\Delta_\cc$, given on components by $\eta_C=\langle\id_C,\id_C\rangle :C\to C\times C$, so that $\pi_C\circ \langle \id_C,\id_C\rangle =\id_C$, for every $C$ in $\cc$, and counit the seminatural transformation $\epsilon:\Delta_\cc\times_\cc\to\id_{\cc\times\cc}$, given on components by $\epsilon_{(A,B)}=(\pi_A,\pi_B):(A\times B, A\times B)\to (A,B)$. 
	\begin{invisible}	
		They are natural transformations as, for any morphism $f:A\to B$ in $\cc$, 
		\[\begin{split}
			(f\times f)\circ \eta_A&=\langle f\pi_A,f\pi_A\rangle\circ\langle \id_A,\id_A\rangle =\langle f,f\rangle\circ\pi_A\circ\langle\id_A,\id_A\rangle =\langle f,f \rangle\circ\id_A\\&=\langle f,f\rangle =\langle\id_B\circ f,\id_B\circ f\rangle =\langle\id_B,\id_B\rangle \circ f=\eta_B\circ f,
		\end{split}
		\]		
		and for any morphism $(f,g):(A,B)\to (A',B')$ in $\cc\times\cc$, 
		\[\begin{split}
			(f,g)\circ \epsilon_{(A,B)}&=(f,g)\circ (\pi_A,\pi_B)=(f\circ\pi_A,g\circ \pi_B)\\&=(\pi_{A'}\circ\langle f\pi_A,g\pi_B\rangle, \pi_{B'}\circ\langle f\pi_A,g\pi_B\rangle)\\
			&=(\pi_{A'}\circ (f\times g),\pi_{B'}\circ (f\times g))=(\pi_{A'},\pi_{B'})\circ (f\times g,f\times g)\\&=\epsilon_{(A',B')}\circ (f\times g,f\times g).
		\end{split}
		\]
		For any object $(A,B)\in\cc\times\cc$, consider $\eta_{A\times B}:A\times B\to (A\times B)\times (A\times B)$, where $(A\times B)\times (A\times B)$ is the semi-product with arrows $\pi_1,\pi_2:(A\times B)\times (A\times B)\to A\times B$. We have that
		\[\begin{split}
			\times_\cc\epsilon_{(A,B)}\circ\eta_{A\times B}&=\langle \pi_A\circ\pi_{1}, \pi_B\circ \pi_{2}\rangle\circ\eta_{A\times B}
			=\langle \pi_A\circ\pi_{1}\circ\eta_{A\times B}, \pi_B\circ \pi_{2}\circ\eta_{A\times B}\rangle=\\
			&=\langle \pi_A\circ\pi_{1}\circ\langle\id_{A\times B},\id_{A\times B}\rangle, \pi_B\circ \pi_{2}\circ\langle\id_{A\times B},\id_{A\times B}\rangle\rangle \\
			&= \langle \pi_A \circ\id_{A\times B},\pi_B\circ\id_{A\times B}\rangle \\
			&=\langle\pi_A,\pi_B\rangle =\id_A\times\id_B =\times_\cc(\id_A,\id_B),
		\end{split}
		\]
		and for any object $A$ in $\cc$ we have that 
		\[
		\begin{split}
			\epsilon_{(A,A)}\circ \Delta_\cc\eta_A&=\epsilon_{(A,A)}\circ (\eta_A,\eta_A)=(\pi_A,\pi_A)\circ (\eta_A,\eta_A)=(\pi_A\circ\eta_A,\pi_A\circ\eta_A)\\
			&=(\pi_A\circ\langle\id_A,\id_A\rangle ,\pi_A\circ \langle \id_A,\id_A\rangle)=(\id_A,\id_A)=\Delta_\cc(\id_A),
		\end{split}
		\]
		hence $\Delta_\cc\dashv_\mathrm{s}\times_\cc$ is a semiadjunction.
	\end{invisible}
	We now study the properties of the semifunctor $\times_\cc$. Note that in general $\times_\cc$ is not faithful by Proposition \ref{prop:char-faithful-semifull} as the components $\epsilon_{(A,B)}=(\pi_A,\pi_B)$ of the counit are not epimorphisms for every $(A, B)$ in $\cc\times\cc$. For instance, in $\Set$ consider a nonempty set $B$ and the binary product $\emptyset\times B$, which is in particular a binary semi-product. Then, the map $\pi_B:\emptyset\times B=\emptyset\to B$ is the empty function from the emptyset into $B$, but it is not an epimorphism, cf. \cite[page 41]{Bor94}. As a consequence, $\times_\cc$ is not separable in general. By Proposition \ref{prop:char-faithful-semifull} we know that $\times_\cc$ is semifull if and only if the component $\epsilon_{(A,B)}=(\pi_A,\pi_B)$ is an $(A,B)$-semisplit-mono for every $(A, B)$ in $\cc\times\cc$, i.e. 
	if and only if for every $(A, B)$ in $\cc\times\cc$ there exists a morphism $\gamma_{(A,B)}=(\gamma_1,\gamma_2): (A,B)\to (A\times B, A\times B)$ in $\cc\times \cc$ such that $\gamma_{(A,B)}\circ \epsilon_{(A,B)}=\Delta_\cc\times_\cc\id_{(A,B)}$, i.e. such that $(\gamma_1\pi_A,\gamma_2\pi_B)=(\langle \pi_A,\pi_B\rangle , \langle\pi_A,\pi_B \rangle)$. Note that the latter condition is not satisfied in general. For instance, in $\Set$ consider again the binary product $\emptyset\times B =\emptyset$ of the emptyset by a nonempty set $B$. Since $\emptyset$ is the initial object in $\Set$ we have $\pi_\emptyset =\id_\emptyset$. Moreover, we have that $\langle\pi_\emptyset,\pi_B\rangle =\id_{\emptyset\times B}=\id_{\emptyset}$, so if $\times_\Set$ were semifull, then there would exist a morphism $\gamma_{(\emptyset,B)}=(\gamma_1,\gamma_2): (\emptyset,B)\to (\emptyset , \emptyset)$ in $\Set\times \Set$ such that $(\gamma_1\id_\emptyset,\gamma_2\pi_B)=(\id_{\emptyset} , \id_{\emptyset} ):\emptyset\to\emptyset$, but such $\gamma_2$ does not exist as there can be no map from a nonempty set to $\emptyset$.
\end{es}

\begin{es}\label{es:LH}
\textbf{The coidentifier category and $E^e$.}\\
 Given a category $\cc$ and an idempotent natural transformation $e:%
	\id_{\cc}\rightarrow \id_{\cc}$, let $\cc_{e}$ be the \emph{coidentifier category} defined as in \cite[Example 17]{FOPTST99}. Explicitly, $\mathrm{Ob}\left( \cc%
	_{e}\right) =\mathrm{Ob}\left( \cc\right) $ and $\Hom_{%
		\cc_{e}}\left( X,Y\right) =\Hom_{\cc}\left(
	X,Y\right) /\hspace{-2pt}\sim $, where $f\sim g$ if and only if $e_{Y}\circ f=e_{Y}\circ g.
	$ We denote by $\overline{f}$ the class of $f\in \Hom_{\cc%
	}\left( X,Y\right) $ in $\Hom_{\cc_{e}}\left( X,Y\right) $.
	There is a canonical functor $H:\cc\rightarrow \cc_{e}$
	acting as the identity on objects and as the canonical projection on
	morphisms. Consider the semifunctor $L:\cc_e\to\cc$ defined as the identity map on objects and by $[\bar{f}:X\to Y]\mapsto [e_Y\circ f:X\to Y]$ on morphisms. Note that it is really a semifunctor as $L\overline{\id }_X =e_X\circ\id_X=e_X\neq \id _{LX}=\id_X$ and it is well defined as $\bar{f}=\bar{g}$ if and only if $e_Y\circ f =e_Y\circ g$. In \cite[Theorem 3.1]{AB22-II} it is shown that $L,H$ form a semiadjunction $L\dashv_\mathrm{s} H$ with unit $\eta : \id_{\cc_e}\to HL$, $\eta_X =\overline{\id }_X : X\to HLX=X$, and counit $\epsilon : LH\to \id _{\cc}$, $\epsilon_Y:=e_Y : LHY=Y\to Y$. In particular, the identities $\epsilon_{LX}\circ L\eta_X =L\overline{\id}_X$ and $H\epsilon_Y\circ\eta_{HY}=H\id_Y$ hold true for any $X,Y$ in $\cc$. Since for every object $X\in\cc_e$, $HL(X)=X$, and for every morphism $\bar{f}$ in $\cc_e$, $HL\bar{f}=H(e_Y\circ f)=He_Y\circ Hf=\id_{HY}\circ\bar{f}=\bar{f}$, we have $HL=\id_{\cc_e}$, hence $\eta =\id_{\id_{\cc_e}}$. Thus, there exists a seminatural transformation $\nu=\id_{\id_{\cc_e}}:HL\to\id_{\cc_e}$ such that $\nu \circ \eta =\id_{\id_{\cc_e}}$ and $\eta \circ \nu =\id_{\id_{\cc_e}}=HL\id$. Then, by Theorem \ref{Th.Rafael} and Theorem \ref{thm:Raf-nat.full}, $L$ is a separable and naturally semifull semifunctor, whence semifully faithful.\par
	Now, by Proposition \ref{prop:E}, given an idempotent seminatural transformation $e=(e_X)_{X\in\cc}:\id_\cc\to\id_\cc$, we have the canonical semifunctor $E^e:\cc\to\cc$, defined by $X\mapsto X$, $[f:X\to Y]\mapsto f\circ e_X=e_Y\circ f$, for any object $X\in \cc$ and for any morphism $f$ in $\cc$. Note that $E^e=LH$. 
	Indeed, $LH(X)=X$ and $LH(f)=L(\bar{f})=f\circ e_X$. We know that $L$ is semifully faithful and $H$ is naturally full, hence in particular they are naturally semifull semifunctors. Thus, by Proposition \ref{prop:comp} it follows that $E^e$ is naturally semifull, whence semiseparable. 
	By Proposition \ref{prop:idempotent} its associated idempotent $\alpha:\id_{\cc%
	}\rightarrow \id_{\cc}$ such that $E^e\alpha=E^e\id$ is given for any $X$ in $\cc$ by $\alpha_X=\p_{X,X}(E^e\id_X)=\p_{X,X}(e_X):X\to X$, and then we have that $\alpha_{X}=\p_{X,X}(e_X)=\p_{X,X}(E^e e_X\circ e_X)=e_X\circ \p_{X,X}(e_X)=e_X\circ\alpha_X=E^e\alpha_X=E^e\id_X=e_X$.
	Moreover, by Proposition \ref{prop:sep-nat.full}, $E^e$ results to be separable if and only if it is faithful. 
	By Proposition \ref{prop:E^e} we know that $E^e$ is separable if and only if $e_X=\id_X$. 
	Thus, $E^e$ is (semifully) faithful if and only if it is the identity functor on $\cc$. 
\end{es}
\begin{es}\label{es:cost} \textbf{The constant semifunctor.}\\ Cf. \cite[Subsection 2.3]{Ho93}. Let $\textbf{1}$ be the category with only one object $1$ and an identity morphism $\id_1$. 
	Any semifunctor $F:\textbf{1}\to\cc$ determines an arrow $F(\id_1):F(1)\to F(1)$ which is idempotent. 
	Conversely, any idempotent arrow $e:X\to X$ in $\cc$ defines a semifunctor $F^e:\textbf{1}\to\cc$, given by $F^e(1)=X$, $F^e(\id_1)=e$. 
	In particular we have the functor $F^{\id_X}:\textbf{1}\to\cc$ given by $F^{\id_X}(1)=X$, $F^{\id_X}(\id_1)=\id_X$. Hence $F^e=E^e\circ F^{\id_X}$, where $E^e:\cc\to\cc$ is the canonical semifunctor.
	Now, we show that $F^e:\textbf{1}\to\cc$ is a separable semifunctor that is not naturally semifull in general. Note that $\Hom_\textbf{1}(1,1)=\{\id_1\}$. Consider the associated natural transformation
	$\f^{F^e}_{1,1}:\Hom_\textbf{1}(1,1)\to\Hom_{\cc}(F^e(1),F^e(1))$, $\f^{F^e}_{1,1}(\id_1)=F^e(\id_1)=e$, and the map $\p^{F^e}_{1,1}:\Hom_{\cc}(F^e(1 ), F^e(1))\to \Hom_\textbf{1}(1,1)$, given by $\p^{F^e}_{1,1}(f)=\id_1$, for any $f:F^e(1 )\to F^e(1)$ in $\cc$. We have that $\p^{F^e}_{1,1}$ is a natural transformation as for any $f:F^e(1 )\to F^e(1)$ in $\cc$, $\p^{F^e}_{1,1}(F^e\id_1\circ f \circ F^e\id_1)=\p^{F^e}_{1,1}(e\circ f \circ e)=\id_1=\id_1\circ\p^{F^e}_{1,1}(f)\circ\id_1$. Moreover, $(\p^{F^e}_{1,1}\circ \f^{F^e}_{1,1})(\id_1)=\p^{F^e}_{1,1}(e)=\id_1$, hence $F^e$ is separable, and in particular semiseparable. Thus, $F^e$ is naturally semifull if and only if it is semifull. If $F^e$ were semifull, then it would follow that $e=e\circ f\circ e$	for any $f:F^e(1 )\to F^e(1)$ in $\cc$, and this does not happen in general (we will see an instance of this in Example \ref{monoid}). 
	More generally, given two categories $\cc$ and $\dd$ and a fixed idempotent arrow $e_D:D\to D$ in $\dd$ we can consider the \emph{constant semifunctor} $K:\cc\to\dd$ given by $K(C)=D$, $K(f)=e_D$, for every object $C\in\cc$ and for every morphism $f:C\to C'$ in $\cc$. It is clearly not faithful and not even semifull. In fact, for any morphism $g:K(C)=D\to K(C')=D$ in $\dd$ we have that $K\id_{C'}\circ g\circ K\id_C =e_D\circ g\circ e_D$. If $K$ were semifull then there would exist a morphism $f:C\to C'$ such that $e_D=K(f)=e_D\circ g\circ e_D$, which is not true in general.
\end{es}
\begin{es}
	Let $F:\cc\to\dd$ be a semifunctor and let $e=(e_X)_{X\in\cc}:\id_\cc\to\id_\cc$ be an idempotent natural transformation different from $\id_{\id_\cc}$.
	Consider $S_{F,e}:\cc\to\dd$, $X\mapsto FX$, assigning to any morphism $f:X\to Y$ in $\cc$ the morphism $F(e_Y\circ f):FX\to FY$ in $\dd$. Since for any $f:X\to Y$, $g:Y\to Z$ in $\cc$, we have that $S_{F,e}(g\circ f)=F(e_Z\circ g\circ f)=F(e_Z\circ e_Z\circ g\circ f)=F(e_Z\circ g\circ e_Y\circ  f)=F(e_Z\circ g)\circ F(e_Y\circ f)=S_{F,e}(g)\circ S_{F,e}(f)$, and $S_{F,e}(\Id_X)=F(e_X\circ\id_X)=F(e_X)$, then $S_{F,e}$ is a semifunctor. Note that $S_{F,e}=F\circ L\circ H=F\circ E^e$, where $E^e$ is the canonical semifunctor and $L:\cc_e\to\cc$, $H:\cc\to\cc_e$ are given as in Example \ref{es:LH}. Moreover, $S_{F,e}$ is not faithful as $S_{F,e}(e_X)=F(e_X\circ e_X)=F(e_X)=F(e_X\circ\id_X)=S_{F,e}(\id_X)$ but $e_X\neq \id_X$ for some $	X\in\cc$. As a consequence, $S_{F,e}$ is not separable. If we assume that the given semifunctor $F$ is naturally semifull, then by Proposition \ref{prop:comp} (i) $S_{F,e}$ results to be naturally semifull, since the canonical semifunctor $E^e$ is naturally semifull as seen in Example \ref{es:LH}. 
\end{es}

\begin{es}\label{es:morph-ring}\textbf{Semifunctors associated with a morphism of rings.}\\
	Let $R,S$ be unital rings and let $R\text{-}\mathrm{Mod}$, $S\text{-}\mathrm{Mod}$ be the categories of left $R$-modules and left $S$-modules, respectively. Consider a morphism of rings $\varphi :R\to S$, that induces
	\begin{itemize}
		\item the restriction of scalars functor $\varphi_* : S\text{-}\mathrm{Mod}\rightarrow R\text{-}\mathrm{Mod}$;
		\item the extension of scalars functor $\varphi^*:= S\otimes_{R}(-):R\text{-}\mathrm{Mod}\rightarrow S\text{-}\mathrm{Mod}$.
	\end{itemize}
	These functors form an adjunction $\varphi^*\dashv\varphi_*$ with unit $\eta$ and counit $\epsilon$ defined by
	$$\eta_M = \varphi\otimes_R M : M\to S\otimes_R M,\, m\mapsto 1_{S}\otimes_R m\quad\text{and}\quad \epsilon_N: S\otimes_R N\to N,\, s\otimes_R n\mapsto sn,$$
	for every $M\in R\text{-}\mathrm{Mod}$ and $N\in S\text{-}\mathrm{Mod}$, respectively. We recall that 
	\begin{itemize} 
	\item $\varphi_*$ is separable if and only if $S/R$ is separable, i.e. the multiplication $m_S:S\otimes_R S\to S$, $s\otimes_R s'\mapsto ss'$ splits as an $S$-bimodule map, see \cite[Proposition 1.3]{NVV89};
	\item $\varphi_*$ is naturally full if and only if it is full, see \cite[Proposition 3.1 (1)]{AMCM06};
	\item $\varphi^*$ is separable if and only if $\varphi$ is a split-mono as an $R$-bimodule map, i.e. if there is $\pi\in {}_{R}\Hom_{R}(S,R)$ such that $\pi\circ \varphi=\id$, see \cite[Proposition 1.3]{NVV89}; 
	\item $\varphi^*$ is naturally full if and only if  $\varphi$ is a split-epi as an $R$-bimodule map, i.e. if there is $\pi\in {}_{R}\Hom_{R}(S,R)$ such that $ \varphi\circ \pi=\id$, see \cite[Proposition 3.1 (2)]{AMCM06}.
	\end{itemize}
	Given an idempotent (semi)natural transformation $e=(e_X)_{X\in R\text{-}\mathrm{Mod}}:\id_{R\text{-}\mathrm{Mod}}\to\id_{R\text{-}\mathrm{Mod}}$, by composing $\varphi^*$ with the canonical semifunctor $E^e:R\text{-}\mathrm{Mod}\to R\text{-}\mathrm{Mod}$, we get the semifunctor 
	$$\varphi^*_e:=\varphi^*\circ E^e: R\text{-}\mathrm{Mod}\rightarrow S\text{-}\mathrm{Mod},\quad M\mapsto S\otimes_R M,\quad f\mapsto S\otimes_R fe_M$$
	for any $f:M\to M'$ in $R\text{-}\mathrm{Mod}$. Consider the semifunctor $$\varphi_{*}^e:=E^e\circ\varphi_*:S\text{-}\mathrm{Mod}\rightarrow R\text{-}\mathrm{Mod},\quad N\mapsto \varphi_*(N), \quad g\mapsto e_{\varphi_*(N')}\circ\varphi_*(g),$$ for any $g:N\to N'$ in $S\text{-}\mathrm{Mod}$. By Corollary \ref{cor:F'G'} we know that $\varphi^*_e$ and $\varphi_{*}^e$
	form a semiadjunction $\varphi^*_e\dashv_\mathrm{s}\varphi_{*}^e:S\text{-}\mathrm{Mod}\rightarrow R\text{-}\mathrm{Mod}$, with unit $\eta^e:=E^e\eta _{E^e}\circ e$ and counit $\epsilon^e:=\epsilon\circ \varphi^*e_{\varphi_*}$. By Theorem \ref{Th.Rafael}, $\varphi_*^e$ is separable if and only if $\epsilon^e$ is a natural split-epi, i.e. there exists a seminatural transformation $\gamma : \Id_{S\text{-}\mathrm{Mod}}\rightarrow \varphi^*_e\varphi_*^e$ such that $\epsilon^e\circ\gamma = \Id_{\id_{S\text{-}\mathrm{Mod}}}$, while $\varphi^*_e$ is separable if and only if $\eta^e$ is a natural split-mono, i.e. there exists a seminatural transformation $\nu : \varphi_*^e\varphi^*_e\rightarrow \Id_{R\text{-}\mathrm{Mod}}$ such that
		$\nu\circ\eta^e = \Id_{\id_{R\text{-}\mathrm{Mod}}}$. In particular, if $\varphi_*^e$ is separable, then by Lemma \ref{lem:comp-sep} (ii) $\varphi_*$ is separable. If $\varphi^*_e$ is separable, then by Remark \ref{rmk:F'G'sep} it results to be $\varphi^*_e=\varphi^*$.\par 
	In Proposition \ref{prop:ringmorph-natsemifull} we study the natural semifullness of $\varphi_*^e$ and $\varphi^*_e$, for which the following lemma will be useful.
	\begin{lem}\label{lem:eSlin}
		Let $\varphi:R\to S$ be a morphism of rings and let $e=(e_X)_{X\in R\text{-}\mathrm{Mod}}:\id_{R\text{-}\mathrm{Mod}}\to\id_{R\text{-}\mathrm{Mod}}$ be an idempotent (semi)natural transformation. Consider the semifunctors $\varphi^*_e$, $\varphi_*^e$ as above. Then,
		\begin{itemize} 
		\item[(1)] $1_S=e_{\varphi_*(S)}(\varphi(e_R(1_R)))$ holds true if, and only if, $\varphi (r)=e_{\varphi_*(S)}(\varphi(e_R(r)))$ holds true for every $r\in R$ if, and only if, 
		\begin{equation}\label{eq:RRsemisplitepi}
		r_S^{-1}\circ\varphi=\varphi_{*}^e\varphi^*_e\id_R\circ r_S^{-1}\circ\varphi
		\end{equation}
	holds true, where $r_S:S\otimes_R R\to S$ is the canonical isomorphism $s\otimes_R r\mapsto s\varphi(r)$.
		\item[(2)] $e_{\varphi_*(S)}$ is a left $S$-module morphism if, and only if, $e_{\varphi_*(N)}$ is a left $S$-module morphism for every $N\in S\text{-}\mathrm{Mod}$. The latter condition means that there is an idempotent natural transformation $\alpha:\id_{S\text{-}\mathrm{Mod}}\to\id_{S\text{-}\mathrm{Mod}}$ such that $\varphi_*\alpha=e\varphi_*$. In this case, let $E^\alpha:S\text{-}\mathrm{Mod}\to S\text{-}\mathrm{Mod}$ be the canonical semifunctor attached to $\alpha$. Then, $E^\alpha\circ\varphi^*=\varphi^*\circ E^e$ and $\varphi_*\circ E^\alpha=E^e\circ\varphi_*$.
		\end{itemize}
	\end{lem}
\begin{proof}
(1). Assume that $1_S=e_{\varphi_*(S)}(\varphi(e_R(1_R)))$. Then, since $e_R$, $e_{\varphi_*(S)}$ are left $R$-module morphisms, we have that $e_{\varphi_*(S)}(\varphi(e_R(r)))=e_{\varphi_*(S)}(\varphi(re_R(1_R)))=e_{\varphi_*(S)}(\varphi(r)\varphi(e_R(1_R)))=\varphi(r)e_{\varphi_*(S)}(\varphi(e_R(1_R)))=\varphi(r)1_S=\varphi(r)$, for every $r\in R$. The converse implication is trivially satisfied. Now we show that the latter condition is also equivalent to \eqref{eq:RRsemisplitepi}. Indeed, we have that $r_S\varphi_{*}^e\varphi^*_e\id_R r_S^{-1}\varphi(r)=r_S\varphi_*^e(S\otimes_R e_R)(\varphi(r)\otimes_R 1_R)=r_Se_{\varphi_*(S\otimes_RR)}\varphi_*(S\otimes_Re_R)(1_S\otimes_R r)=e_{\varphi_*(S)}r_S(1_S\otimes_R e_R(r))=e_{\varphi_*(S)}(\varphi(e_R(r)))$, for every $r\in R$.\\	
(2). If $e_{\varphi_*(N)}$ is a left $S$-module morphism for every $N\in S\text{-}\mathrm{Mod}$, then clearly $e_{\varphi_*(S)}$ is a left $S$-module morphism. On the other hand, assume that $e_{\varphi_*(S)}$ is a left $S$-module morphism. Consider a left $S$-module $N$ and the left $S$-module morphism $f_n:S\to N$, $s\mapsto sn$, for $n\in N$. By naturality of $e$, we have that $e_{\varphi_*(N)}(sn)=(e_{\varphi_*(N)}\circ \varphi_*(f_n))(s)=(\varphi_*(f_n)\circ e_{\varphi_*(S)})(s)=e_{\varphi_*(S)}(s)n$. Since $e_{\varphi_*(S)}$ is left $S$-linear, we get $e_{\varphi_*(S)}(s)n=se_{\varphi_*(S)}(1_S)n=se_{\varphi_*(N)}(n)$, so $e_{\varphi_*(N)}$ is a left $S$-module morphism for every $N\in S\text{-}\mathrm{Mod}$, which means that there is $\alpha_N:N\to N$ in $S\text{-}\mathrm{Mod}$ such that $\varphi_*(\alpha_N)= e_{\varphi_*(N)}$. For any left $S$-module morphism $f:N\to N'$ we have that $\varphi_*(\alpha_{N'}\circ f)=\varphi_*(\alpha_{N'})\circ\varphi_*(f)=e_{\varphi_*(N')}\circ \varphi_*(f)=\varphi_*(f)\circ e_{\varphi_*(N)}=\varphi_*(f)\circ\varphi_*(\alpha_N)=\varphi_*(f\circ\alpha_N)$, hence $\alpha_{N'}\circ f=f\circ\alpha_N$ as $\varphi_*$ is faithful. Moreover, for any $N\in S\text{-}\mathrm{Mod}$, $\varphi_*(\alpha_N\circ\alpha_N)=\varphi_*(\alpha_N)\circ\varphi_*(\alpha_N)=e_{\varphi_*(N)}\circ e_{\varphi_*(N)}=e_{\varphi_*(N)}=\varphi_*(\alpha_N)$, thus $\alpha_N^2=\alpha_N$, so we obtain an idempotent natural transformation $\alpha:\id_{S\text{-}\mathrm{Mod}}\to\id_{S\text{-}\mathrm{Mod}}$. Consider the canonical semifunctor $E^\alpha:S\text{-}\mathrm{Mod}\to S\text{-}\mathrm{Mod}$ attached to $\alpha$. We show that $E^\alpha\circ\varphi^*=\varphi^*\circ E^e$ and $\varphi_*\circ E^\alpha =E^e\circ\varphi_*$. In fact, the semifunctor $E^\alpha\circ\varphi^*: R\text{-}\mathrm{Mod}\to S\text{-}\mathrm{Mod}$ maps $M\mapsto S\otimes_RM$, $[f:M\to M'] \mapsto \alpha_{S\otimes_R M'}\circ (S\otimes_R f)= (S\otimes_R f)\circ \alpha_{S\otimes_RM}$. By naturality of $e$, we have that $e_{\varphi_*(S\otimes_R M)}(1_S\otimes_R m)=1_S\otimes_R e_M(m)$. Since $e_{\varphi_*(S\otimes_RM)}$ is a left $S$-module morphism, we get that $e_{\varphi_*(S\otimes_RM)}(s\otimes_R m)=s e_{\varphi_*(S\otimes_RM)}(1_S\otimes_Rm)=s(1_S\otimes_R e_M(m))=s\otimes_R e_M(m)=\varphi_*(S\otimes_R e_M)(s\otimes_R m)$, so $e_{\varphi_*(S\otimes_RM)}=\varphi_*(S\otimes_R e_M)$, i.e. $\varphi_*(\alpha_{S\otimes_RM})=\varphi_*(S\otimes_R e_M)$, i.e. $\alpha_{S\otimes_RM}=S\otimes_R e_M$ as $\varphi_*$ is faithful. Thus, $E^\alpha\circ\varphi^*=\varphi^*\circ E^e$. The semifunctor $\varphi_*\circ E^\alpha : S\text{-}\mathrm{Mod}\to R\text{-}\mathrm{Mod}$ maps $N\mapsto \varphi_*(N)$, $[f:N\to N']\mapsto \varphi_*(f\circ\alpha_N)=\varphi_*(f)\circ\varphi_*(\alpha_N)$. Note that $\varphi_*(f)\circ\varphi_*(\alpha_N)=\varphi_*(f)\circ e_{\varphi_*(N)}=e_{\varphi_*(N')}\circ \varphi_*(f)$, hence $\varphi_*\circ E^\alpha =E^e\circ\varphi_*$.
\end{proof}
	\begin{prop}\label{prop:ringmorph-natsemifull}
		Let $\varphi:R\to S$ be a morphism of rings and let $e=(e_X)_{X\in R\text{-}\mathrm{Mod}}:\id_{R\text{-}\mathrm{Mod}}\to\id_{R\text{-}\mathrm{Mod}}$ be an idempotent (semi)natural transformation. \begin{itemize}
			\item[(1)] If the semifunctor $\varphi_*^e$ is semifull and $e_{\varphi_*(S)}$ is a left $S$-module morphism, then $\varphi_*^e$ is naturally semifull. 
			\item[(2)] If the semifunctor $\varphi^*_e$ is naturally semifull, then there is $\psi$ in ${}_R\Hom_R(S,R)$ such that 
			\begin{equation}\label{eq:varphi-semisplitepi}
				\varphi\circ \psi=r_S\circ \varphi_*^e\varphi^*_e\Id_R\circ r_S^{-1},
			\end{equation}
			i.e. such that $r_S^{-1}\circ\varphi\circ \psi\circ r_S=\varphi_*^e\varphi^*_e\id_R, $ so $r_S^{-1}\circ\varphi:R\to \varphi_*^e\varphi^*_eR$ is an $R$-semisplit-epi (cf. Section \ref{sect:semisplit-morph}) as an $R$-bimodule map; if in addition \eqref{eq:RRsemisplitepi}
		holds true, then $r_S^{-1}\circ\varphi$ is an $(R,R)$-semisplit-epi.
		
		On the other hand, assume that $e_{\varphi_*(S)}$ is a left $S$-module morphism. If there is $\psi$ in ${}_R\Hom_R(S,R)$ such that \eqref{eq:varphi-semisplitepi} holds true and $1_S=e_{\varphi_*(S)}(\varphi(e_R(1_R)))$ is satisfied, i.e. $r_S^{-1}\circ\varphi$ is an $(R,R)$-semisplit-epi as an $R$-bimodule map through $\psi\circ r_S$, then $\varphi^*_e$ is naturally semifull.			
		\end{itemize}
	\end{prop}
	\begin{proof}
		(1). Assume that $\varphi_*^e$ is semifull and $e_{\varphi_*(S)}$ is a morphism of left $S$-modules, i.e $e_{\varphi_*(N)}$ is a morphism of left $S$-modules for any $N\in S\text{-}\mathrm{Mod}$ by Lemma \ref{lem:eSlin} (2). By Proposition \ref{prop:char-faithful-semifull}, for every $N\in S\text{-}\mathrm{Mod}$, $\epsilon^e_N$ is an $N$-semisplit-mono, i.e. there exists $\gamma^e_N : N\rightarrow \varphi^*_e\varphi_*^e N=S\otimes_R\varphi_*(N)$ in $S\text{-}\mathrm{Mod}$ such that $\gamma^e_N\circ\epsilon^e_N = \varphi^*_e\varphi_*^e\id_N$. Then, $\gamma^e_N\circ\epsilon_N\circ (S\otimes_R e_{\varphi_*(N)}) =\gamma^e_N\circ\epsilon^e_N =\varphi^*_e\varphi_*^e\id_N =S\otimes_R (e_{\varphi_*(N)}\circ e_{\varphi_*(N)}\circ\varphi_*(\id_N))=S\otimes_Re_{\varphi_*(N)}$. 
		Thus, for every $n\in N$, we have that $\gamma^e_N(e_{\varphi_*(N)}(n))=(\gamma^e_N\circ\epsilon_N)(1_S\otimes_Re_{\varphi_*(N)}(n))=(S\otimes_Re_{\varphi_*(N)})(1_S\otimes_Rn)=1_S\otimes_Re_{\varphi_*(N)}(n)$, so $\gamma^e_N(e_{\varphi_*(N)}(n))=1_S\otimes_Re_{\varphi_*(N)}(n)$, for every $n\in N$. 
		By Lemma \ref{lem:eSlin} (2) there is an idempotent natural transformation $\alpha:\id_{S\text{-}\mathrm{Mod}}\to\id_{S\text{-}\mathrm{Mod}}$ such that $\varphi_*\alpha=e\varphi_*$. So, we have that $\gamma^e_N(\alpha_N(n))=\gamma^e_N(\varphi_*(\alpha_N)(n))=\gamma^e_N(e_{\varphi_*(N)}(n))=1_S\otimes_Re_{\varphi_*(N)}(n)=1_S\otimes_R\varphi_*(\alpha_{N})(n)=1_S\otimes_R\alpha_{N}(n)$, for every $n\in N$. Now, although $\gamma^e_N$ is not natural a priori, still we can show that $\overline{\gamma}^e:\id_{S\text{-}\mathrm{Mod}}\to\varphi^*_e\varphi_*^e$, defined for every $N\in S\text{-}\mathrm{Mod}$ by $\overline{\gamma}^e_N:=\gamma^e_N\alpha_N :N\to \varphi^*_e\varphi_*^eN$, $n\mapsto 1_S\otimes_R \alpha_{N}(n)$, is a natural transformation. In fact, for any left $S$-module morphism $f:N\to N'$ we have that $(\overline{\gamma}^e_{N'}\circ f)(n)=1_S\otimes_R\alpha_{N'}(f(n))=1_S\otimes_R f(\alpha_N(n))=(S\otimes_R\varphi_*(f)\varphi_*(\alpha_N))(1_S\otimes_R\alpha_N(n))=(1_S\otimes_R\varphi_*(f)e_{\varphi_*(N)})(1_S\otimes_R\alpha_N(n))=(\varphi^*_e\varphi_*^e(f)\circ\overline{\gamma}^e_N)(n)$, for every $n\in N$. Moreover, $(\overline{\gamma}^e_N\circ\epsilon^e_N)(s\otimes_Rn)=(\gamma^e_N\circ\alpha_N\circ\epsilon_N)(s\otimes_Re_{\varphi_*(N)}(n))=(\gamma^e_N\circ\alpha_N\circ\epsilon_N)(s\otimes_R\varphi_*(\alpha_N)(n))=\gamma^e_N(\alpha_N(s\alpha_N(n)))=s\gamma^e_N(\alpha_N(n))=s\otimes_R\alpha_N(n)=s\otimes_Re_{\varphi_*(N)}(n)=\varphi^*_e\varphi_*^e\id_N(s\otimes_Rn)$, for every $s\in S$, $n\in N$, so $\varphi_*^e$ is naturally semifull by Theorem \ref{thm:Raf-nat.full} (2).
			%
		   
		(2). Assume that $\varphi^*_e$ is naturally semifull. Then, by Theorem \ref{thm:Raf-nat.full} (1), there exists a seminatural transformation $\nu^e : \varphi_*^e\varphi^*_e\rightarrow \Id_{ R\text{-}\mathrm{Mod}}$ such that $\eta^e\circ\nu^e =\varphi_*^e\varphi^*_e\Id $. Consider the map $\psi$ in ${}_R\Hom_R(S,R)$ given by $\psi(s)=(e_R\circ\nu^e_R\circ r_S^{-1})(s)=(e_R\circ\nu^e_R)(s\otimes_R1_R)$, where $r_S:S\otimes_R R\to S$ is the canonical isomorphism $s\otimes_R r\mapsto s\varphi(r)$. The right $R$-linearity of $\psi$ follows from the naturality of $\nu^e$ and $e$. Indeed, consider the left $R$-module map $f_r:R\to R$, $r'\mapsto r'r$. For any $s\in S$, $r\in R$, we have that 
		\[\begin{split}
			\psi(s)r&=e_R(\nu^e_R(s\otimes_R 1_R))r=f_r (e_R(\nu^e_R(s\otimes_R 1_R)))=e_R(f_r(\nu^e_R(s\otimes_R1_R)))\\
			&=(e_R\nu^e_R\varphi_*^e\varphi^*_e(f_r))(s\otimes_R1_R)=(e_R\nu^e_R e_{\varphi_*(S\otimes_RR)}\varphi_*(S\otimes_R f_re_R))(s\otimes_R1_R)\\
			&=(e_R\nu^e_R e_{\varphi_*(S\otimes_RR)}\varphi_*(S\otimes_R e_Rf_r))(s\otimes_R1_R)\\&=(e_R\nu^e_R e_{\varphi_*(S\otimes_RR)}\varphi_*(S\otimes_R e_R)\varphi_*(S\otimes_R f_r))(s\otimes_R1_R)\\
			&=(e_R\nu^e_R\varphi_*^e\varphi^*_e(e_R)\varphi_*(S\otimes_R f_r))(s\otimes_R1_R)=(e_Re_R\nu^e_R\varphi_*(S\otimes_R f_r))(s\otimes_R1_R)\\
			&=(e_R\nu^e_R \varphi_*(S\otimes_Rf_r))(s\otimes_R1_R)=e_R(\nu^e_R(s\otimes_R r))=e_R(\nu^e_R(sr\otimes_R1_R))=\psi(sr).
		\end{split}\]	
		Then, for every $s\in S$, we have that
		\[
		\begin{split}
			(\varphi\circ\psi)(s)&=(r_S\circ\eta_R\circ e_R\circ\nu^e_R)(s\otimes_R1_R)=(r_S\circ\eta^e_R\circ \nu^e_R)(s\otimes_R1_R)\\&=(r_S\circ\varphi_*^e\varphi^*_e\Id_R)(s\otimes_R1_R)=(r_S\circ\varphi_*^e\varphi^*_e\Id_R\circ r_S^{-1})(s),
		\end{split}	
		\]
		thus $r_S^{-1}\circ\varphi\circ \psi\circ r_S=\varphi_*^e\varphi^*_e\Id_R$, so $r_S^{-1}\circ\varphi$ is an $R$-semisplit-epi. Moreover, if \eqref{eq:RRsemisplitepi} is satisfied, then $r_S^{-1}\circ\varphi$ is an $(R,R)$-semisplit-epi. 
	
		Conversely,	assume that $1_S=e_{\varphi_*(S)}(\varphi(e_R(1_R)))$ (which is equivalent to \eqref{eq:RRsemisplitepi} by Lemma \ref{lem:eSlin} (1)) and that there is $\psi$ in ${}_R\Hom_R(S,R)$ such that $\varphi\circ \psi=r_S\circ\varphi_*^e\varphi^*_e\Id_R\circ r_S^{-1}$, i.e. $r_S^{-1}\circ\varphi$ is an $(R,R)$-semisplit-epi as an $R$-bimodule map through $\psi\circ r_S$. 
		Define for every $M$ in $R\text{-}\mathrm{Mod}$, $\nu^e_M:\varphi_*^e\varphi^*_e M=\varphi_*(S\otimes_RM)\to M$ by $\nu^e_M(s\otimes_R m)=\psi(s)e_M(m)$, for any $m\in M$ and $s\in S$. Note that for any $M\in R\text{-}\mathrm{Mod}$, $\nu^e_M$ is a morphism of left $R$-modules as so is $\psi$. We get a natural transformation $\nu^e:\varphi_*^e\varphi^*_e\to\id_{R\text{-}\mathrm{Mod}}$ as for any morphism $f:M\to M'$ in $R\text{-}\mathrm{Mod}$ we have that 
		\[\begin{split}
			(f\circ\nu^e_M)(s\otimes_Rm)&=f(\psi(s)e_M(m))=\psi(s)f(e_M(m))=\psi(s)f(e_M(e_M(m)))\\
			&=\psi(s)e_{M'}(f(e_M(m)))=\psi(s)e_{M'}(e_{M'}(f(e_M(m))))\\
			&=e_{M'}(\psi(s)e_{M'}(f(e_M(m))))=e_{M'}(\nu^e_{M'}(s\otimes_R f(e_M(m))))\\
			&=(\nu^e_{M'}\circ e_{\varphi_*(S\otimes_RM')}\circ (S\otimes_Rfe_M))(s\otimes_R m)\\&=(\nu^e_{M'}\circ \varphi_*^e\varphi^*_e f)(s\otimes_R m),
		\end{split}\] 
	for every $m\in M$, $s\in S$. Note that $\varphi\psi(s)=r_S\varphi_*^e\varphi^*_e\Id_R(s\otimes_R1_R)=r_Se_{\varphi_*(S\otimes_RR)}(s\otimes_R e_R(1_R)) =e_{\varphi_*(S)}r_S(s\otimes_R e_R(1_R))=e_{\varphi_*(S)}(s\varphi(e_R(1_R)))$. If $e_{\varphi_*(S)}$ is a left $S$-module morphism, we have that $\varphi\psi(s)=se_{\varphi_*(S)}(\varphi(e_R(1_R)))=s1_S=s$, so $\varphi\circ \psi=\id_S$ in this case. Then, 
		\[
		\begin{split}
			\eta^e_M\nu^e_M(s\otimes_Rm)&=\eta^e_M(\psi(s)e_M(m))= \psi(s)\eta^e_M(e_M(m))=\psi(s)\eta_M(e_M(e_M(m)))\\
			&=\psi(s)e_{\varphi_*(S\otimes_RM)}(\eta_M(e_M(m)))=\psi(s)e_{\varphi_*(S\otimes_RM)}(1_S\otimes_Re_M(m))\\
			&=e_{\varphi_*(S\otimes_RM)}(\psi(s)1_S\otimes_Re_M(m))=e_{\varphi_*(S\otimes_RM)}(\varphi\psi(s)\otimes_Re_M(m))\\
			&=e_{\varphi_*(S\otimes_RM)}(s\otimes_Re_M(m))=\varphi_*^e\varphi^*_e\Id_M(s\otimes_R m),
		\end{split}
		\]
		for every $m\in M$, $s\in S$, so $\varphi^*_e$ is naturally semifull. Alternatively, we observe that, since $\varphi\circ\psi=\id_S$, by \cite[Proposition 3.1 (2)]{AMCM06} $\varphi^*$ is naturally full, so since $E^e$ is naturally semifull (cf. Example \ref{es:LH}), $\varphi^*_e$ results to be naturally semifull also by Proposition \ref{prop:comp} (i).
		\qedhere
	\end{proof}
\end{es}

\begin{es}\label{monoid}\textbf{Semifunctor on a monoid.}\\
	Let $(M,\cdot_M)$ be a monoid. It can be viewed as a category with a single object, denoted by $*$ (see \cite[Example 1.2.6.d]{Bor94}). Arrows $a:*\to *$ are the elements of the monoid, which are closed under composition and the identity element $1_M$ is the identity arrow $1_M:*\to *$. A monoid homomorphism $f:M\to N$ is a functor, as it preserves compositions of arrows and the identity arrow; it sends the unique object $*$ of $M$ into the unique object $\star$ of $N$. A natural transformation $\alpha=(\alpha_*): f\to g$ between functors $f,g:M\to N$ is an arrow $\alpha_*:f(*)\to g(*)$ in the category $N$ (i.e. an element of the monoid $N$) such that, for all elements $a$ of $M$, $g(a)\cdot_N\alpha_* =\alpha_*\cdot_N f(a)$, where $\cdot_N$ is the composition law of $N$.\par
	Consider the direct product $M\times M$ of $M$ by itself with componentwise multiplication $(a,b)\cdot_{M\times M}(a',b')=(a\cdot_M a',b\cdot_M b')$, which we also denote by $(a,b)(a',b')=(aa',bb')$ for shortness. Let $e\neq 1_M$ be an idempotent element of $M$, i.e. $e^2=e$. Consider the map $f_e=(e,-):M\to M\times M$, $f_e(b)=(e,b)$, which is a semigroup homomorphism. Indeed, for any $b,c\in M$, $f_e(bc)=(e, bc)=(ee, bc)=(e, b)(e,c)=f_e(b)f_e(c)$, but it does not preserve the unit, as $f_e(1_M)=(e,1_M)\neq (1_M,1_M) =1_{M\times M}$. 
	Moreover, if we view $M\times M$ as the product category with a single object $\star$ and morphisms given by the pairs $(m,n):\star\to \star$, where $m,n:*\to *$ are elements of $M$, and whose composition is that induced by the multiplication of $M\times M$, then $f_e:M\to M\times M$, $f_e(*)=\star$, $f_e(b)=(e,b)$ results to be a semifunctor as it preserves compositions, but not the identity arrow $1_M:*\to *$. Consider the associated natural transformation
	$\f^{f_e}_{*,*}:\Hom_M(*,*)\to\Hom_{M\times M}(\star,\star)$, $\f^{f_e}_{*,*}(m)=f_e(m)=(e,m)$, and the map $\p^{f_e}_{\star,\star}:\Hom_{M\times M}(\star, \star)\to \Hom_M(*,*)$, $\p^{f_e}_{\star,\star}((m,n))=n$, which is a natural transformation as for any $m,n,b, c\in M$, $\p^{f_e}_{\star,\star}(f_e(b)\cdot_{M\times M} (m,n)\cdot_{M\times M} f_e(c))=\p^{f_e}_{\star,\star}((e,b)(m,n)(e,c))=\p^{f_e}_{\star,\star}((eme,bnc))=bnc=b\cdot_M\p^{f_e}_{\star,\star}((m,n))\cdot_M c$. Then, for any $m\in M$, we have $(\p^{f_e}_{\star,\star}\circ\f^{f_e}_{*,*})(m)=\p^{f_e}_{\star,\star}((e,m))=m$, hence $f_e$ is separable. Now, we show that the semifunctor $f_e$ is not in general semifull, and hence not even naturally semifull. If $f_e$ were semifull, then for any $(m,n)$ in $M\times M$ there would exist an element $a\in M$ such that $(e,a)=(e,1_M)(m,n)(e,1_M)=(eme,n)$. But in general it is not true that $(e,a)=(eme,n)$. For instance, consider the commutative monoid $\{x,1,e\}$ with the following multiplication $\cdot$ laws
	\begin{displaymath}
		x\cdot x=e,\quad x\cdot 1=x,\quad x\cdot e=x,\quad 1\cdot 1=1,\quad 1\cdot e=e,\quad e\cdot e=e.
	\end{displaymath}
	Fix the idempotent element $e$. Then $e\cdot x\cdot e=e\cdot (x\cdot e)=e\cdot x=x$ does not result to be equal to $e$. Thus, $f_e$ is not semifull.
\end{es}
\begin{es}\label{es:ring}\textbf{Semifunctor on a ring.}\\
	A category $\cc$ is \emph{preadditive} (also called an \emph{Ab-category}) if for any pair of objects $X,Y\in\cc$, the hom-set $\Hom_\cc(X,Y)$ is an additive abelian group such that the composition map $\Hom_\cc(Y,Z)\times\Hom_\cc(X,Y)\to\Hom_\cc(X,Z)$ is bilinear. 
	Let $A$ be a unital ring. Then, $A$ is a preadditive category (denoted by the same symbol $A$) with a single object $*$ and $\Hom_A(*,*)=A$. The composition of morphisms in $A$ is the multiplication of elements of $A$. The group structure of $\Hom_A(*,*)$ is that of the underlying additive group $A$. 
	Given unital rings $R$, $S$, any homomorphism $R\to S$ of rings which possibly does not preserve the unit is a semifunctor which sends the single object of $R$ into the single object of $S$. For instance, consider a morphism of rings $g:R\to S$ and let $z\neq 1_S\in S$ be an idempotent such that $z\in S^{R}=\{u\in S\text{ }\vert\text{ } ug(r)=g(r)u \text{ for every } r\in R\}$. Then, any morphism $f:R\to S$ given by $f(r)=g(r)z$ is a semifunctor.
	
	As particular cases, we have the following.
	\begin{itemize}
		\item If $g=\id_R$, where $R$ is a ring with unit $1_R$, consider an idempotent element $z\neq 0_R, 1_R$ in the center $Z(R)=\{r'\in R\text{ }\vert\text{ } r'r=rr' \text{ for every } r\in R\}$ of $R$. Then, let $f:R\to R$ be the map given by $f(x)=xz$, for $x\in R$. 
		It is a non-unital endomorphism of rings,
		\begin{invisible}
			Indeed, $f(x+y)=z(x+y)=zx+zy=f(x)+f(y)$ and $f(xy)=zxy=zzxy=zxzy=f(x)f(y)$, but $f(1_R)=1_Rz=z\neq 1_R$.
		\end{invisible}	thus $f$ defines a semifunctor $f:R\to R$ which is not faithful as $f(z)=z=f(1_R)$, but $z\neq 1_R$; for any $x\in R$ we have that $f(1_R)xf(1_R)=zxz=xzz=xz=f(x)$, hence $f$ is semifull. Further, let $*$ be the single object of $R$ and let $\f^{f}_{*,*}:\Hom_R(*,*)\to\Hom_{R}(*,*)$, $\f^{f}_{*,*}(x)=f(x)=xz$, for $x\in R$, be the natural transformation associated with $f$. Consider the map $\p^{f}_{*,*}=\f^{f}_{*,*}:\Hom_{R}(*, *)\to \Hom_R(*,*)$, $\p^{f}_{*,*}(x)=xz$, for every $x$ in $R$. 
		For every $x\in R$, we have $(\f^{f}_{*,*}\circ\p^{f}_{*,*})(x)=\f^{f}_{*,*}(xz)=xzz=zxz=f(1_R)xf(1_R)$, hence $f$ is naturally semifull. As an instance, consider as ring the product $R\times S$ of unital rings $R,S$ and the idempotent element $z=(0_R,1_S)\in Z(R\times S)$. Then, $f:R\times S\to R\times S$, $(x,y)\mapsto (x,y)z=(0_R,y)$ is a semifunctor which is naturally semifull but not faithful. \begin{invisible} Indeed, $f$ is not faithful as for $x\neq x'$ we have that $f((x,y))=(0_R,y)=f((x',y))$. It is a semifunctor as $f((1_R,1_S))=(0_R,1_S)\neq (1_R,1_S)$ and it is semifull as for any $(x,y)\in R\times S$ we have that $f((1_R, 1_S))(x,y)f((1_R,1_S))=(0_R, 1_S)(x,y)(0_R,1_S)=(0_R,y)=f((x,y))$. Consider the map $\p^{f}_{*,*}=\f^{f}_{*,*}:\Hom_{R\times S}(*, *)\to \Hom_{R\times S}(*,*)$, $\p^{f}_{*,*}((x,y))=(x,y)z=(0_R,y)$, for every $(x,y)$ in $R\times S$. We have that $(\f^{f}_{*,*}\circ\p^{f}_{*,*})((x, y))=\f^{f}_{*,*}((x,y)z)=(x,y)zz=(x,y)z=(0_R,y)=(0_R,1_S)(x,y)(0_R,1_S)=f((1_R, 1_S))(x,y)f((1_R,1_S))$, for every $(x,y)$ in $R\times S$.\end{invisible}
		\item Let $R$ be a ring with unit $1_R$ and consider the ring $\mathrm{M}_n(R)$ of square matrices of order $n\in \mathbb{N}$ with coefficients in $R$. Consider the canonical inclusion $g:R\to \mathrm{M}_n(R)$, $r\mapsto r\mathrm{I}_n$, where $\mathrm{I}_n$ is the identity matrix in $\mathrm{M}_n(R)$. Let $f:R\to \mathrm{M}_n(R)$ be the map given, for any $m\in R$, by 
		\begin{equation*}
			m\mapsto m\mathrm{E}_{ii},
		\end{equation*}
		where $\mathrm{E}_{ii}=(\delta_{ia}\delta_{ib})_{ab}$ is the matrix unit in which the entry $ii$, with $i\in \{ 1, \dots,n\}$, is the unique nonzero entry.
		It is a homomorphism of rings which does not preserve the unit, since $f(1_R)=\mathrm{E}_{ii}\neq 1_{\mathrm{M}_n(R)}=\mathrm{I}_n$. It defines a semifully faithful semifunctor $f:R\to \mathrm{M}_n(R)$. Indeed, it is clear that $f$ is faithful, and it is semifull as, given a matrix $A=(a_{ij})$ in $\mathrm{M}_n(R)$, we have that
		\begin{equation*}
			f(1_R)A f(1_R)=\mathrm{E}_{ii}A\mathrm{E}_{ii}=a_{ii}\mathrm{E}_{ii}=f(a_{ii}).
		\end{equation*}
	\end{itemize}
\end{es}

\noindent \textbf{Acknowledgements.}
The author would like to thank her supervisor Alessandro Ardizzoni for giving her the opportunity to work on this topic and for providing many meaningful comments and helpful insights during the preparation of this paper. This paper was written while the author was member of the
``National Group for Algebraic and Geometric Structures and their
Applications'' (GNSAGA-INdAM) and she was partially supported by MUR
within the National Research Project PRIN 2017 ``\emph{Categories, Algebras: Ring-Theoretical and Homological Approaches} (CARTHA)''.


\begin{thebibliography}{99}

\bibitem{AB22} Ardizzoni A., Bottegoni L., \emph{Semiseparable functors}, preprint. \href{https://arxiv.org/abs/2202.12113}{arXiv:2202.12113}

\bibitem{AB22-II} Ardizzoni A., Bottegoni L., \emph{Semiseparable functors and conditions up to retracts}, preprint. \href{https://arxiv.org/abs/2306.07091}{arxiv:2306.07091}

\bibitem{AMCM06} Ardizzoni A., Caenepeel S., Menini C., Militaru G., \emph{Naturally full functors in nature}. Acta Math. Sin. (Engl. Ser.) \textbf{22} (2006), no. 1, 233--250.


\bibitem{Bor94} Borceux F., \emph{Handbook of categorical algebra. 1. Basic category theory.} Encyclopedia of Mathematics and its Applications, \textbf{50}. Cambridge University Press, Cambridge, 1994.


\bibitem{CMZ02} Caenepeel S., Militaru G., Zhu S., \emph{Frobenius and separable functors for generalized module categories and nonlinear equations}. Lecture Notes in Mathematics, \textbf{1787}. Springer-Verlag, Berlin, 2002.


\bibitem{EZ76} Elkins B., Zilber J. A., \emph{Categories of actions and Morita equivalence.} Rocky Mountain J. Math. \textbf{6} (1976), no. 2, 199--225.

\bibitem{FOPTST99} Freyd P. J., O'Hearn P. W., Power A. J., Takeyama M., Street R., Tennent R. D., \emph{Bireflectivity}. Mathematical foundations of programming semantics (New Orleans, LA, 1995). Theoret. Comput. Sci. \textbf{228} (1999), no. 1-2, 49--76.

\bibitem{FrSc90} Freyd P. J., Scedrov A., \emph{Categories, allegories.} North-Holland Mathematical Library, \textbf{39}. North-Holland Publishing Co., Amsterdam, 1990.

\bibitem{Ha85} Hayashi S., \emph{Adjunction of semifunctors: categorical
structures in nonextensional lambda calculus}. Theoret. Comput. Sci. \textbf{41} (1985), no. 1, 95--104.


\bibitem{Ho90} Hoofman R., \emph{A Note on Semi-Adjunctions}. Technical Report RUU-CS-90-41 (1990). Department of Information and Computing Sciences, Utrecht University. \href{http://www.cs.uu.nl/research/techreps/repo/CS-1990/1990-41.pdf}{RUU-CS-90-41}


\bibitem{Ho93} Hoofman R., \emph{The theory of semi-functors}. Math. Structures Comput. Sci. \textbf{3} (1993), no. 1, 93--128.

\bibitem{HoM95} Hoofman R., Moerdijk I., \emph{A remark on the theory of semi-functors.} Math. Structures Comput. Sci. \textbf{5} (1995), no. 1, 1--8.

\bibitem{Kar78} Karoubi M., \emph{K-theory. An introduction.} Grundlehren der Mathematischen Wissenschaften, Band \textbf{226}. Springer-Verlag, Berlin-New York, 1978.



\bibitem{Law73} Lawvere F. W., \emph{Metric spaces, generalized logic, and closed categories.} Rend. Sem. Mat. Fis. Milano \textbf{43} (1973), 135--166.

\bibitem{Mac98} Mac Lane S., \emph{Categories for the working mathematician}. Second edition. Graduate Texts in Mathematics, \textbf{5}. Springer-Verlag, New York, 1998.

\bibitem{MW13} Mesablishvili B., Wisbauer R., \emph{QF functors and (co)monads}. J. Algebra \textbf{376} (2013), 101--122.

\bibitem{NVV89} N\v{a}st\v{a}sescu C., Van den Bergh M., Van Oystaeyen F., \emph{Separable functors applied to graded rings}. J. Algebra \textbf{123} (1989), no. 2,  397--413.


\bibitem{Raf90} Rafael M.D., \emph{Separable functors revisited}. Comm. Algebra \textbf{18} (1990), no. 5, 1445--1459.


\end{thebibliography}
\end{document}